\documentclass{article}
\usepackage{amssymb,amsfonts,amsmath,amsthm}
\usepackage{caption}
\usepackage{subcaption}
\usepackage{authblk}
\usepackage{color}
\usepackage{esvect}
\usepackage{empheq} 
\usepackage{float}
\usepackage{graphicx}
\usepackage{afterpage}
\usepackage{mwe}
\usepackage{comment}
\usepackage{mathtools}
\usepackage{txfonts}
\usepackage{graphicx}

\numberwithin{equation}{section}
\allowdisplaybreaks
\newtheorem{thm}{Theorem}[section]
\newtheorem{defn}[thm]{Definition}
\newtheorem{cor}[thm]{Corollary}
\newtheorem{lemma}[thm]{Lemma}
\newcommand{\f}[2]{\displaystyle\frac{#1}{#2}}

\newcommand{\per}{\mathrm{Per}}
\newcommand{\sgn}{\textup{sgn}}

\newtheorem{rmrk}[thm]{Remark}
\newtheorem{exmp}[thm]{Example}
\newcommand{\htan}{h^{\tau}}
\newcommand{\hnor}{h^{\nu}}
\newcommand{\e}{\varepsilon}
\newcommand{\R}{\mathbb{R}}

\newtheorem{prop}[thm]{Proposition}

\newcommand{\abs}[1]{\left\vert{#1}\right\vert}

\newcommand{\ba}{\begin{array}}
\newcommand{\ea}{\end{array}}

\newcommand{\bthm}{\begin{thm}}
\newcommand{\ethm}{\end{thm}}
\newcommand{\bstp}{\begin{stp}}
\newcommand{\estp}{\end{stp}}
\newcommand{\blemma}{\begin{lemma}}
\newcommand{\elemma}{\end{lemma}}
\newcommand{\bprop}{\begin{prop}}
\newcommand{\eprop}{\end{prop}}
\newcommand{\bpf}{\begin{pf}}
\newcommand{\epf}{\end{pf}}
\newcommand{\bdefn}{\begin{defn}}

\newcommand{\edefn}{\end{defn}}
\newcommand{\brk}{\begin{rmrk}}
\newcommand{\erk}{\end{rmrk}}
\newcommand{\bcrl}{\begin{crl}}
\newcommand{\ecrl}{\end{crl}}
\newcommand{\beg}{\begin{exmp}}
\newcommand{\eeg}{\end{exmp}}
\newcommand{\norm}[1]{\left\|#1\right\|}

\newcommand{\beqn}{\begin{equation}}
\newcommand{\eeqn}{\end{equation}}

\newcommand{\supp}{\operatorname{supp}}

\renewcommand{\leq}{\leqslant}
\renewcommand{\geq}{\geqslant}

\newcommand{\beq}{\begin{equation}}
\newcommand{\eeq}{\end{equation}}
\newcommand{\bea}{\begin{eqnarray}}

\newcommand{\eea}{\end{eqnarray}}

\newcommand{\dive}{\mathrm{div}\,}
\newcommand{\hdiv}{H_{\dive}( \Omega;\R^2)}

\newcommand{\sh}{\mathcal{H}}
\newcommand{\divcon}{\stackrel{\wedge}{\rightharpoonup} }

\newcommand{\boxedeq}[2]{\begin{empheq}[box={\fboxsep=6pt\fbox}]{align}\label{#1}#2\end{empheq}}

\newcommand{\uone}{u_\e^{(1)}}
\newcommand{\utwo}{u_\e^{(2)}}

\title{A Model Problem for Nematic-Isotropic Transitions with Highly Disparate Elastic Constants}

\begin{document}
\renewcommand\Authfont{\small}
\renewcommand\Affilfont{\itshape\footnotesize}

\author[1]{Dmitry Golovaty\footnote{dmitry@uakron.edu}}
\author[3]{Michael Novack\footnote{mrnovack@indiana.edu}}
\author[3]{Peter Sternberg\footnote{sternber@indiana.edu}}
\author[4]{Raghavendra Venkatraman\footnote{rvenkatr@andrew.cmu.edu}}
\affil[1]{Department of Mathematics, University of Akron, Akron, OH 44325}
\affil[2,3]{Department of Mathematics, Indiana University, Bloomington, IN 47405}
\affil[4]{Department of Mathematical Sciences, Carnegie Mellon University, Pittsburgh, PA 15213}

\maketitle

\noindent {\bf Abstract:} Continuing the program initiated in \cite{GSV}, we analyze a model problem based on highly disparate elastic constants that we propose in order to understand corners and cusps that form on the boundary between the nematic and isotropic phases in a liquid crystal. For a bounded planar domain $\Omega$ we investigate the $\e \to 0$ asymptotics of the variational problem
\begin{align*}
    \inf \f{1}{2}\int_\Omega \left(  \frac{1}{\e} W(u)+\e |\nabla u|^2 + L_\e(\dive u)^2  \right) \,dx
\end{align*}
within various parameter regimes for $L_\e > 0.$ Here $u:\Omega\to\R^2$ and $W$ is a potential vanishing on the unit circle and at the origin. When $\e\ll L_\e\to 0$, we show that these functionals  $\Gamma-$converge to a constant multiple of the perimeter of the phase boundary and the divergence penalty is not felt. However, when $L_\e \equiv L  > 0$, we find that a tangency requirement along the phase boundary for competitors in  the conjectured $\Gamma$-limit becomes a mechanism for development of singularities. We establish criticality conditions for this limit and under a non-degeneracy assumption on the potential we prove compactness of energy bounded sequences in $L^2$. The role played by this tangency condition on the formation of interfacial singularities is investigated through several examples: each of these examples involves analytically rigorous reasoning motivated by numerical experiments. We argue that generically, ``wall'' singularities between $\mathbb{S}^1$-valued states of the kind analyzed in \cite{GSV} are expected near the defects along the phase boundary. 

\section{Introduction}
Our purpose in this article is to propose and then initiate an analysis of a family of models inspired by phase transitions in liquid crystals. We have in mind the islands of phase known as tactoids, whose singular phase boundaries separate a locally well-ordered state of nematic liquid crystals from a disordered isotropic state. Our models should be relevant more generally to other phase transition problems for which large disparity in the elastic constants is a salient feature. Our analysis is mainly rigorous, but also includes formal calculations as well as computational experiments. 

Many models, of course, exist for nematic liquid crystals, including the Oseen-Frank energy, based on the elastic deformations of an $\mathbb{S}^1$- or $\mathbb{S}^2$-valued director $n$, and the $Q$-tensor based Landau-de Gennes model, whose energy density consists of a bulk potential favoring either a uniaxial nematic state, an isotropic state, or both, depending on temperature. What distinguishes our effort here is the attempt to capture the often singular structure of nematic/isotropic phase boundaries using a model reminiscent of Landau-de Gennes. 

The modeling of phase transitions in thin liquid crystalline films has attracted the attention of materials scientists and physicists for some time, \cite{FTKLR,KSL,RB,vBOS}. In experiments, one observes thin liquid crystal samples separated into nematic and isotropic phases. The islands of phase, i.e. the ``tactoids," appear as planar regions, with boundaries consisting of two or more smooth curves. 
Depending on temperature and on the type of liquid crystals, these smooth boundary curves may meet each other at singular points, known as ``boojums,'' forming angles or perhaps even cusps. 

Regarding the significance of tactoids as an object of study, we quote from the recent computational study of tactoids \cite{DA}, ``Tactoid structures have been shown to act as sensors via chirality amplification and can be used to guide motile bacteria.  They are also valuable architectural elements of self assembly, for example providing nucleation sites for growth of the smectic phase." 

In modeling these regions, the typical approach found in the materials science literature is to use a director theory and to postulate a surface energy that depends on the angle the director $ n$ makes with the normal $\nu$ to the phase boundary. Calling the region occupied by the  phase of uniaxial nematic say $\Omega_N$, and  writing $ n=(\cos\theta,\sin\theta)$ and $ \nu=(\cos\phi,\sin\phi)$, this leads to minimization of a surface energy of the form
\beq
F_s( n):=\int_{\partial\Omega_N}\sigma(\theta-\phi)\,ds \label{RP}
\eeq
where a typical choice for the function $\sigma:\R\to\R$, based on symmetries (and simplicity), is given by
\[
\sigma(\theta-\phi)=c_1+c_2\cos2(\theta-\phi),
\]
a form referred to as a Rapini-Papoular type surface density, (see e.g. \cite{MN}, section 3.4). In some studies within the physics literature the phase domain $\Omega_N$ is taken as a given region having a simple geometry such as a disk and then the minimization, taken over director fields $ n:\Omega_N\to \mathbb{S}^1$, may involve  coupling the surface term above to an elastic term such as $\int_{\Omega_N}\abs{\nabla n}^2\,dx$, corresponding to the so-called `equal constants' form of elastic energy, see e.g. \cite{vBOS}. In other studies, the shape itself is an unknown, but then, due to the difficulty of the analysis, the director field is often `frozen,' that is, taken to be a constant so that there is no elastic energy contribution and one minimizes \eqref{RP} alone. Then the problem resembles somewhat the Wulff shape problem arising in the classical study of crystal morphology, see e.g. \cite{F,RB}. 

Rather than postulating a specific surface energy, here we seek a model based on an order parameter, $u:\Omega\to\R^2$
defined on a planar domain $\Omega$ in which the singularities of the phase boundary emerge as a result of large disparity between the values of the elastic constants.
We are not alone in taking this viewpoint; see for example,
\cite{DA}, where the authors write ``It is clear that significant shape deformation is only achieved with the introduction of elastic anisotropy."

In \cite{GSV}, our first endeavor in this direction, we propose a model problem coupling the Ginzburg-Landau potential to an elastic energy density with large elastic disparity, namely
\begin{equation}
\inf_{u\in H^1(\Omega;\R^2)}\f{1}{2} \int_\Omega \left(\f{1}{\e}(1-|u|^2)^2+ \e|\nabla u|^2  +L(\dive u)^2 \right) \,dx.\label{divBBH}
\end{equation}
The minimization is taken over competitors satisfying an $\mathbb{S}^1$-valued Dirichlet condition on $\partial\Omega$ so as to avoid a trivial minimizer.
Here one might view the positive constant $\e\ll 1$ as being comparable in size to the elastic constant $L_1$ in say a Landau-de Gennes elastic energy density while the positive constant $L$, independent of $\e$, is playing the role of $L_2$, the coefficient of squared divergence in more standard elastic energy densities.

This choice of potential clearly favors $\mathbb{S}^1$-valued states, 
which are a stand-in in our models for uniaxial nematic states. As such, the model \eqref{divBBH} precludes any phase transitions between $\mathbb{S}^1$-valued states and the isotropic state $u=0$, and corresponds to the situation where the temperature--and therefore the potential-- favor only the nematic state. Analysis of \eqref{divBBH} in the $\e\to 0$ limit involves a `wall energy' along a jump set $J_u$ penalizing jumps of any $\mathbb{S}^1$-valued competitor $u$, and bulk elastic energy favoring low divergence. The conjectured $\Gamma$-limit of \eqref{divBBH} is
\beq
\frac{L}{2} \int_{\Omega} (\dive u)^2 \,dx + \frac{1}{6}\int_{J_{u}\cap \Omega} |u_+ - u_-|^3 \,d \sh^1,\label{limdivBBH}
\eeq
where $u_+$ and $u_-$ are the one-sided traces of $u$ along $J_u$. The natural space for competitors for this limit should be some subset of $H_\dive(\Omega;\mathbb{S}^1)$, the Hilbert space of $L^2$ vector fields having $L^2$ divergence. In order to make sense of the jump set we make the additional assumption in \cite {GSV} that $u\in BV(\Omega;\mathbb{S}^1)$, though this is surely not optimal. As a simple consequence of the Divergence Theorem, it follows that allowable jumps for an $H_\dive$ vector field must satisfy continuity of the normal component
\beq
u_+\cdot\nu=u_-\cdot \nu\quad\mbox{along}\;J_u,\label{tangood}
\eeq
where $\nu$ denotes the normal to $J_u$. Hence the cubic jump cost is penalizing the jump in the tangential component only.

In the present paper, we allow for co-existence of both nematic and isotropic phases by replacing the Ginzburg-Landau potential in \eqref{divBBH} with a potential $W:\R^2\to [0,\infty)$ that still depends radially on $u$ but that instead vanishes on $\mathbb{S}^1\cup\{0\}$. This is reminiscent of the zero set of the Landau-de Gennes potential in the critical temperature regime within the thin film context, see e.g. \cite{BPP}.  A prototype for what we have in mind is a potential of the form $W(u)=W_{CSH}(u):=\abs{u}^2\big(\abs{u}^2-1\big)^2$, or what is known in other physical contexts as the Chern-Simons-Higgs potential, see e.g. \cite{Spirn-Kurzke}.

We thus arrive at two models based on this potential. In the first model, analyzed in Section 2, we examine the asymptotic limit in $\e$ of the energy
\[
F_\e(u):= \f{1}{2}\int_\Omega \left(\f{1}{\e}W(u)+ \e|\nabla u|^2  +L_\e(\dive u)^2 \right) \,dx,
 \]
 where we assume $\e \ll L_\e\to 0$. Our main result for this model is Theorem \ref{Leps}, which states that in the $L^1$-topology, this sequence
of energies $\Gamma$-converges to a perimeter functional, measuring the arclength of the phase boundary between the $\mathbb{S}^1$-valued phase and the zero phase. In short, despite the much stronger penalty on divergence--think of
say $L_\e=\frac{1}{\abs{\log\e}}$--this amount of `elastic disparity' is too weak to be felt in the limit. In particular, minimizers of the limit, even under a boundary condition or area constraint to induce co-existence of $\mathbb{S}^1$-valued and $0$ phases, will have smooth phase boundaries. We mention that in \cite{Spirn-Kurzke}, the authors study the $\Gamma$-convergence of $\frac{1}{\e}F_\e$ for $L_\e=0$. In that scaling, vortices rather than perimeter contributes at leading order.

Our second model, and the main focus of our paper, involves the same type of potential $W$ as in $F_\e$, but now we   `ramp up' the cost of divergence still further, leading us to the energy
\beq
E_\e(u) := \f{1}{2} \int_\Omega \left(  \f{1}{\e}W(u)+\e|\nabla u|^2 +L(\dive u)^2 \right) \,dx,\label{introLenergy}
\eeq
where $L$ is a positive constant {\it independent of} $\e$. As $\e\to 0$ in this model, the jump set $J_u$ features two distinct types of discontinuities: as in \eqref{limdivBBH}, there are what we will call `walls' involving a jump discontinuity between two $\mathbb{S}^1$-valued states that respect \eqref{tangood}, and there are what we will call `interfaces' involving a jump between an $\mathbb{S}^1$-valued state and the isotropic $0$ phase.

We mention that one can consider minimization of $E_\e$ subject to a Dirichlet condition $g:\partial\Omega\to\R^2$, or a constraint such as $\int_{\Omega}\abs{u}^2=const$, or both in order to induce the co-existence of phases. The weak $H_\dive$ convergence of energy bounded sequences, however, implies that the appropriate condition for the limiting functional $E_0$ is that it inherits only the condition 
\beq
u\cdot n_{\partial\Omega}=g\cdot n_{\partial\Omega}\quad\mbox{along}\;\partial\Omega\label{diri},
\eeq
or simply $\mbox{meas}\big(\{u=0\}\big)=const$ in the case of the constraint. 

In any event,  it is the interfaces that represent the nematic/isotropic phase boundary and in light of the requirement \eqref{tangood}, one sees that whatever form the $\Gamma$-limit takes, the competitors, being in $H_\dive$, must have $\mathbb{S}^1$-valued traces {\it that are tangent to the phase boundaries}. As we will demonstrate through examples and numerics in Section 4, it is this tangency requirement that may induce singularities in the phase boundary. On this point, we mention that in this article we chose to penalize divergence more than other elastic energy terms, but had we replaced the term $L\int (\dive u)^2$ in \eqref{introLenergy} by $L\int (\dive R_\theta u)^2$ where $R_\theta$ is any rotation matrix, we would arrive at a limiting requirement on the nematic/isotropic interface in which tangency is replaced by $u$ making some non-zero angle with the tangent to the phase boundary.  In particular, for $\theta=\pi/2$ one penalizes the curl rather than the divergence and the resulting interface requirement is that the trace is orthogonal to the boundary.

In Fig. \ref{compareoleg}, we present an example of experimental nematic/isotropic configuration obtained in the laboratory of Oleg Lavrentovich along with a figure showing a numerically generated phase boundary based on gradient flow for $E_\e$. Both figures represent transient states but we point out the similar nature of the singular phase boundaries. Note that in the experimental picture, the phase boundary is singular only for the isotropic island whose surrounding nematic phase has degree $0$ on the boundary of the isotropic tactoid, not for the island where the degree is $1$. This distinction will come up frequently in our analysis.

\begin{figure}
\begin{center}
\includegraphics[scale=.3]{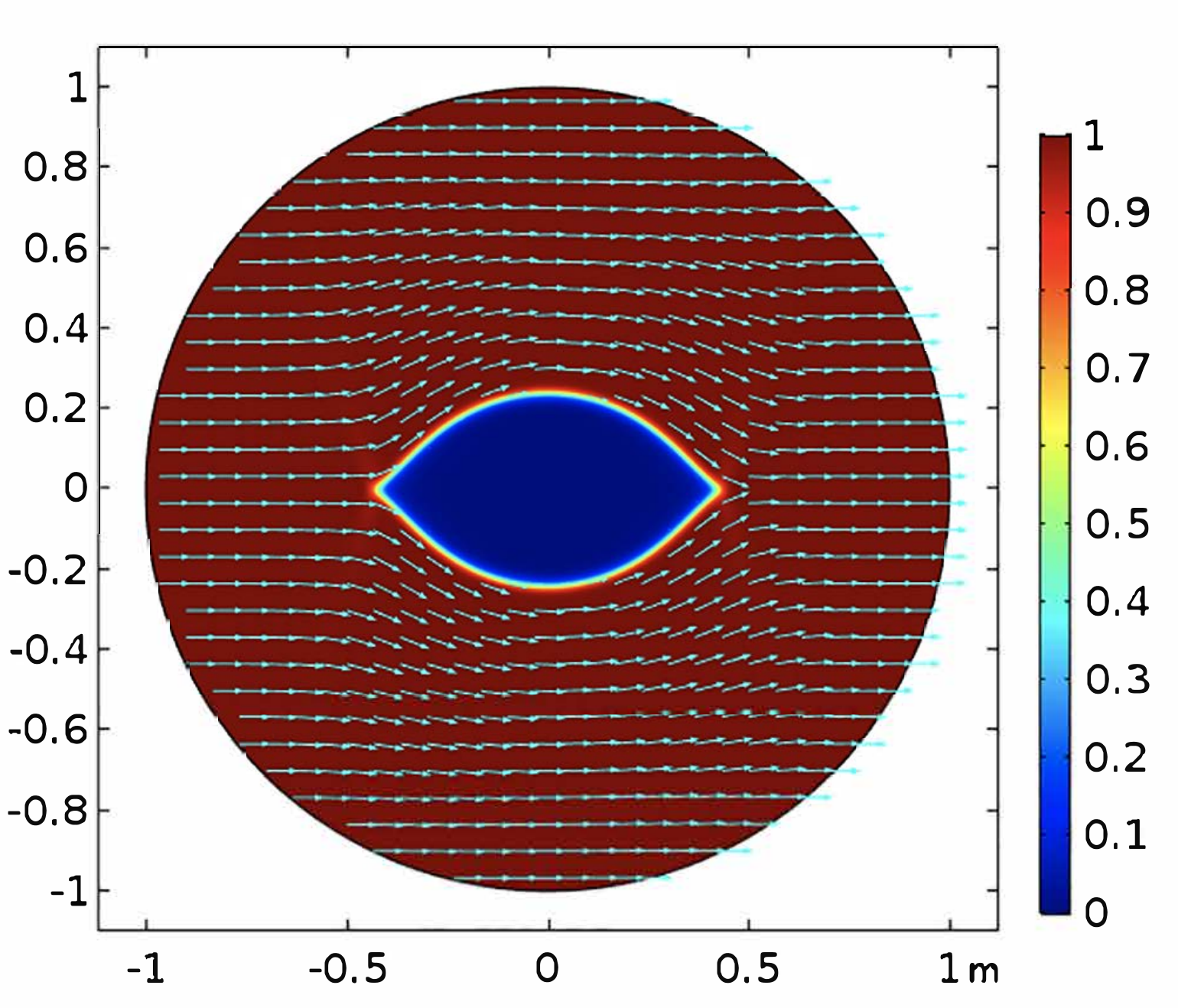}\qquad
\includegraphics[scale=.53]{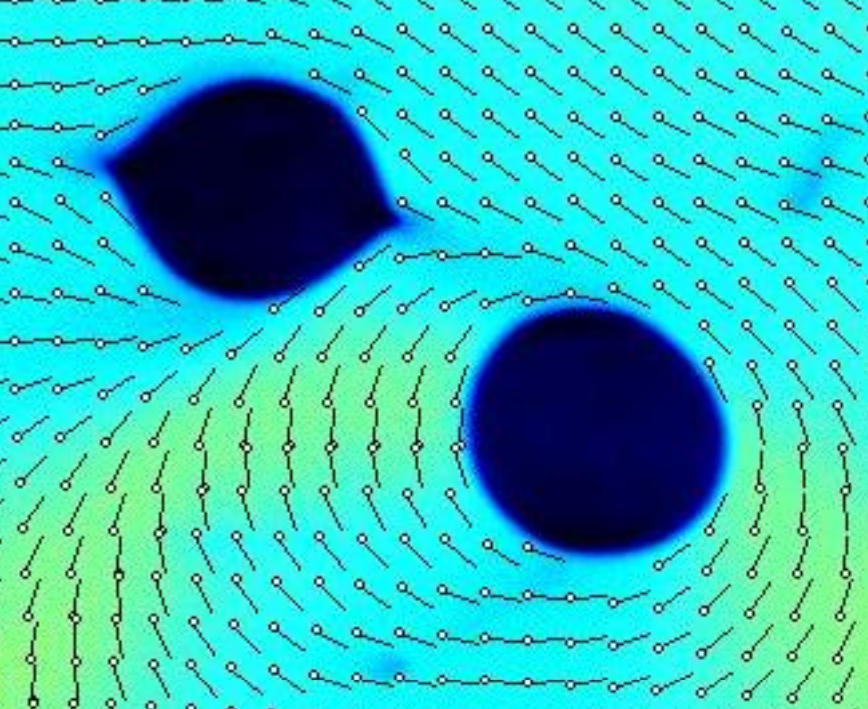}
\caption{Tactoids observed in simulations (left) and the experiments (right). The figure on the right is courtesy of O.~D.~Lavrentovich.}\label{compareoleg}
\end{center}
\end{figure}
Regarding a rigorous identification of the $\Gamma$-limit of $E_\e$, we only have partial results at this point. We present rigorous compactness results in $L^2(\Omega;\R^2)$ in Theorem \ref{fullcompact} based on an adaptation of \cite{DKMO}, but roughly put, it is easy to verify that any limit $u$ of an energy bounded sequence, i.e. $\{u_\e\}$ such that
$E_\e(u_\e)<C$, is a vector field $u\in H_\dive(\Omega;\mathbb{S}^1\cup\{0\})$ such that the isotropic phase $\{x:u(x)=0\}$ is a set of finite perimeter. Then making the extra assumption that $u$ is of bounded variation in the nematic phase where $u(x)\in \mathbb{S}^1$, one can invoke a combination of known techniques \cite{GSV,pollower} to establish a lower bound on the limit of the form 
\beq
E_0(u)=\f{L}{2} \int_\Omega (\dive u)^2 \,d x+c_0\, \mathrm{Per}_\Omega \big(\{|u| = 0\}\big) + \int_{J_u \cap \{ |u| = 1\}} K(u \cdot \nu) \,d \sh^1,\label{Ezerointro}
\eeq
where $c_0$ is the standard Modica-Mortola cost of an interface, cf. \eqref{MMcost}, and $K:\R\to [0,\infty)$ is a wall cost, arising through an abstractly defined solution to a certain cell problem. We wish to emphasize that, unlike for example \eqref{RP}, the limiting  problem that arises involves both interfacial energy terms and a bulk term.

We strongly suspect that this wall cost $K$ is in fact the cost associated with the heteroclinic connection between the states $(-u\cdot\tau,u\cdot \nu)$
and $(u\cdot\tau,u\cdot \nu)$ where $\tau$ is the approximate tangent to the jump set, see \eqref{1d} and \eqref{CSHwallcost}. The upper bound based on a recovery sequence for such a ``one-dimensional" wall where only the tangential component varies across the boundary layer is the content of Theorem \ref{Lupperbound}. 

The optimality of one-dimensional walls is a delicate point that turns out to hold in the analysis of \eqref{divBBH}-\eqref{limdivBBH}, cf. \cite{GSV}, as well as in the analysis of the divergence-free, or equivalently $L=\infty$, versions of these problems known as the Aviles-Giga problem, see e.g. \cite{ADM,AG,CD,DKMO,Ignat, Lorent,Lorent2,polupper1,pollower}. However, for Aviles-Giga and in \cite{GSV}, the matching of lower bound to upper bound is achieved through the somewhat miraculous Jin-Kohn entropy, cf. \cite{JinKohn} and \eqref{jinkohn}. The divergence of this vector field on the one hand bounds the Aviles-Giga energy from below but at the same time yields a value for the cost of a wall that coincides with the one-dimensional upper bound construction described above. As far as we can tell, there is no analogous entropy that works similarly for \eqref{introLenergy}.

In Section \ref{oned}, in contrast to the partial results from Section \ref{Conjecture}, we establish a complete $\Gamma$-convergence analysis along with optimal compactness, in the case where $\Omega$ is an interval. 

In Section \ref{critsec}, we turn to the derivation of criticality conditions for the proposed $\Gamma$-limit, $E_0$. As in \cite{GSV}, we find that in the $\mathbb{S}^1$-valued phase, away from walls, we can phrase criticality in terms of a system of conservation laws sharing characteristics, cf. Corollary \ref{conservation}. Characteristics turn out to be circular arcs along which divergence is constant with the curvature of the arc being given by the value of the divergence. We also explore criticality conditions for the wall and interface in Theorems \ref{critthm1} and \ref{critthm2}, as well as for possible junctions between walls and interfaces in Theorem \ref{junction}, whose somewhat technical proof we delay until the appendix.

Section \ref{examp} is crucial to our paper in that we explore the possible morphology of vortices, interfaces and walls through a series of examples. We focus on constructing critical points to the formal $L\to\infty$ limit of $E_0$ which one might describe as the Aviles-Giga $\Gamma$-limit augmented by isotropic regions, see \eqref{AVlimit}. These constructions are in particular divergence-free competitors for $E_0$ for $L$ finite that should be close to optimal for $L$ large. One might expect that when no area constraint on the size of the isotropic phase is imposed and  $\mathbb{S}^1$-valued Dirichlet data $g$ is specified in \eqref{diri} for $E_0$,  then only critical points that are nematic--i.e. $\mathbb{S}^1$-valued-- would emerge, with perhaps a certain number of defects in order to accommodate the degree of $g$, as in \cite{BBH}. However, in Example \ref{astroid}, we take $\Omega$ to be the unit disk and $g$ to have negative degree, and we show that, somewhat surprisingly, an $O(1)$ isotropic region opens up. We provide a possible explanation for this phenomenon in Theorem \ref{thmdiv}. 

In Section \ref{eyeballs}, we construct a divergence-free example in all of $\R^2$ in which a singular phase boundary encloses an isotropic island and in which the infinite nematic complement of this island obeys a trivial degree zero condition at infinity, i.e. $u\to \vec{e}_1$ as $\abs{x}\to\infty$. Unlike in the first example, this island is induced through an area constraint. This somewhat delicate calculation involves construction of both interfaces and walls with proper junction conditions holding at their intersection. 

In this section we also comment on the following crucial feature of the model observed in several of our examples. At defects on the phase boundary, the director $u$ often switches the sense of tangency. If a defect is a corner in the interior of the domain and a change in tangency occurs, then walls necessarily emanate from the defect in order to avoid infinite energy from the bulk divergence term; see Fig. \ref{nowalls} and the discussions at the end of Section \ref{divraghav} and preceding Example \ref{isotropiceyeball} .

Needless to say, this article represents just the initial investigation of a problem which holds within it a rich array of phenomena yet to be understood and questions to be pursued. We also mention that upgrading this model to the setting of $Q$-tensors should not pose significant obstacles. \\

\section{First try: A model whose elastic disparity is weak}\label{firsttry}

In this section, we begin our examination of the effect of disparity in elastic energy. Throughout this section, we will consider a continuous potential $W:\R^2\to [0,\infty)$ which vanishes on
$\mathbb{S}^1\cup\{0\}$. We assume that for some continuous function $V:\R\to [0,\infty)$, one has $W(u)=V(\abs{u})$ with then $V(0)=V(1)=0$ and $V>0$ elsewhere. The prototype for what we have in mind is the Chern-Simons-Higgs potential 
\beq
W_{CSH}(u):=\abs{u}^2(\abs{u}^2-1)^2.\label{CSH}
\eeq
Then for a sequence of positive numbers $L_\e \downarrow 0,$ we consider the sequence of functionals 
\begin{align}
\label{eq:fep}
F_\e(u) := \left\{ 
\begin{array}{cc}
\f{1}{2}\displaystyle\int_\Omega\left(\displaystyle\frac{1}{\e}W(u) + \e|\nabla u|^2 +L_\e(\dive u)^2 \right) \,dx &\mbox{if}\; u\in H^1(\Omega;\R^2), \\ \\
+\infty & \mbox{ otherwise}.
\end{array}
\right.
\end{align}

Though the $\Gamma$-convergence result below holds for any sequence $\{L_\e\}$ approaching zero, we are especially interested in the situation where
\[
\frac{L_\e}{\e}\to\infty\;\mbox{as}\;\e\to 0,
\]
so that the divergence term in the elastic energy is heavily emphasized. Our goal is to explore whether or not this disparity can produce a $\Gamma$-limit whose minimizers possess the types of phase boundary singularities reminiscent of isotropic-nematic interfaces as described in the introduction. What we shall find is that this level of elastic disparity is in fact {\it not} sufficiently strong to achieve this goal.

To this end, we define our candidate for the $\Gamma$-limit:
\begin{align*}
F_0(u) := \left\{
\begin{array}{cc} 
c_0 \per_\Omega(\{|u|=0\}) &\mbox{if}\; |u| \in BV(\Omega;\{0,1\}), \\ \\
+\infty &\mbox{ otherwise}.
\end{array}
\right.
\end{align*}
Here, 
\beq
c_0 := \int_0^1 \sqrt{V(s)}\,ds.\label{MMcost}
\eeq
The reader may well recognize this $\Gamma$-limit as precisely the well-known limit of the Modica-Mortola energies, an indication that to leading order in the energy, the divergence term has no effect on the asymptotic behavior of minimizers.
 
Our main result for this section is:
\bthm\label{Leps}
The sequence $\{ F_\e \}$ $\Gamma$-converges to $F_0$ in the topology induced by the $L^1$ norm of the modulus $| \cdot |$. That is,
\begin{enumerate}
\item for any $u\in L^1(\Omega;\R^2)$ and for any sequence $\{u_\e\}$ in $L^1(\Omega;\R^2)$, 
\begin{equation}\label{lbnd}
	|u_\e| \rightarrow |u| \text{ in } L^1(\Omega) \text{ implies } \liminf_{\e \rightarrow \infty} F_\e(u_\e) \geq F_0(u),
\end{equation} 
and
\item for each $u\in L^1(\Omega,\R^2)$ there exists a recovery sequence $\{w_\e\}$ in $L^1(\Omega,\R^2)$ satisfying 
\begin{equation}\label{upper}
|w_\e| \rightarrow |u| \text{ in } L^1(\Omega;\R^2)\quad\textit{ and}\quad
\limsup_{\e\rightarrow\infty} F_\e(w_\e)\leq F_0(u).
\end{equation} 
In fact, we can construct the sequence $\{w_\e\}$ so that $w_\e \to u$ in $L^1$.
\end{enumerate}
\ethm

\begin{rmrk} Regarding the asymptotic behavior of global minimizers, this result does not seem to address the possibility of a phase transition since there is no `incentive' for a minimizer of $F_\e$ to take on both $0$ and $\mathbb{S}^1$ values. To encourage a phase transition for a minimizer, one could, for example, impose a mass constraint such as
\begin{align*}
\int_\Omega |u_\e|^2 \,dx = m\quad\mbox{or}\quad \int_\Omega |u_\e| \,dx = m\quad\mbox{where}\; m\in (0,|\Omega|)\quad\mbox{with}\;\abs{\Omega}=\;\mbox{Lebesgue measure of}\;\Omega.
\end{align*}
Alternatively, one could impose a Dirichlet condition on $\partial\Omega$ such as $u_\e=g_\e$ where $g_\e$ is $\mathbb{S}^1$-valued on one portion of the boundary and then transitions smoothly down to $0$ on the rest of the boundary. Either of these alterations in the problem can be easily accommodated using what are by now standard techniques in $\Gamma$-convergence, see e.g.
\cite{ModicaARMA,OwenRubinsteinSternberg,Sternberg86}.
However, in order to present the main ideas without excessive technicalities, we formulate and prove a $\Gamma$-convergence theorem without either of these conditions, and merely remark that they could be incorporated if desired. 
\end{rmrk}

Though as indicated below \eqref{upper}, we can in fact establish $\Gamma$-convergence in the stronger topology $L^1(\Omega)$, it is not possible to obtain $L^1$-compactness for an arbitrary energy bounded sequence due to the degeneracy of the well $\mathbb{S}^1$. However, $L^1$-compactness of $\{|u_\e|\}$ follows by a standard argument, cf. e.g. \cite[Proposition 3]{Sternberg86}.
\begin{prop}
Let $\{ u_\e \}$ be a sequence of maps from $\Omega$ to $\R^2$ and assume that the sequence of energies $F_\e(u_\e)$ is uniformly bounded. Then there exists a subsequence $\{ u_{\e_j} \}$ and $u \in L^1(\Omega; \mathbb{S}^1\cup \{0\})$ such that $|u_{\e_j}| \to |u|$ in $L^1(\Omega)$.
\end{prop}
As observed in \cite{novack1}, this rather weak form of compactness is nonetheless sufficient to imply the existence of local minimizers of $F_\e$ given a local minimizer of $F_0$ which is isolated in this weaker topology, by modifying an argument of \cite{kohnsternberg}. For example, on a ``dumbbell"-type domain, there always exist local minimizers of $F_\e$ for $\e$ sufficiently small, cf. \cite[Theorems 4.2, 5.1]{novack1}.
\begin{proof}[Proof of the lower semi-continuity condition \eqref{lbnd}] Lower semi-continuity follows as in the Modica-Mortola setting since one simply ignores the divergence term. Since the argument is short, however, we present it here. The cases in which $\liminf_{\e \to 0}F_\e(u_\e)=\infty$ or $W(u)\neq 0$ on a set of positive measure are trivial. We therefore assume that $\liminf_{\e \to 0} F_\e(u_\e) = C < \infty$, and suppose that $|u_\e| \to |u|$ in $L^1(\Omega).$ Suppose also for now that $|u_\e| \leqslant 1,$ an assertion we will justify later by means of a truncation procedure.  In the argument below, we will make use of the function $\Phi(t):= \int_0^t \sqrt{V(s)} \,ds.$ As $L_\e \geqslant 0,$ we have
\begin{align*}
F_\e(u_\e) = \f{1}{2}\int_\Omega \left(\frac{1}{\e}W(u_\e)+ \e |\nabla u_\e|^2  + L_\e(\dive u_\e)^2\right) \,dx &\geqslant \int_\Omega \sqrt{V(|u_\e|)}\big|\nabla |u_\e|\big| \,dx \\
&\geqslant  \int_\Omega |\nabla \Phi(|u_\e|)|\,dx . 
\end{align*}
By the assumption that $\liminf F_\e(u_\e) = C < \infty,$ we obtain a uniform bound on $\{\Phi(|u_\e|)\}_{\e > 0}$ in $BV(\Omega)$, implying the existence of a subsequence converging in $L^1$ to $\Phi(|u|)$. Therefore, by lower semi-continuity in $BV$, 
\begin{align*}
\liminf_{\e \to 0} F_\e(u_\e) &\geqslant  \liminf_{\e \to 0} \int_\Omega |\nabla \Phi(|u_\e|) |\,dx \\
&\geqslant  \int_\Omega |\nabla \Phi(|u|)| \,dx \\
&= c_0 \per_\Omega(\{|u|=1\}). 
\end{align*}
This then completes the proof of \eqref{lbnd} under the assumption that $|u_\e| \leqslant 1.$ 

If it does not hold that $|u_\e|\leq 1$ then we define
\begin{align*}
u_\e^*(x) := \left\{
\begin{array}{cc}
u_\e(x) & \mbox{if}\;|u_\e(x)| \leqslant 1, \\
\frac{u_\e(x)}{|u_\e(x)|} & \mbox{if}\; |u_\e(x)| > 1. 
\end{array}
\right.
\end{align*}
We compute that 
\begin{align}\label{abc}
F_\e(u_\e) \geqslant \f{1}{2}\int_\Omega \left(\frac{1}{\e}W(u_\e)+\e|\nabla u_\e|^2 \right) \,dx \geqslant \f{1}{2}\int_\Omega \left( \frac{1}{\e} W(u_\e^*) +\e|\nabla u_\e^*|^2\right)\,dx.
\end{align}
Finally, we have that 
\begin{align*}
 \|\,|u_\e^*| - |u|\,\|_{L^1(\Omega)} \leqslant \|\,|u_\e| - |u|\,\|_{L^1(\Omega)} \to 0,
\end{align*}
so that we can combine the previous arguments with \eqref{abc} to obtain lower semi-continuity for the original sequence $\{u_\e\}$.
\end{proof}
\begin{proof}[Proof of the recovery sequence condition \eqref{upper}] Suppose we are given $u : \Omega \to \mathbb{S}^1 \cup \{0\}$ with $|u| \in BV(\Omega;\{0,1\}).$ We will construct a sequence $w_\e \subset H^1(\Omega;\R^2)$  with $w_\e \to u$ in $L^1(\Omega;\R^2)$ such that $\limsup_{\e \to 0}F_\e(w_\e) \leq F_0(u)$. We first briefly discuss the main idea, in order to motivate the construction that follows. Suppose that $u$ is smooth on the set, say $N$, where it is $\mathbb{S}^1$-valued, except for finitely many singular points $a_i$, and suppose  $u$ carries degree $d_i$ around each ``vortex" $a_i$. Suppose also that $\partial N$ is smooth.  We would like to define $w_\e$ using a boundary layer near $\partial N$ which bridges the values of $\left. u\right|_N$ near $\partial N$ to $0$ outside. In order to recover the correct $\Gamma$-limit with constant $2c_0$, we must define $w_\e$ on a neighborhoods $\mathcal{N}_\e$ of $N$ so that
\begin{equation}\notag\f{1}{2}
\int_{\mathcal{N}_\e}\left(\frac{1}{\e}W(w_\e)+\e|\nabla w_\e|^2 + L_\e (\dive w_\e)^2 \right)\,dx \to c_0\, \per_\Omega(\{|u|=1\}).
\end{equation}
As this is the least upper bound we could achieve even if $L_\e=0$, we must therefore construct $w_\e$ on $\mathcal{N}_\e$ so that 
\begin{equation}\notag
\int_{\mathcal{N}_\e} L_\e (\dive w_\e)^2 \,dx \to 0
\end{equation}
and so that the gradient squared and potential terms give the correct asymptotic limit. Since there is no assumption on how fast the sequence $\{ L_\e \}$ approaches zero, a natural construction to try is to define $w_\e$ on $\mathcal{N}_\e$ so that it is divergence-free there. This can be done by setting
\begin{equation}\label{tangent}
w_\e = f_\e (d(x))(\nabla^\perp d )(x),
\end{equation}
where $d(x)$ is the distance function to $\partial N$ and $f_\e$ is a suitably defined scalar function bridging the values $0$ and $1$.  Then 
\begin{equation}\label{divzero}
\dive w_\e = f_\e'(d)(\nabla d)\cdot \nabla^\perp d + f_\e(d)\dive(\nabla^\perp d)=0.
\end{equation}
It is easy to check that if $w_\e$ is a smooth, non-zero vector field tangent to level sets of $d$, as above, then its degree restricted to such a level set is $1$. If, however, $\sum_i d_i\neq 1$, then degree considerations imply that it is impossible to define smooth $w_\e$ which are non-zero and tangent to $\partial N$ but equal to $u$ in the interior of $N$ away from the boundary. In addition, even if $\sum d_i=1$, defining $w_\e$ inside $N$ by mollifying $u$ could yield vortices which result in unbounded energy as $\e \to 0$; see Theorem \ref{thmdiv}. 

To address these issues, it is instructive to consider the case in which $\Omega$ is the ball of radius $2$ centered at the origin, $N:=\{ |u|=1 \}$ is the unit disk with $u\equiv \vec{e}_1$ there and $u$ vanishes on the annulus $\{1<\abs{x}<2\}$. As explained above, there is no smooth field tangent to the boundary of the disk and equal to $u$ inside the disk. However, suppose we alter the boundary of the disk by adding two small cusps. Then we can define a continuous vector field tangent to the modified boundary, except at the cusps, which has degree zero. This tangent vector field allows for the construction of a boundary layer similar to \eqref{tangent} which contributes a perimeter term differing from $F_0(u)$ by a negligible amount, and a second, $\mathbb{S}^1$-valued boundary layer inside the disk which bridges the degree zero tangent field to the constant $\vec{e}_1$. The energetic contribution of this second layer vanishes in the limit.

Our general construction utilizes this basic idea. Given any component of the nematic region $N$, we first approximate $u$ there by  a map with degree zero around any closed curve lying in that component. This allows us to avoid the creation of vortices which are energetically too expensive for the divergence term.  Then, we add two cusps to the boundary components of the nematic regions; and finally, we use two boundary layers to bridge $0$ to the values in the nematic regions. We should emphasize that the approximations will be close to the original function $u$ in $L^1$ but of course will not be close in a stronger topology as that would violate basic properties of degree. 

We now fix any $u : \Omega \to \mathbb{S}^1\cup \{0\}$ such that $|u| \in BV(\Omega; \{0,1\})$ and begin our construction of the recovery sequence. We first approximate $u$ by vector fields $u_n$, then construct a recovery sequence for any $u_n$. A standard diagonal procedure will then imply the existence of a recovery sequence for $u$. We begin by showing that there exists an intermediate sequence of vector fields $\{v_n\}:\R^2 \to \mathbb{S}^1\cup\{0\}$ such that
\begin{enumerate}
\item $\{|v_n|=1\}=:\tilde{A}_n$ has $C^2$ boundary,
\item $v_n$ is smooth restricted to $\tilde{A}_n$,
\item for each $n$, there exists a non-empty arc $I_n \subset \mathbb{S}^1$ such that $v_n(x) \notin I_n$ for all $x\in \tilde{A}_n$,
\item $\mathcal{H}^1(\partial \tilde{A}_n \cap \partial \Omega)=0$ where $\mathcal{H}^1$ denotes one-dimensional Hausdorff measure,
\item $v_n \to u$ in $L^1$, and
\item $\per_\Omega(\tilde{A}_n\cap \Omega) \to \per_\Omega(\{|u|=1\})$.
\end{enumerate}
It is standard that there exist $\tilde{A}_n$ such that $(i)$, $(iv)$, and $(vi)$ hold and $\chi_{\tilde{A}_n} \to \chi_{\{|u|=1\}}$ in $L^1$, see e.g. \cite[Theorem 1.24]{giusti}. Next, we define a sequence $\tilde{v}_n$ by
\[ \tilde{v}_n= \begin{cases} 
      u(x) & \textrm{if } x\in \tilde{A}_n \cap \{|u|=1 \}, \\
      \vec{e}_1 & \textrm{if } x\in \tilde{A}_n \setminus \{|u|=1 \},\\
	0 & \textrm{if }x\notin \tilde{A}_n.
   \end{cases}
\]
The choice of $\vec{e}_1$ is arbitrary, since any unit vector would suffice. From the convergence of $\chi_{\tilde{A}_n}$ to $\chi_{\{|u|=1\}}$ and the dominated convergence theorem, it follows that, up to a subsequence, $\tilde{v}_n \to u$ in $L^1(\Omega;\R^2)$. The sequence $\{\tilde{v}_n\}$ satisfies properties $(i)$ and $(iv)$--$(vi)$, so it remains to argue we can modify it so that $(ii)$ and $(iii)$ hold as well. For each $n$, we define for $1\leq j \leq n$
\begin{equation}\notag
C_{j}^n:=\{x \in \Omega : \tilde{v}_n(x) \in (\cos([2\pi (j-1)/n,2\pi j/n)),\sin([2\pi (j-1)/n,2\pi j/n))) \}
,\end{equation}
and observe that for some $j_n$, $|C^n_{j_n}|\leq \abs{\Omega}/n$, since $\sum_j|C_{j}^n|\leq|\Omega|$. Then for $x \in C_{j_n}^n$, we redefine $\tilde{v}_n(x)$ to be identically $(\cos(2\pi (j_n-1)/n),\sin(2\pi (j_n-1)/n))$, so that the $\tilde{v}_n$ now avoids an arc $I_n\subset\mathbb{S}^1$ of length $2\pi/n$.
Now we can mollify $\tilde{v}_n$ to obtain smooth $v_n$ which also avoid $I_n$ and satisfy $(i)$--$(vi)$. Indeed, this can be done by choosing an interval $[a_n,b_n]$ in which to define the values of the phase of $v_n$ and mollifying the phase function itself. We also point out that inside $\tilde{A}_n$, the degree of $v_n$ around any simple, closed curve is zero, since $v_n$ cannot take values in $I_n$.\par
Next, for each $n$, we add small cusps to the sets $\tilde{A}_n$ and modify $v_n$ to obtain $u_n$. For each connected component of $\partial \tilde{A}_n$, we add two cusps pointing into the isotropic region, which change the perimeter of $\tilde{A}_n$ by at most $1/n$. We denote the resulting modification of $\tilde{A}_n$ by $A_n$, and smoothly alter the values of the function $v_n$, yielding $u_n$. This procedure can be carried out in such a fashion so that properties $(ii)$--$(vi)$ above still hold for the sequence $\{u_n\}$, and property $(i)$, the smoothness of $\partial A_n$, holds except at the cusps. This completes the construction of the sequence $\{u_n\}$.\par
For each $n$, we now construct a recovery sequence $\{ u_\e\}$, suppressing the dependence of $\{ u_\e \}$ on $n$ for ease of notation. Away from $\partial A_n$, $ u _\e$ will be identically equal to $u_n$. Near $\partial A_n$, we will use a boundary layer of the form $f_\e \vec{t}$, where $\vec{t}$ is a unit vector field tangent to level sets of the signed distance function $d$ to $A_n$ and where $f_\e$ solves a certain ODE. Away from the cusps, the level sets of the $d$ are smooth, which will be enough for our purposes. For each connected component of $\partial A_n$, we define $\vec{t}$ there by choosing a unit vector field tangent to that component and continuous on all of that component; see Fig. \ref{Fig. 1} below.\par
\afterpage{
\begin{figure}[h]
\centering
\includegraphics[scale=.8]{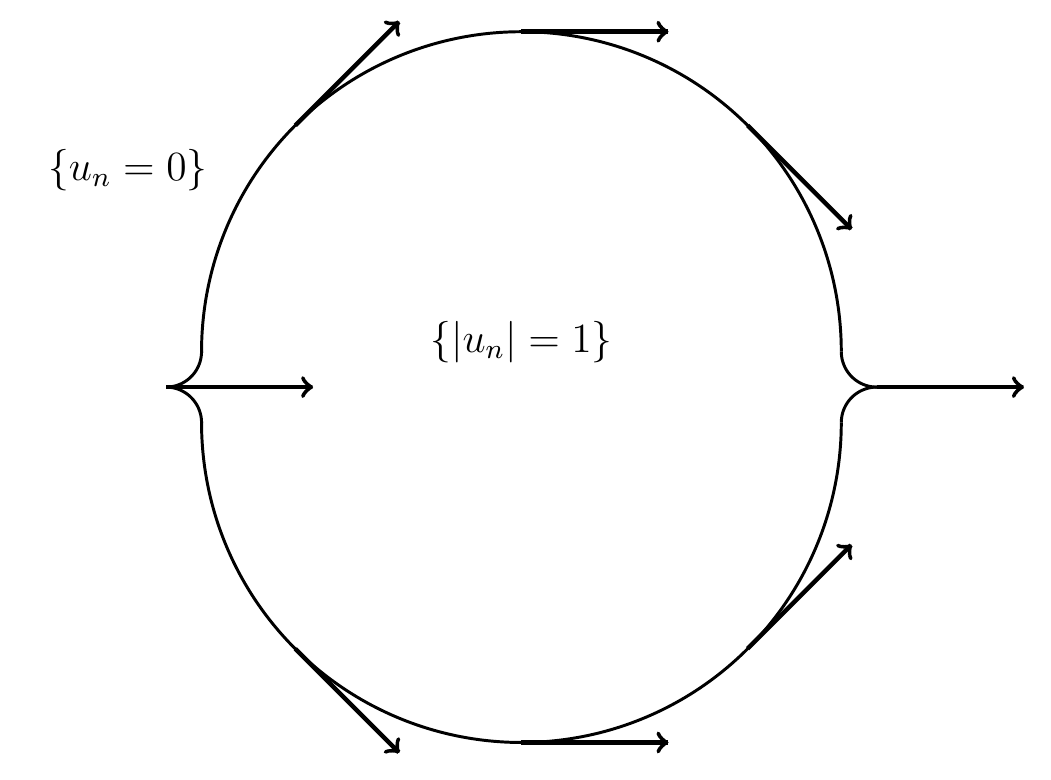}
\caption{The Lipschitz vector field $\vec{t}$ is tangent to this connected component of $\partial A_n$ and has degree zero around it.}
 \label{Fig. 1}
\end{figure}
}
The fact that each component contains two cusps implies that for the field $\vec{t}$ to be continuous, it must change the sense of tangency at every cusp. Thus on $\partial A_n$, $\vec{t}$ is always equal to $\pm \nabla^\perp d$. From these observations it follows that the degree of $\vec{t}$ around any connected component of $\partial A_n$ is zero. We then extend $\vec{t}$ to a continuous, unit vector field tangent to level sets of $d$ for $x$ such that $d(x)$ is small and positive and the nearest point projection $x$ onto $\partial A_n$ is not contained in any one of a union of rectangles near each cusp; see Fig. \ref{Fig. 2} below. To bridge the divergence free field $f_\e \vec{t}$ to the values of $u_n$ inside $A_n$, there is a second boundary layer, which is defined via an $\mathbb{S}^1$-valued homotopy between $\vec{t}$ and the values of $u_n$ inside $A_n$. This is only possible because $\vec{t}$ has degree zero around $\partial A_n$, as does $u_n$ around any simple, closed curve in $A_n$. The energy contribution from this layer in the limit will be zero, since $W(u_\e)=0$ there.\par
We now specify $u_\e$ in the first boundary layer, which contributes the perimeter term in the asymptotic limit. In the interior of $A_n$ and in $A_n^c$ at sufficient distances away from $\partial A_n$ to be specified shortly, we set $u_\e$ equal to $u_n$.

First, for some fixed $\delta >0$, we consider the following ODE, similar to \cite[Equation 3.2]{baldo}:
\begin{align*}
\left( \displaystyle\frac{\partial}{\partial s}f_\e(s)\right)^2 = \displaystyle\frac{\delta + V(f_\e(s))}{\e^2 (f_\e(s))^2}.
\end{align*}
As argued in \cite{baldo}, there exists a constant $C$, depending on $\delta$, such that for every $\e$, there exist positive numbers $C_\e$ and strictly decreasing solutions $f_\e:[0,C_\e]\to [0,1]$ of this ODE such that 
\begin{equation}\label{ceps}
C_\e \leq C\e
\end{equation}
and
\begin{equation}\notag
f_\e(0)=1 \textup{ and }f_\e(C_\e)=0
.\end{equation}Each $f_\e$ in fact depends on $\delta$, but we suppress this dependence.
Next, we excise a small rectangle at each cusp. Let 
\begin{equation}\label{meps}
m_\e:=\max\{\e,L_\e\}.
\end{equation}
For each cusp $c_i$, consider a rectangle $R_i^\e$ with one side of length $2C_\e$, centered at $c_i$, and perpendicular to the one-sided tangents at $c_i$, such that $R_i^\e$ protrudes $m_\e^{2/3}$ into the isotropic set $\{u=0\}$ in the other direction;  see Fig. \ref{Fig. 2} below.
\afterpage{
\begin{figure}[h]
\centering
\includegraphics[scale=.8]{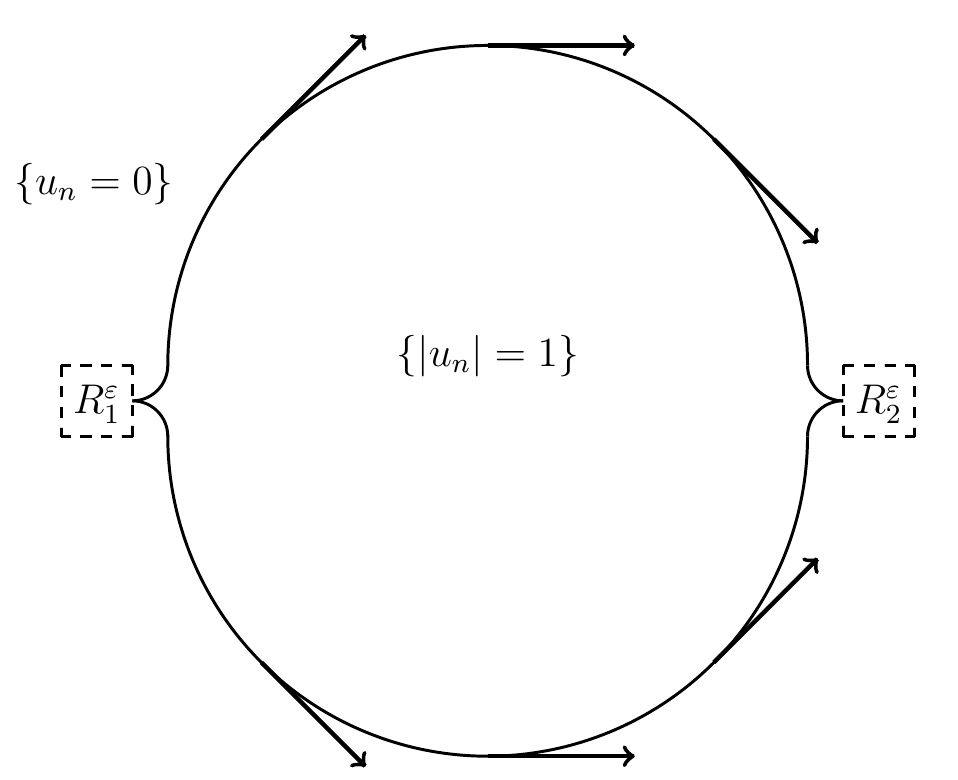}
\caption{Each rectangle $R_i^\e$ has length $m_\e^{2/3}$ and height $2C_\e$ and is perpendicular to the one-sided tangent vectors at the cusp $c_i$. For $x$ near the interface and not in $R_i^\e$, we can extend the tangent field $\vec{t}$ to be unit valued and tangent to level sets of the distance function to the interface.}
 \label{Fig. 2}
\end{figure}
} We denote by $\mathcal{R}_\e$ the union of all $R_i^\e$'s and then we define
\begin{align*}
u_\e(x) := \left\{
\begin{array}{cc}
u_n(x) & \textup{ if } C_\e \leq d(x)  \textup{ or } d(x)\leq -m_\e^{2/3}, \\
f_\e(d(x)) \vec{t}(x) & \textup{ if }0 < d(x) < C_\e \textup{ and }x\notin \mathcal{R}_\e. 
\end{array}
\right.
\end{align*}
In the definition above and in the remainder of the argument, we take $d$ to denote the signed distance function to $\partial A_n$ which is negative inside $A_n$.
We will deal with $u_\e$ on $\{ x: -m_\e^{2/3} < d(x) \leq 0\}$ and on $\mathcal{R}_\e$ at the end. It can be shown, by calculations similar to those preceding \cite[Equation 3.33]{novack1} that
\begin{equation}\label{perimeter}
\limsup_{\e \to 0}\f{1}{2}\int_{\{ 0 \leq d(x) \leq C_\e\}}\left(\f{1}{\e}W(u_\e)+  \e| \nabla u_\e|^2  + L_\e(\dive u_\e)^2 \right)\,dx \leq c_0\,\per_\Omega( A_n )+O(\sqrt{\delta}),
\end{equation}
observing in the process the crucial fact that the divergence of $u_\e$ on this set is zero, cf. \eqref{divzero}. Furthermore,
\begin{samepage}\begin{align}\label{00}\f{1}{2}\int_{\{ C_\e\leq d(x) \textup{ or } d(x) \leq -m_\e^{2/3}\}} &\left(\f{1}{\e}W(u_\e)+\e| \nabla u_\e|^2  +  L_\e(\dive u_\e)^2 \right)\,dx \\ \notag&=\f{1}{2}\int_{\{ C_\e\leq d(x) \textup{ or } d(x) \leq-m_\e^{2/3}\}} \left( \f{1}{\e}W(u_n)+ \e| \nabla u_n|^2  +L_\e(\dive u_n)^2 \right)\,dx\underset{\e \to 0}{\to}  0,\end{align}\end{samepage} since $W(u_n)=0$ and $|\nabla u_n|^2$ and $(\dive u_n)^2$ are bounded functions independent of $\e$ away from $\partial A_n$.\par
It remains to define $u_\e$ on the second boundary layer, $\{x: -m_\e^{2/3}< d(x) \leq 0 \}$, and on $\mathcal{R}_\e$. Let us first consider the second boundary layer. Because of the fact that $u_\e$ defined thus far has degree zero around $\partial A_n$ and $\{d(x) = - m_\e^{2/3} \}$, there exist Lipschitz phases $\psi_1:\partial A_n \to \R$, $\psi_2:\{d(x) = - m_\e^{2/3} \} \to\R$ such that $u_\e = (\cos(\psi_1),\sin(\psi_1))$ on $\partial A_n$ and $u_\e = (\cos(\psi_2),\sin(\psi_2))$ on $\{d(x) = - m_\e^{2/3} \}$. Then we can interpolate on the intermediate region using convex combinations of $\psi_1$ and $\psi_2$ so that $|\nabla u_\e|^2$ and $(\dive u_\e)^2$ are both $O(m_\e^{-4/3})$. Since $u_\e$ is a unit vector field here, $W(u_\e)$ is $0$. Hence we can calculate
\begin{samepage}
\begin{align}\notag
\f{1}{2}\int_{\{x: -m_\e^{2/3}< d(x) \leq 0 \}} \left(\f{1}{\e}W(u_\e) + \e |\nabla u_\e|^2 + L_\e(\dive u_\e)^2 \right)\,dx &\lesssim |\{ -m_\e^{2/3}< d(x) \leq 0\} |(\e m_\e^{-4/3}+L_\e m_\e^{-4/3}) \\ \label{twob}
& \leq O( m_\e^{1/3}).
\end{align}
\end{samepage}
So $u_\e$ on the second boundary layer contributes nothing to the asymptotic limit.\par
Finally, we treat $u_\e$ on the union $\mathcal{R}_\e$ of rectangles $R_i^\e$. It suffices to demonstrate the construction on a single $R_i^\e$ such that the cusp $c_i$ contained on one of its sides is the origin and the isotropic phase is to the right of the $x_2$-axis. Up to a translation, this is the situation depicted in Fig. \ref{Fig. 2} with $R_2^\e$. In these coordinates we may describe $R_i^\e$ as the rectangle $[0,m_\e^{2/3}]\times[-C_\e,C_\e]$. We set 
\begin{equation}\notag
u_\e(x_1,x_2) = f_\e(|x_2|)(1-m_\e^{-2/3}x_1)\vec{t}(0)
\end{equation}
on $R_i^\e$, which ensures compatibility with $u_\e$ as previously defined. We remark that $\vec{t}(0)$ is either plus or minus $\vec{e}_1$. Then on $R_i^\e$, $(\dive u_\e)^2 \sim O(m_\e^{-4/3})$, and $|\nabla u_\e|^2 \sim O(\e^{-2})$. Since the area of $R_i^\e$ is $2C_\e m_\e^{2/3}\leq2 C\e m_\e^{2/3}$ by \eqref{ceps} and \eqref{meps}, we have for small $\e$
\begin{equation}\label{o}
\f{1}{2}\int_{R_i^\e}\left( \f{1}{\e}W(u_\e)
+\e|\nabla u_\e|^2 + L_\e (\dive u_\e)^2\right)\,dx \leq O(m_\e^{2/3}).
\end{equation}
Combining \eqref{perimeter}--\eqref{o}, we obtain
\begin{align*}
F_\e(u_\e) \to F_0(u_n)+O\sqrt{\delta}. 
\end{align*}
In addition, the $u_\e$ converge in $L^1$ to $u_n$ by virtue of the dominated convergence theorem, since they are bounded and the set where they differ from $u_n$ has measure going to zero. Therefore, recalling that $\{u_\e\}$ depends on $\delta$ as well, we diagonalize over $\e$ and $\delta$ to obtain a recovery sequence for $u_n$. Since $u_n$ converge in $L^1$ to $u$ and $F_0(u_n) \to F_0(u)$, a second diagonalization argument over $n$ and $\e$ yields a recovery sequence for $u$.
\end{proof}

\section{ A model with large elastic disparity and singular phase boundaries}\label{secondtry}
In the previous section we saw that disparity in the elastic energy density of the form
\[
\e\abs{\nabla u}^2+L_\e(\dive u)^2\quad\mbox{with}\;\e\ll L_\e\to 0
\]
is insufficient to induce a singular phase boundary between the isotropic state $0$ and an $\mathbb{S}^1$-valued nematic state in minimizers of the $\Gamma$-limit. We now introduce a model with still larger disparity, and it is this model we will work with for the remainder of the article.

To this end, for a positive constant $L$ independent of $\e$ we define
\beq
E_\e(u) := \f{1}{2} \int_\Omega \left(  \f{1}{\e}W(u)+\e|\nabla u|^2 +L(\dive u)^2 \right) \,dx,\label{Lenergy}
\eeq
where $W(u)=V(|u|)$ for some continuous $V:[0,\infty)\to [0,\infty)$ that vanishes only at $0$ and $1$. As always, our prototype is the potential given by 
$W_{CSH}(u)=\abs{u}^2\left(\abs{u}^2-1\right)^2$, but in what follows
we can allow for more general potentials vanishing at $0$ and $1$, provided that for some constant $c>0$ one has the condition
\beq\label{hat}
H(t):=\min(t^2,|1- t^2|)\leq c\sqrt{V(t)} 
\eeq
for any $t \in [0,\infty)$.

In light of the divergence term in $E_\e$, it is clear that energy-bounded sequences $\{u_\e\}$ will have divergences that converge weakly in $L^2(\Omega).$  As we will discuss in Section 3.3, under the assumption \eqref{hat}, an adaptation of the compactness techniques of \cite{DKMO} allows us to also establish that a subsequence of $\{u_\e\}$ will converge strongly in $L^2(\Omega)$ to a limit taking values in $\mathbb{S}^1\cup\{0\}$. We will write $u_\e\divcon u$ when both $\dive u_\e\rightharpoonup \dive u$ weakly in $L^2(\Omega)$ and
$u_\e\to u$  strongly in $L^2(\Omega;\R^2)$. See Theorem \ref{fullcompact}.

These compactness results naturally lead us to consider the Hilbert space $H_{\dive}(\Omega;\R^2)$ of $L^2$ vector fields having $L^2$ divergence, and more specifically $H_{\dive}(\Omega;\mathbb{S}^1\cup \{0\})$, in light of the assumed zero set of the potential $W$. 

A vector field $u\in H_{\dive}(\Omega;\mathbb{S}^1\cup \{0\})$ that additionally lies in the space $BV(\Omega;\mathbb{S}^1\cup\{0\})$ is known to have a countably $1$-rectifiable jump set $J_u$ off of which $u$ is approximately continuous. In our pursuit of a possible candidate for the limit of the sequence $\{E_\e\}$ we will focus on functions lying in the intersection of these two spaces. 

Mappings in $BV$ have well-defined traces, say $u_+$ and $u_-$ on either side of $J_u$ and an easy application of the Divergence Theorem reveals that when $H_{\dive}$ vector fields have jump discontinuities across $J_u$ then necessarily the normal component is continuous, i.e.
\beq
u_+\cdot\nu=u_-\cdot \nu\quad\mathcal{H}^1\;a.e.\;\mbox{on}\;J_u\label{legaljumps}
\eeq
where $\nu$ is the (approximate) normal to $J_u$.

This brings us to a crucial distinction when attempting to identify a limiting energy for the sequence $\{E_\e\}$ --a mapping $u\in (H_{\dive}\cap BV)(\Omega;\mathbb{S}^1\cup \{0\})$ may undergo a jump between two $\mathbb{S}^1$-valued states, in which case \eqref{legaljumps} is supplemented by the additional 
requirement that
\beq
u_+\cdot\tau=-(u_-\cdot\tau)\quad\mbox{along}\;J_u,\label{tanjump}
\eeq
where $\tau$ is the approximate tangent to $J_u$. We will refer to any component of $J_u$ bridging two $\mathbb{S}^1$-valued states as a {\it wall}. On the other hand, $u$ may jump between an $\mathbb{S}^1$-valued state, say $u_+$, and $u_-=0$, in which case \eqref{legaljumps} implies that $u_+$ must coincide with $\pm\tau$. We will refer to any such component of $J_u$ as an {\it interface}. It is this tangency requirement along an interface that can induce singularities in the isotropic-nematic phase boundary.

\subsection{A conjecture for the $\Gamma$-limit of $E_\e$}\label{Conjecture}

Our goal in this section is to make the case for a proposed $\Gamma$-limit of the sequence $\{E_\e\}$ defined in \eqref{Lenergy}. While we do not at present have matching upper and lower asymptotic bounds for this sequence, we do have a construction leading to an asymptotic upper bound which we strongly suspect is sharp. We will begin with a description of this construction and then discuss various strategies for lower bounds, why the analogue of what works for the Ginzburg-Landau potential, cf. \cite{GSV}, apparently fails here and what the evidence is to support our conjecture on the sharpness of the upper bound.

We should say at the outset that our pursuit of the $\Gamma$-limit $E_0(u)$ begins with the assumption that $u\in (BV\cap H_{\dive})(\Omega;\mathbb{S}^1\cup\{0\})$. While this is not the natural space from the standpoint of compactness, the identification of the correct limiting space is non-trivial and we do not attempt to address it here. We refer the reader to \cite{ADM,DO,LO} for more discussion of this issue.  We make the $BV$ assumption here in order to speak sensibly about the $1$-rectifiability of the jump set $J_u$, though for that part of $J_u$ corresponding to interfaces, i.e. to $\partial\{\abs{u}=1\}$, as we will note below, this rectifiability comes easily from the fact that limits $u$ of energy-bounded sequences satisfy $\abs{u}\in BV(\Omega).$

As noted above, for such a vector-valued function $u$, the jump set naturally splits into two types: walls and interfaces, though these two types of singular curves may well meet in junctions, see e.g. Theorem \ref{junction} and Fig. \ref{Junctionfig}. An upper bound construction then rests on efficiently smoothing out these jump discontinuities, and in both cases, we rely on a one-dimensional type of resolution described formally below. The rigorous execution of these ideas follows the approach of \cite{CD} as adapted in \cite{GSV}.

To resolve an interface separating an isotropic region where $u=0$ from a nematic region where $u\in \mathbb{S}^1$ we invoke a by-now standard Modica-Mortola type of heteroclinic connection in the {\it modulus}. More precisely, after mollifying the interface to smoothen it if necessary, we mollify $u$ in the nematic region and make a boundary layer construction, say $\{w_\e\}$, of the form 
\beq
w_\e(x)=h\bigg(\frac{{\rm dist}\,(x,J_u)}{\e}\bigg)\,u(x)\label{ansatz}
\eeq
where ${\rm dist}\,(x,J_u)$ denotes the signed distance function to $J_u$ and where $h:\R\to\R$ minimizes the $1$d energy
\beq
\int_{-\infty}^{\infty}V(f)+\abs{f'}^2\,dt\quad\mbox{taken over}\;f\in H^1(\R)\;\mbox{such that}\;f(-\infty)=0,\;f(\infty)=1.\label{modcon}
\eeq
This leads to the same `interfacial cost' encountered in Section 2, namely
\begin{equation*}c_0=\int_0^1 \sqrt{V(s)}\,ds,\end{equation*}
multiplying the perimeter of the interface. Since
\[
\dive w_\e(x)=h\bigg(\frac{{\rm dist}\,(x,J_u)}{\e}\bigg)\,\dive u(x)+
\frac{1}{\e}h'\bigg(\frac{{\rm dist}\,(x,J_u)}{\e}\bigg)\nabla {\rm dist}\,(x,J_u)\cdot u(x),
\]
the term $L\int (\dive w_\e)^2$ in $E_\e(w_\e)$ will contribute nothing to such a boundary layer construction in the limit $\e\to 0$ since the first term is controlled by the fact that $u\in H_{\dive}$ and the second term is negligible due to the required tangency of $u$ and $J_u$ along an interface. We note that the ansatz \eqref{ansatz} would fail for the sequence of functionals $F_\e$ analyzed in Section \ref{firsttry} since there $u$ is not required to lie in $H_\dive$ and so
the term $\nabla {\rm dist}\,(x,J_u)\cdot u(x)$ will in general not vanish.

With appropriate care taken to treat issues of regularity, this can be made rigorous. What is more, this construction, based only on appropriate interpolation of the modulus between $0$ and $1$, gives a sharp upper bound on the interfacial energy, in light of the inequality
\beq
E_\e(u)\geq\left( \frac{1}{2}\int_{\Omega}\frac{1}{\e}V(\abs{u})+\e\abs{\nabla\abs{u}\,}^2\right)\,dx\quad\mbox{for any}\;u\in H^1(\Omega;\R^2).\label{MMineq}
\eeq
Since this is the classical scalar Modica-Mortola functional in terms of the function $\abs{u}$, when applied in a neighborhood of the interface it yields the matching lower bound of $c_0\,{\rm Per}_{\Omega}(\{\abs{u}=1\}).$ 

Our boundary layer construction in a neighborhood of a wall separating two $\mathbb{S}^1$-valued states, say $u_+$ and $u_-$, is one-dimensional in a different sense. In light of the continuity of the normal component of $u$ across a wall, cf. \eqref{legaljumps}, a natural  choice is to fix the value of $u\cdot\nu$ across the boundary layer and use a heteroclinic connection to bridge the value of $u_-\cdot\tau$ to $u_+\cdot\tau$, that is, to bridge
$-\sqrt{1-(u\cdot\nu)^2}$ to $\sqrt{1-(u\cdot\nu)^2}$ in light of \eqref{tanjump}. 

At a point on the wall, such a choice leads to a cost per unit length given by the minimum of a heteroclinic connection problem that is a bit different from \eqref{modcon}, namely
\begin{equation*}
\inf_{f}\int_{-\infty}^{\infty}W\big(f\tau+(u\cdot\nu)\nu\big)+
\abs{f'}^2\,dt
=\inf_{f}\int_{-\infty}^{\infty}V\bigg(\sqrt{f^2+(u\cdot\nu)^2}\bigg)+
\abs{f'}^2\,dt,
\end{equation*}
taken over $f\in H^1(\R)$ such that
\[
f(-\infty)=(u_-\cdot\tau)=-\sqrt{1-(u\cdot\nu)^2}\quad\mbox{and}\quad
f(\infty)=(u_+\cdot\tau)=\sqrt{1-(u\cdot\nu)^2}.
\]
One easily checks that this infimum is given by $K(u\cdot\nu)$ where we define
\begin{equation}
K(z):=\int_{-\sqrt{1-z^2}}^{\sqrt{1-z^2}}\sqrt{V}\left(\sqrt{z^2+y^2}\right)\,dy,
\label{1d}
\end{equation}
which in the prototypical case of $W_{CSH}(u):=\abs{u}^2(\abs{u}^2-1)^2$ takes the form
\begin{equation}
K(z)=\int_{-\sqrt{1-z^2}}^{\sqrt{1-z^2}}\sqrt{z^2+y^2}\left(1-z^2-y^2\right)\,dy.\label{CSHwallcost}
\end{equation}
We point out that $c_0=\frac{K(0)}{2}$ and also note that $K$ is not a monotone function of $z$ on $[0,1]$, but rather increases to a unique maximum and then decreases down to zero at $z=1$.

Such an upper bound construction leads us to our conjectured $\Gamma$-limit when 
$u\in (H_{\dive}\cap BV)(\Omega;\mathbb{S}^1 \cup \{0\})$, namely $E_0$ given by
\boxedeq{E0def}{E_0(u)=\f{L}{2} \int_\Omega (\dive u)^2 \,d x+ \frac{K(0)}{2} \mathrm{Per}_\Omega (\{|u| = 1\}) + \int_{J_u \cap \{ |u| = 1\}} K(u \cdot \nu) \,d \sh^1.}
 \par
One should also impose upon competitors in the minimization of $E_0$ a boundary condition of the form \eqref{diri} if one imposes the Dirichlet condition $u_{|_{\partial\Omega}}=g$ for $E_\e$ or an area constraint on the measure of the isotropic or nematic region within $\Omega$ if an integral constraint has been imposed on $E_\e$.

In particular, we can rigorously assert:

\bthm\label{Lupperbound}
For any $u \in  (H_{\mathrm{div}}\cap BV)(\Omega,\mathbb{S}^1 \cup \{0\})$,
there exists a sequence $\{w_\e\} \in H^1(\Omega;\R^2)$ with $w_\e \divcon u$ such that 
\begin{align} \label{ubndnew}
\limsup_{\e \to 0} E_\e(w_\e) = E_0(u).
\end{align}
\ethm
\vskip.1in
Furthermore, we state a conjecture:\\

\noindent\textbf{Conjecture :} Suppose $W$ satisfies \eqref{hat}. Then for any $u \in  ( H_{\mathrm{div}}\cap BV)(\Omega, \mathbb{S}^1 \cup \{0\})$ and any sequence $u_\e \divcon u$ we have 
\beq\liminf_{\e \to 0} E_\e(u_\e) \geqslant E_0(u). \label{conjecture}
\eeq

\begin{proof}
The proof of \eqref{ubndnew} is similar to the proof of \cite[Theorem 3.2(ii)]{GSV}, which itself is an adaptation of the techniques laid out in \cite{CD} for Aviles-Giga recovery sequences, so we omit the details. The only difference between the arugment here and the argument in \cite{GSV} is that, as discussed above, in addition to walls, there are also interfaces now in which $u$ jumps from a tangent $\mathbb{S}^1$-valued state to $0$. However, this does not pose a serious obstacle to the construction, as the important technical components are the rectifiability of the jump set $J_u$ and the condition \eqref{legaljumps} satisfied along $J_u$ at either a wall or interface, which goes to guarantee that the boundary layer constructions do not contribute asymptotically to the $L^2$-norm of the divergence.
\end{proof}
\brk
We have not addressed in \eqref{E0def} or in Theorem \ref{Lupperbound} the issue of boundary conditions, so we describe now how to incorporate them. Suppose one were to fix Dirichlet data $g_\e\in H^{1/2}(\partial \Omega; \R^2)$ for admissible functions in $E_\e$. The functions $g_\e$ could be $\mathbb{S}^1$-valued, or could transition smoothly between $\mathbb{S}^1$ and $\{0\}$ if we are trying to induce a phase transition. Let us assume that $g_\e \to g$ in $L^2(\partial \Omega;\R^2)$ for some $g:\partial \Omega \to \mathbb{S}^1\cap \{0\}$. We observe that for a sequence $\{u_\e \}\in H^1(\Omega;\R^2)$ satisfying $u\e=g\e$ on $\partial\Omega$ and so in particular $u_\e \cdot \nu_\Omega=g\cdot \nu_\Omega$, under the convergence $u_\e \divcon u$ with $u\in (BV \cap H_{\dive}(\Omega; \mathbb{S}^1\cap \{0\})$, it follows from the divergence theorem and the convergence of $g_\e$ to $g$ that
\begin{align*}
u\cdot \nu_\Omega = g \cdot \nu_\Omega.    
\end{align*}
In this case, the limiting energy $E_0$ would also contain integrals around the portion of $\partial \Omega$ where $u\cdot \tau_\Omega \neq g \cdot \tau_\Omega$, and the cost along these portions would either be given by $K(0)/2$ or $K(u\cdot \nu_\Omega)$.
\erk
\brk
An a priori sharper upper bound for the wall cost $K$ could be obtained for these energies using the techniques of \cite{polupper1}. There, the author obtains an upper bound without assuming that the optimal profile is one-dimensional. Instead, the cost is defined via a cell problem. As the class of admissible functions for the cell problem is strictly larger than the class of $1$d competitors, the cell problem yields what could in theory be a sharper upper bound. However, since we conjecture that the one-dimensional profile is optimal and since at present we see no way to analyze the abstract cell problem to make this comparison, we do not pursue the strategy from \cite{polupper1}.
\erk

Given the presence of arguments leading to matching lower bounds for one-dimensional constructions in the Aviles-Giga problem \cite{AG} and for the energy $E_0$ with the potential replaced by a Ginzburg-Landau potential $W_{GL}(v):=(1-|v|^2)^2$, in  \cite{GSV}, it behooves us to comment on why, at present, we have no such argument here. In \cite{AG} and in \cite{GSV}, the authors employ the celebrated Jin-Kohn entropy \cite{JinKohn}. Defining 
\begin{equation}
\Xi(v_1,v_2)=2 \left(\f{1}{3}v_2^3+v_2v_1^2-v_2,\f{1}{3}v_1^3+v_1v_2^2-v_1 \right),\label{jinkohn}
\end{equation}
the version of these entropies well-suited to the situation where the jump set is parallel to one of the coordinate axes, one can then calculate
\begin{align}\label{divxi}
\dive \Xi(v_1,v_2) = 2(|v|^2-1)(\partial_{x_1} v_2 + \partial_{x_2} v_1) + 4 v_1 v_2 \dive v.
\end{align}
In the divergence-free Aviles-Giga setting of \cite{JinKohn}, the last term drops out and an application of the inequality $a^2+b^2 \geq 2ab$ allows one to bound the Aviles-Giga energy density from below by $\dive \Xi(v_1,v_2)$. When the divergence is possibly non-zero, as in \cite{GSV}, a slight modification yields 
\begin{equation*}
\dive \Xi(v_1,v_2) \leq \left(\e|\nabla v|^2+ \f{1}{\e}(|v|^2-1)^2+L(\dive v)^2 \right)+ \textup{ error terms}
,\end{equation*}
which is the crux of the argument.\par
Unfortunately, for most radial potentials that are not the Ginzburg-Landau potential $W_{GL}$, this technique does not seem to work. First, we note that 
\begin{align*}
\Xi(v_1,v_2) &= \left(\int_{-v_2}^{v_2} (v_1^2+s^2-1) \,ds, \int_{-v_1}^{v_1} (s^2+v_2^2-1) \,ds\right) ,
\end{align*}
where the integrands are, up to signs, given by $\sqrt{W_{GL}}$. Therefore, to obtain a version of \eqref{divxi} with $W_{GL}$ replaced by $\sqrt{W}$, where $W$ is our potential vanishing on $\mathbb{S}^1\cup\{0\}$, the natural choice for the vector field to replace $\Xi$ would be
\begin{equation*}
\Xi_{W}(v_1,v_2)= \left(\int_{-v_2}^{v_2} \sqrt{W(v_1,s)} \,ds, \int_{-v_1}^{v_1} \sqrt{W(s,v_2)} \,ds\right).
\end{equation*}
When we calculate the divergence of $\Xi_W(v_1,v_2)$, we get
\begin{align*}
\dive &\Xi(v_1,v_2)\\ &= 2\sqrt{W(v)}(\partial_{x_1} v_2 + \partial_{x_2} v_1) +\partial_{x_1} v_1 \int_{-v_2}^{v_2}(\partial_{v_1}\sqrt{W(v_1,s)})\,ds +\partial_{x_2} v_2 \int_{-v_1}^{v_1}(\partial_{v_2}\sqrt{W(v_2,s)})\,ds .
\end{align*}
The only way for $\dive v$ to factor out of the last two terms is if 
\begin{equation*}
\int_{-v_2}^{v_2}\partial_{v_1}\sqrt{W(v_1,s)}\,ds = \int_{-v_1}^{v_1}\partial_{v_2}\sqrt{W(v_2,s)}\,ds,
\end{equation*}
which holds for radial $W$ when $\sqrt{W}$ is linear in $|v|^2$. This cannot hold for any $W$ that vanishes only at $\mathbb{S}^1$ and $\{0\}$.\par

We point out that a related problem that has resisted resolution for several decades is the determination of
a sharp lower bound for the sequence of energies
\beq
\int_{\Omega} \frac{1}{\e}(\abs{u}^2-1)^p+\e\abs{\nabla u}^2\quad\mbox{with}\;p<2\label{penergy}
\eeq
where competitors $u:\Omega\to\R^2$ must be divergence-free. Here too it is conjectured that the 
optimal lower bound for the wall cost is based on a one-dimensional ansatz,  \cite{ADM}, but a proof has not been found, and in particular, no version of the Jin-Kohn entropy is evident. An abstract lower bound involving a cell problem for functionals of this type has been derived in \cite{pollower}, but has not yet to our knowledge been matched by a corresponding upper bound. The strategy in this and other papers involving a lower bound phrased in terms of a cell problem is based on a blow up procedure introduced in \cite{FM}. Such a lower bound of the form $\int_{J_u}\tilde{K}(u\cdot\nu)\,d\mathcal{H}^1$ for some $\tilde{K}:[0,\infty)\to[0,\infty)$ defined as the solution to a cell problem could be derived for our wall energy as well, but we do not include the argument since it does not provide much insight here.

On the other hand, for $p>2$ in \eqref{penergy}, as shown in \cite{ADM}, the one-dimensional ansatz is {\it not} optimal, with an oscillatory construction, often referred to as `microstructure,' whose modulus hews close to $\mathbb{S}^1$, yielding a lower asymptotic energy.   

So what is the rationale behind our conjecture \eqref{conjecture}? One key point is that for $W$ given by $W_{CSH}$ or more generally by a potential satisfying \eqref{hat}, the level of degeneracy of the $\mathbb{S}^1$ potential well is no flatter than that of $W_{GL}$ where again it is known that walls follow a one-dimensional profile asymptotically. Thus, it seems unlikely that microstructure of the type emerging, for example, in \eqref{penergy} for $p>2$ would appear here since for our model it is no more beneficial energetically to abandon one-dimensionality in order to be nearer to $\mathbb{S}^1$ across a wall than it was in \eqref{divBBH}. 

Other evidence for our conjecture is numerical. Repeated numerical experiments in the form of gradient flow for $E_\e$ with $\e$ small in a variety of domains, for a variety of boundary conditions and for a wide range of $L$ values have not indicated any lack of one-dimensionality in the wall structure. Were the transition to be truly $2d$, one might expect the wall to exhibit some oscillation or other instability.
For example,
 in \cite{GSV} while we prove that for \eqref{divBBH}-\eqref{limdivBBH} the wall cost is based on a one-dimensional construction, we also find that when minimizing \eqref{divBBH} in a rectangle with $\mathbb{S}^1$-valued Dirichlet data given by $(\pm a,\sqrt{1-a^2})$ for
$a\in [0,1)$ on the top and bottom respectively and periodic boundary conditions on the sides, there exists a parameter regime in $L$ and in the box dimensions where the minimizer is not one-dimensional, cf. \cite{GSV}, Thm. 6.6. Indeed this theorem is supported by numerics revealing the eventual instability of a horizontal wall and the emergence of so-called `cross-ties' commonly arising in studies of micromagnetics such as \cite{AlougesRiviereSerfaty}. On the other hand, as we discuss in Section \ref{1dex}, numerically we detect no such instability of a horizontal wall for $E_\e$
under these boundary conditions. Then a numerical examination in Section \ref{diskex} of wall structure for a version of our problem posed in a disk also indicates a one-dimensional heteroclinic connection for the wall structure. This gives us further confidence in the conjectured one-dimensionality of the wall cost.

\subsection{Compactness}
In this section we establish a compactness result for energy-bounded sequences.
Recalling the assumption \eqref{hat}, we begin by observing that
\beq 
E_\e(u) \geq \f{1}{2}\int_\Omega  \left( \frac{1}{\e c^2} H^2(|u|)+\e|\nabla u|^2 + L(\dive u)^2  \right) \,d x.
\eeq 
Both the Ginzburg-Landau and the Chern-Simons-Higgs potentials satisfy this inequality and in \cite{GSV} it is shown that for $W$ given by the Ginzburg-Landau potential, the compactness result of \cite{DKMO} generalizes to $E_\e$. In this section, we show that this compactness approach generalizes to potentials also vanishing at the origin provided we assume \eqref{hat}. 
\bthm\label{fullcompact}
Let $\{u_\e\}_{\e > 0} \subset H^1(\Omega;\R^2)$ be a sequence such $E_\e(u_\e) \leqslant C$, with $C$ independent of $\e$.  Then there exists a subsequence (still denoted here by $u_\e$) and a function $u\in H_{\dive}(\Omega;\mathbb{S}^1 \cup \{0\})$ with $\abs{u}\in BV(\Omega;\{0,1\})$ such that
\begin{eqnarray}
&&u_\e\rightharpoonup u\quad\mbox{in}\;\hdiv, \label{weakdiv}\\
&& u_\e\to u\quad\mbox{in}\;L^2(\Omega;\R^2).\label{strongL2}
\end{eqnarray}
\ethm
The fact that for a subsequence of $\{u_\e\}$, one has $\abs{u_\e}\to \abs{u}$ in $L^1(\Omega)$ where $\abs{u}\in BV(\Omega;\{0,1\})$ follows
from inequality \eqref{MMineq} via the standard Modica-Mortola approach, cf. \cite{ModicaARMA} or \cite{Sternberg86}.
 The proof of \eqref{weakdiv} follows immediately from the uniform bound on the $L^2$ norm of the divergences, so we turn to the proof of \eqref{strongL2}. The proof follows closely the proof in \cite[Proposition 1.2]{DKMO}, with the details suitably modified to account for the fact that the potential may now possibly vanish at $0$ in addition to $\mathbb{S}^1$. Below we outline the procedure and indicate which portions require changes from \cite{DKMO}.\par
The proof relies on compensated compactness and a careful analysis of the Young measure $\{\mu_x\}_{x\in \Omega}$ generated by the sequence $\{u_\e\}$. One of the key tools in this analysis is the concept of an entropy, defined here as a mapping $\Phi\in C_0^{\infty}(\R^2;\R^2)$ such that 
\[
\Phi(0)=0,\quad D\Phi(0)=0\quad\mbox{and for all}\;z\in\R^2\;\mbox{one has}\;z\cdot D\Phi(z)\,z^\perp=0,
\]
where $z^\perp=(-z_2,z_1)$, cf. \cite[Definition 2.1]{DKMO}. A crucial property of any such entropy is that $\Phi$ satisfies a certain equation relating $\nabla \cdot [ \Phi(u)]$ and $\nabla(1-|u|^2)$ for any $u \in H^1(\Omega;\R^2)$. We state this equation precisely in \eqref{3.2}, and refer the reader to \cite[Lemmas 2.2, 2.3]{DKMO} for the proof, which is a straightforward calculation. In Lemma \ref{L2.2}, we prove that the class of entropies is large enough for our purposes. Next, in Proposition \ref{P1.2}, we prove the requisite compactness for the sequence $\{u_\e\}$. We achieve this by first adapting the proof of \cite[Proposition 1.2]{DKMO} using the aforementioned equation \eqref{3.2} to show that for any entropy $\Phi$, 
\begin{align}\notag
\left\{ \nabla \cdot [\Phi(u_\e)]\right\} \textup{ is compact in }H^{-1}(\Omega).
\end{align}
This compactness then allows us to use the div-curl lemma of Murat and Tartar \cite{murat2,tartar} and the result of Lemma \ref{L2.2} to conclude that each $\mu_x$ is a Dirac measure. One can then quickly deduce, in the same fashion as in \cite[page 843]{DKMO}, that the sequence $\{u_\e\}$ is precompact is $L^2$. We begin the proof with 
\begin{lemma}{(cf. \cite[Lemma 2.2]{DKMO})}\label{L2.2}
Let $\mu$ be a probability measure on $\R^2$ supported on $\mathbb{S}^1 \cup \{0\}$. Suppose it has the property
\begin{align}\label{el2.2}
\int \Phi \cdot \tilde{\Phi}^\perp \,d\mu =  \int \Phi\,d\mu \cdot \int  \tilde{\Phi}^\perp \,d\mu \textit{ for all entropies }\Phi, \tilde{\Phi}.
\end{align}
Then $\mu$ is a Dirac measure.
\end{lemma}
\brk We point out that the proof of this lemma does not generalize to the case where the potential vanishes on a pair of circles that both have non-zero radius. As a consequence, this proof of Theorem \ref{fullcompact} does not generalize to such situations. 
\erk
\begin{proof}
We begin by recalling the definition of ``generalized entropy" from \cite[Lemma 2.5]{DKMO}. These are functions $\Phi$ defined by
\begin{align}\notag
\Phi(z)&= \left\{ \begin{array}{cc}
|z|^2e & \mbox{ for }z\cdot e >0 \\
0 & \mbox{ for } z\cdot e \leq 0
\end{array}
\right.
\end{align}
for any fixed $e\in \mathbb{S}^1$.
Any such $\Phi$ can be approximated closely enough by entropies $\Phi_n$ such that \eqref{el2.2} holds for $\Phi$ as well. Using the fact that these generalized entropies vanish at the origin, we have
\begin{align*}
\int_{\mathbb{S}^1} \Phi \cdot \tilde{\Phi}^\perp \,d\mu =  \int_{\mathbb{S}^1} \Phi\,d\mu \cdot \int_{\mathbb{S}^1}  \tilde{\Phi}^\perp \,d\mu .
\end{align*}
We rewrite this as
\begin{align*}
e \cdot \tilde{e}^\perp\,\mu(\{z\cdot e >0\}\cap \{ z\cdot \tilde{e}>0\}\cap \mathbb{S}^1)=e \cdot \tilde{e}^\perp\mu(\{z\cdot e >0\}\cap \mathbb{S}^1)\,\mu ( \{ z\cdot \tilde{e}>0\}\cap \mathbb{S}^1)\textup{  for all }e,\tilde{e}\in \mathbb{S}^1
\end{align*}
or
\begin{align*}
\mu(\{z\cdot e >0\}\cap \{ z\cdot \tilde{e}>0\}\cap \mathbb{S}^1)=\mu(\{z\cdot e >0\}\cap \mathbb{S}^1)\,&\mu ( \{ z\cdot \tilde{e}>0\}\cap \mathbb{S}^1)\\ &\textup{ for all }\tilde{e} \in \mathbb{S}^1\setminus\{e,-e\} \textup{ and all }e\in \mathbb{S}^1.\end{align*}
Letting $\tilde{e}$ approach $e$, we obtain
\begin{align*}
\mu(\{z\cdot e >0 \}\cap \mathbb{S}^1)\leq \mu(\{z\cdot e >0\}\cap \mathbb{S}^1)\mu(\{z\cdot e \geq 0\}\cap \mathbb{S}^1)\textup{ for all }e \in \mathbb{S}^1
\end{align*}
or
\begin{align}\label{option}
\mu(\{z\cdot e >0\}\cap \mathbb{S}^1)=0 \textup{  or  }\mu(\{z\cdot e \geq 0\}\cap \mathbb{S}^1)\geq 1 \textup{ for all }e\in \mathbb{S}^1.
\end{align}
If $\mu(\{0\})>0$ then it cannot be that $\mu(\{z\cdot e \geq 0\}\cap \mathbb{S}^1)\geq 1$ for any $e \in \mathbb{S}^1$. In this case $\mu(\{z\cdot e >0\}\cap \mathbb{S}^1)=0$ for all $\mathbb{S}^1$-valued $e$, and $\mu$ is clearly a Dirac measure concentrated at zero. So we may assume that $\mu(\{0\})=0$, implying that $\mu$ is a probability measure on $\mathbb{S}^1$. In this case, we deduce from \eqref{option} that
\begin{align*}
\supp \mu \subset \{ z\cdot e \leq 0\}\cap \mathbb{S}^1\textup{  or  }\supp \mu \subset \{z\cdot e \geq 0\}\cap \mathbb{S}^1\textup{ for all }e\in \mathbb{S}^1.
\end{align*}
As $\mu$ is a probability measure on $\mathbb{S}^1$, this implies that $\mu$ is concentrated on a single point.
\end{proof}
We can now prove the main result.
\bprop{(cf. \cite[Proposition 1.2]{DKMO})}\label{P1.2}
Let $\Omega \subset \R^2$ be open and bounded. Let $\{ u_\e\}\subset H^1(\Omega;\R^2)$ be such that
\begin{align}\label{1.4}
\nabla \cdot u_\e \textup{ are uniformly bounded in }L^2,
\end{align}
\begin{align}\label{1.5}
\| H(|u_\e|)\|_{L^2(\Omega)} \xrightarrow[\e \to 0]{} 0,
\end{align}
and
\begin{align}\label{1.6}
\|\nabla u_\e \|_{L^2}\|H(|u_\e|) \|_{L^2} \textup{ are uniformly bounded.}
\end{align}
Then
\begin{align}\label{1.7}
\{u_\e \} \subset L^2(\Omega;\R^2) \textup{ is relatively compact.}
\end{align}
\eprop
\begin{proof}
First, we modify our sequence slightly for convenience. By choosing real numbers $r_\e$ close enough to $1$ and considering the sequence $\{r_\e u_\e \}$, we can without loss of generality assume that for each $\e$,
\begin{align}\label{zeromeas}
\abs{\bigg\{x \in \Omega:|r_\e u_\e(x)|=\frac{1}{\sqrt{2}} \bigg\}}=0.
\end{align}
In addition, we can choose $r_\e$ so that $\{r_\e u_\e\}$ has uniformly bounded energies and is precompact in $L^2(\Omega;\R^2)$ if and only if $\{u_\e \}$ is as well. Henceforth we refer to the modified sequence as simply $\{u_\e\}$ and assume that these conditions hold for $u_\e$. \par

We aim to show for any entropy $\Phi$ that
\begin{align}\label{3.1}
\left\{ \nabla \cdot [\Phi(u_\e)]\right\} \textup{ is compact in }H^{-1}(\Omega).
\end{align}
Utilizing \eqref{1.4}, \eqref{zeromeas} and \cite[Lemmas 2.2, 2.3]{DKMO}, we see that there exists $\Psi \in C_0^\infty(\R^2)^2$ and $\alpha \in C_0^\infty(\R^2)$ such that at a.e. point in $\Omega$ one has
\begin{eqnarray*}
-\alpha(u_\e)\dive u_\e + \nabla \cdot [\Phi(u_\e)]&&=  
\Psi(u_\e)\cdot \nabla(1-|u_\e|^2) =
-\Psi(u_\e)\cdot \nabla(|u_\e|^2) \\
&&= \sgn\left(|u_\e|-1/\sqrt{2}\right)\Psi(u_\e)\cdot \nabla H(|u_\e|).\label{3.2}
\end{eqnarray*}
Before proceeding, we let $s:\R \to \R$ denote a smooth, increasing, bounded function with bounded derivative such that $s(z) \equiv -1$ for $z \leqslant \frac{1}{2\sqrt{2}}$ and $s(z) \equiv 1$ for $z \geqslant \frac{1}{2\sqrt{2}}$. We will utilize the sequence $s(|u_\e|)$, and we remark that $\frac{1}{2\sqrt{2}}$ could readily be replaced by any number less than $\frac{1}{\sqrt{2}}$. Replacing $\sgn(|u_\e| - 1/\sqrt{2})$ by $s(\abs{u_\e})$ as such allows us to maintain $L^1$ control on $\nabla s(\abs{u_\e})$ as opposed to having to analyze the distributional gradient of a sgn function, as will be necessary in a step at the end of the proof. Continuing from \eqref{3.2}, we find that 
\begin{align} \label{3.4a}
-\alpha(u_\e) \dive u_\e + \nabla \cdot [\Phi(u_\e)] = s(|u_\e|) \Psi(u_\e) \cdot \nabla H(|u_\e|) + R_\e, \\ \label{3.4b}
R_\e := \left(\sgn\left( |u_\e| - 1/\sqrt{2}\right) - s(|u_\e|)\right) \Psi(u_\e) \cdot \nabla H(|u_\e|).
\end{align}
We claim the remainder terms $R_\e$ are bounded uniformly in $L^1$. Noticing that $s(|u_\e|)=\sgn(|u_\e-1/\sqrt{2}|)$ if $\left||u_\e|-1/\sqrt{2}\right|\geq \frac{1}{2\sqrt{2}}$, we have
\begin{align*}
\int_\Omega |R_\e| &\leq C\int_\Omega \left|\chi_{\left\{\left||u_\e|-\frac{1}{2\sqrt{2}}\right|<\frac{1}{2\sqrt{2}}\right\}}\right| \left|\nabla |u_\e|\right|\,dx. \end{align*}
Continuing now using Holder's inequality and the bound $\int_\Omega W(u_\e) \leq C\e$, we have
\begin{align*}
\int_\Omega |R_\e|&\leq C\norm{\nabla |u_\e|}_{L^2}\cdot \mbox{meas}\left\{\abs{\abs{u_\e} - \frac{1}{\sqrt{2}}} < \frac{1}{2\sqrt{2}} \right\}^{1/2}\\
&\leqslant C\frac{1}{\sqrt{\e}}\sqrt{\e}\\
&\leq C.
\end{align*}\par
To prove \eqref{3.1}, we will prove that the sequence \begin{align}\label{minus}
\left\{\nabla \cdot \left[\Phi(u_\e)-s(|u_\e|) H(|u_\e|)\Psi(u_\e)\right] \right\}\textup{ is compact in }H^{-1}(\Omega).\end{align}
Since the energy bound implies that $s(|u_\e|) H(|u_\e|)\Psi(u_\e)$ converges to $0$ in $L^2$, the divergence of this sequence converges to $0$ in $H^{-1}$. Thus \eqref{minus} implies \eqref{3.1}. Thanks to \eqref{3.4a}, we have that 
\begin{align*}
\nabla \cdot [\Phi(u_\e) - s H(|u_\e|) \Psi(u_\e)] &= \nabla \cdot [\Phi (u_\e)] - \nabla \cdot [s \Psi(u_\e)] H(|u_\e|) - s \Psi(u_\e) \cdot \nabla [H(|u_\e|)] \\
&= R_\e + \alpha(u_\e) \dive u_\e - \nabla \cdot [s \Psi(u_\e)] H(|u_\e|)
\end{align*}\par
We will show the desired compactness by appealing to a lemma of \cite{Murat}, cf. \cite[Lemma 3.1]{DKMO}. This entails verifying the following two claims: \\
(1) The sequence $\big\{\nabla \cdot [\Phi(u_\e) - s H(|u_\e|) \Psi(u_\e)]\big\}$ is uniformly bounded in $L^1(\Omega).$ \\
(2) The sequence $\big\{|\Phi(u_\e) - s(|u_\e|) H(|u_\e|) \Psi(u_\e) |^2\big\}_{\e > 0}$ is uniformly integrable.\\
\underline{Proof of (1):} We have shown that the $R_\e$ are uniformly bounded in $L^1,$ and the boundedness of the function $\alpha$ along with the $L^2 $ bound on $\dive u_\e$ yield that $\alpha(u_\e)\dive u_\e$ is uniformly bounded in $L^1.$ It remains to show that the last term, namely $\nabla \cdot [s(|u_\e|) \Psi(u_\e)] H(|u_\e|)$, is bounded in $L^1.$ We have
\begin{align*}
H(|u_\e|)\nabla \cdot [s \Psi(u_\e)] = H(|u_\e|) \left( s^\prime \left( |u_\e| \right) \nabla |u_\e| \cdot \Psi(u_\e) + s(|u_\e|) \dive \Psi(u_\e)\right).
\end{align*}
The desired $L^1$ bound follows from Cauchy-Schwarz along with the energy bound. \\ 
\underline{Proof of (2):} This is clear from the fact that $\Phi$, $s$, $H$, and $\Psi$ are bounded functions.\par
We have now proved that $\left\{ \nabla \cdot [\Phi(u_\e)]\right\}$ is compact in $H^{-1}(\Omega)$.  The rest of the proof follows as in the second step of \cite[Propositon 1.2]{DKMO}.
\end{proof}

\subsection{The $\Gamma$-limit of $E_\e$ among 1d competitors}\label{oned}
In this section we analyze $\Gamma$-convergence of $E_\e$
where competitors $u_\e=(u_\e^{(1)},u_\e^{(2)})$ are defined on an interval $[-H,H]$ for some $H>0$ and are required to satisfy $\mathbb{S}^1$-valued boundary conditions of the form
\beq
u(\pm H)=(\pm \sqrt{1-a^2},a)\quad\mbox{for some}\;a\in [0,1).\label{1dbc}
\eeq
Under the one-dimensional assumption, $E_\e$ takes the form
\begin{align}
    E_\e^{1D}(u) := \left\{ 
    \begin{array}{cc}
        \displaystyle\frac{1}{2}\displaystyle\int_{-H}^H \left(\frac{1}{\e} W(u)+\e|u^\prime|^2+ L (u^{(2)}\,')^2 \right)  \,d{x_2}  &\mbox{if}\; u \in H^1\big( (-H,H); \R^2\big),  \\\\
        +\infty  &  \mbox{ otherwise }.
    \end{array}
    \right.
\end{align}

In a manner similar to \cite{GSV}, Section 6, within this one-dimensional ansatz we can obtain a sharp compactness theorem for energy bounded sequences, a complete $\Gamma-$convergence result of the $E_\e$ functionals and a complete characterization of minimizers of the $\Gamma$-limit. Since the proofs of the results in this section are completely analogous to those in \cite[Section 6]{GSV}, we only sketch the arguments highlighting differences. 

In Section \ref{1dex}, we present results of numerical simulations obtained via gradient flow for $E_\e$ with $\e > 0$ for the two-dimensional problem in a rectangle $\Omega = (-1/2, 1/2) \times (-H,H)$, subject to the boundary conditions \eqref{1dbc} on the top and bottom and periodic boundary conditions on the left and right sides. These computations suggest convergence in large time to configurations that resemble the one-dimensional minimizers of this section, lending further evidence to our conjecture \eqref{conjecture}. We emphasize that the initial data for these numerics were \textit{not} restricted to be one-dimensional. 

We continue making the assumption \eqref{hat} on our potentials.
Recall that we are writing $W(u) = V(|u|)$. We begin with a compactness result.
\begin{thm}
Let $u_\e=(u_\e^{(1)},u_\e^{(2)})\in H^1((-H,H);\R^2)$ with $E_\e^{1\textup{D}}(u_\e)\leq C$.
Then there exists $u=(u_1,u_2)$ with
\begin{equation}\notag
    \Psi(u_1) := \int_{-u_{1}}^{u_{1}} \sqrt{V}\left(\sqrt{s^2+1-u_1^2} \right)\,ds\in BV((-H,H);[0,1] )
\end{equation}
such that up to a subsequence, $\uone \to u_1$ in $L^2(-H,H)$. In addition, $u_2\in H^1((-H,H);\R)$, $\utwo \to u_2$ in $C^{0,\gamma}(-H,H)$ for all $\gamma < 1/2$, and $|(u_1,u_2)|=1 \textup{ or }0$ $a.e.$ 
\end{thm}
\begin{proof} Throughout the course of the proof, we repeatedly pass to further and further subsequences of $\e$ converging to zero but suppress this from our notation. 
We notice that thanks to the uniform $L^{4}$ bound from \eqref{hat}, after passing to a subsequence,
\begin{align} \label{wk}
    u_\e \rightharpoonup u = (u_1, u_2) \mbox{ in } L^{4}.
\end{align}
Furthermore, this bound, along with the uniform $L^2$ bound on $(u_\e^{(2)})^\prime$ yields after passing to a further subsequence that
\beq
 u_\e^{(2)} \rightharpoonup u_2 \mbox{ in } H^1\quad\mbox{and}\quad 
    u_\e^{(2)} \to u_2 \mbox{ in } C^{0,\gamma}([-H,H]) \;\mbox{ for every } \gamma < 1/2.\label{L2conv}
\eeq
Finally, from the bound on the potential, there exists $\rho \in L^2((-H,H);\{0,1\})$ such that 
\begin{align} \label{normconv}
    |u_\e| \to \rho \mbox{ in } L^2. 
\end{align}\par
It remains to upgrade the convergence of $u_\e^{(1)}$ from weak to strong convergence. An algebraic identity is used in the proof of \cite[Theorem 6.1]{GSV} to obtain strong convergence. Here, without an explicit expression for $W,$ we proceed differently. As in \cite{GSV}, we utilize the ``entropy '' $\psi$ defined by
\begin{align}
    \psi(u_\e) := \int_{-u_\e^{(1)}}^{u_\e^{(1)}} \sqrt{V}\left(\left|(s,u_\e^{(2)})\right| \right)\,ds  = u_\e^{(1)} \int_{-1}^1 \sqrt{V}\left(\left|(u_\e^{(1)}t,u_\e^{(2)})\right|  \right)\,dt .
\end{align}
We set $ J_\e = \displaystyle\int_{-1}^1 \sqrt{V}\left(\left|(\uone t,u_\e^{(2)})\right| \right)\,dt $, so that \begin{equation}\label{jee}
u_\e^{(1)}J_\e = \psi(u_\e).
\end{equation}
On the one hand, 
\begin{align*}
    J_\e = \int_{-1}^1 \sqrt{V}\left(\sqrt{(|u_\e|^2 - (u_\e^{(2)})^2)t^2 + (u_\e^{(2)})^2} \right)\,dt;
\end{align*}
thus \eqref{L2conv}-\eqref{normconv} yield that 
\begin{align}\label{je}
    J_\e \to  \int_{-1}^1 \sqrt{V}\left(\sqrt{(\rho^2 - (u^{(2)})^2)t^2 + (u^{(2)})^2} \right)\,dt =: J \quad \mbox{ a.e. in }  (-H,H). 
\end{align}

On the other hand, using \eqref{hat} and a Cauchy-Schwarz argument completely analogous to that found in \cite{GSV}, we note that $\psi(u_\e)$ is bounded in $BV$. Upon passing to a subsequence, we conclude that $\{\psi(u_\e)\}$ converges in $L^1,$ and upon passing to a further subsequence, $\{\psi(u_\e)\} $ converges almost everywhere. Consequently, using \eqref{jee} and \eqref{je}, we find that $u_\e^{(1)}$ converges almost everywhere as well.

Finally, since $|u_\e^{(1)}| \leqslant |u_\e|$ and $|u_\e| \to \rho$ strongly in $L^2$, we can apply the Lebesgue dominated convergence theorem to conclude that $u_\e^{(1)}$ converges strongly in $L^2$ to some limit. From \eqref{wk}, this limit is $u_1,$ and it follows that $|(u_1,u_2)|=0$ or $1$ a.e. and that the limit of $\psi(u_\e)$ is $\psi(u_1,u_2)\in BV$. Since $\psi$ is $0$ when $u_1=0$, we see that $\Psi(u_1)=\psi(u)$, which concludes the proof.
\end{proof}

We turn next to $\Gamma$-convergence in this one-dimensional setting. The analogue of $E_0$ from \eqref{E0def} is the energy
\begin{align} \label{1dgamm}
    E_0^{1D}(u) := \frac{L}{2} \int_{-H}^H (u^{(2)}\,')^2 \,d{x_2} + \sum_{x_2 \in J_{u^{(1)}} \cap \{|u| = 1\}} K(u^{(2)}(x_2)) + c_0 \sh^0_{(-H,H)} \left(\partial (\{|u| = 1\}) \right).
\end{align}

 One can establish the $\Gamma$-convergence of $E_\e^{1D}$ to $E_0^{1D}$ in a completely analogous manner to the proof of  Theorem 6.2 in \cite{GSV}, so
we omit the details. 

Finally, as in Theorem 6.4 of \cite{GSV}, and with  identical proofs, one can characterize the minimizers of \eqref{1dgamm} explicitly. When the boundary conditions \eqref{1dbc} are different from $(\pm 1 , 0)$ the minimizer is unique and consists of a single wall occuring at $y=0$, no interfaces and bulk contribution in the regions $\{y > 0\} \cup \{ y < 0\}$: the function $u^{(2)}$ is piecewise linear and $u^{(1)}$ jumps from $\sqrt{1 - (u^{(2)})^2}$ to $- \sqrt{1 - (u^{(2)})^2}$ across $y=0$. The optimal jump value is easily determined by optimizing over the bulk and jump contributions. Finally when the Dirichlet boundary conditions on the top and the bottom are given by $(\pm 1, 0),$ we find two parameter regimes similar to the situation in \cite{GSV}. When $L/H$ is smaller than a certain threshold, the minimizer is unique and has both bulk divergence and jump contributions. However for larger $L/H$ values, the minimizer only has perimeter contribution, along with two interfaces, one connecting $(-1,0)$ to $(0,0)$ and the other connecting $(0,0)$ to $(1,0).$ These interfaces divide the interval into subintervals in each of which the minimizer is a constant. See Section \ref{1dex} for numerical simulations.

\subsection{Criticality conditions for $E_0$}\label{critsec}
In this section we will describe criticality conditions associated with critical points $u$ of the conjectured $\Gamma$-limit $E_0$ given by \eqref{E0def}.
For $u\in  (H_{\dive}\cap BV)(\Omega;\mathbb{S}^1\cup\{0\})$, we recall the notation $K(u\cdot\nu)$, with $K$ given by \eqref{1d} or \eqref{CSHwallcost} in the case of the Chern-Simons-Higgs potential, for the cost per arclength of a jump from one $\mathbb{S}^1$-valued state, say $u_1$  to another one, say $u_2$ across a jump set $J_u$, with $\nu$ denoting the unit normal pointing from the $1$ side of a wall to the $2$ side. We recall that for such a jump, an $H_{\dive}$ vector field must satisfy the requirement
\begin{equation}
u_1\cdot \nu=u_2\cdot \nu\quad\mbox{along the jump set}\;J_u.\label{cnormal}
\end{equation}
In light of \eqref{cnormal}, we will sometimes write just $u\cdot\nu$ when evaluating the normal component of $u$ along $J_u$.

We also recall that for portions of $J_u$ corresponding to a jump from the isotropic state $0$ to an $\mathbb{S}^1$-valued state $u$, the cost per unit arclength is given by $\frac{K(0)}{2}$ and condition \eqref{cnormal} becomes simply $u\cdot\nu=0$.

 Parts of the argument follow the same lines as in the proof of Theorem 4.1 in \cite{GSV} except that the cost in that paper is the one associated with a Ginzburg-Landau potential, namely $K_{GL}(u\cdot\nu)$ where
\[
K_{GL}(z):=\int_{-\sqrt{1-z^2}}^{\sqrt{1-z^2}}\left(1-z^2-y^2\right)\,dy,
\]
which can also be written as $\frac{4}{3}\left(1-(u\cdot \nu)^2\right)^{3/2}$ or equivalently $\frac{1}{6}\abs{u_--u_+}^3.$

However, in the present context, we will need to distinguish between variations of the `walls' separating two $\mathbb{S}^1$-valued states and `interfaces' separating the isotropic state from an $\mathbb{S}^1$-valued state. We will also examine criticality conditions at a junction corresponding to the intersection of these two kinds of curves. We begin with:

\bthm\label{critthm1}(Variations that fix the jump set)\\
Consider any $u\in BV(\Omega,  \mathbb{S}^1\cup\{0\}) \cap H_{\rm{div}} (\Omega,  \mathbb{S}^1\cup\{0\})$ such that $u_{\partial\Omega}\cdot\nu_{\partial \Omega}=g\cdot \nu_{\partial \Omega}$ on $\partial\Omega$. Denote by $J_u$ its jump set. Then if the first variation of $E_0$ evaluated at $u$ vanishes when taken with respect to perturbations compactly supported in $\Omega\setminus J_u$, one has the condition
 \begin{equation}
\label{graddiv}
u^\perp\cdot\nabla\dive u =0\mbox{ holding weakly on }\Omega\backslash J_u,
\end{equation}
where $u^\perp=\left(-u_2,u_1\right)$.

Now assume the first variation vanishes at $u$ when taken with respect to perturbations that fix $J_u$ and are supported within any ball $B(p,R)$ centered at a smooth point of $p\in J_u\cap\Omega$. If $J_u$ separates the ball  $B(p,R)$ into two regions where $u$ is given by $\mathbb{S}^1$-valued states $u_1$ and $u_2$ and if the traces $\dive u_1$ and $\dive u_2$ are sufficiently smooth, then one has the condition
\begin{equation}
\label{natbc}
K'(u\cdot\nu)=L[\dive u]\mbox{ on }J_u\cap\Omega,
\end{equation}
where $[\dive u]=\dive u_2-\dive u_1$ represents the jump in divergence across $J_u$ and $\nu$ is the unit normal to $J_u$ pointing from region $1$ to region $2$.
\ethm
\begin{rmrk} 
There is no natural boundary condition analogous to \eqref{natbc} for such variations taken about a point of $J_u$ where $J_u$ separates an isotropic state $0$ from an $\mathbb{S}^1$-valued state since the requirement of tangency in such a configuration is too rigid to allow for
a rich enough class of perturbations.
\end{rmrk}

\begin{proof}
To derive conditions \eqref{graddiv} and \eqref{natbc} we assume that for some point $p\in\Omega$ and for some $R>0$, either
$B(p,R)\cap J_u=\emptyset$ or else $p\in J_u$ and the following conditions hold $B(p,R)$:  \\
(i) The set $B(p,R)\cap J_u$ is a smooth curve, which we denote by $\Gamma$ and which admits a smooth parametrization by arclength, which we denote by $r:[-s_0,s_0]\to\Omega$ for some $s_0>0$ with $r(0)=p.$\\
(ii) On either side of $\Gamma$ the critical point $u$ and $\dive u$ possess smooth traces on $J_u$. We will denote the two components of  $B(p,R)\setminus \Gamma$ by $\Omega_1$ and $\Omega_2$ and we denote $u$ on these two sets by $u_j:\Omega_j\to \mathbb{S}^1$, for $j=1,2$.

We will present the argument for case (ii), indicating how the easier case (i) follows from the same analysis.

To define an allowable perturbation $u^t$ of the critical point $u$ given by $u_1$ and $u_2$, we must maintain
the property of being $\mathbb{S}^1$-valued, so to that end we introduce smooth functions $\phi_j:B(p,R)\times (-T,T)\to\R$ for some $T>0$ such that the perturbations of $u_1$ and $u_2$ take the form
\beq
u^t_j(x):=u_j(x)e^{it\phi_j(x,t)},\label{phit1}
\eeq
shifting just for the moment to complex notation. Introducing $\phi_j(x):=\phi_j(x,0)$, expanding \eqref{phit1} and reverting back to an $R^2$-valued description of $u^t_j$
we find that for $x\in\Omega_j^t$  one has
\beq
u^t_j(x)\sim u_j(x)+t\phi_j(x)u_j(x)^{\perp}.\label{utexp1}
\eeq

Along $J_u$, we must also be sure to preserve to $O(t)$ the $H_\dive$ condition \eqref{cnormal}, namely
\beq
u^t_1\cdot \nu=u^t_2\cdot \nu\quad\mbox{along}\;\Gamma.\label{Hdivet1}
\eeq
Invoking \eqref{cnormal} for the unperturbed critical point, along with \eqref{utexp1} we find that 
$u^t_1\cdot\nu=u^t_2\cdot\nu$ to $O(t)$ provided 
\[
\phi_1u_1^\perp\cdot \nu=\phi_2u_2^\perp\cdot\nu.
\]
However, since
\beq
u_j^\perp\cdot\nu= u_j\cdot\tau\quad\mbox{and}\quad u_1\cdot\tau=-u_2\cdot\tau\not=0\label{usefulstuff}
\eeq
along the jump set $J_u$ bridging $\mathbb{S}^1$-valued states, it follows that we must require
\beq
\phi_1(x)=\phi_2(x)\quad\mbox{for}\;x\in\Gamma\label{phi12}.
\eeq
For later use, we also record that from \eqref{utexp1} and \eqref{usefulstuff} one has along $\Gamma$ the expansion
\beq
u^t\cdot\nu\sim u_1\cdot\nu+t\phi_1(u_1\cdot\tau)+o(t).\label{normt}
\eeq

Now we calculate
\begin{eqnarray*}
&&\frac{d}{dt}_{|_{t=0}} E_0(u^t)=\frac{d}{dt}_{|_{t=0}} \left\{\frac{L}{2}\sum_{j=1}^2\int_{\Omega_j}\big(\dive\,u^t_j\big)^2\,dx\right\}+
\frac{d}{dt}_{|_{t=0}}\int_{\Gamma}K( u^t\cdot\nu)\,ds\\
&&= \frac{d}{dt}_{|_{t=0}} \left\{\sum_{j=1}^2 \frac{L}{2}\int_{\Omega_j}\big(\dive\,(u_j(x)+t\phi_j(x)u_j(x)^{\perp})\big)^2\,dx\right\}+\\
&&\frac{d}{dt}_{|_{t=0}}\int_{\Gamma}K( u_1\cdot\nu+t\phi_1(u_1\cdot\tau))\,ds
\end{eqnarray*}
Taking the $t$-derivatives and evaluating at $t=0$ we obtain
\begin{eqnarray*}
&&\frac{d}{dt}_{|_{t=0}} E_0(u^t)=\left\{L\sum_{j=1}^2\int_{\Omega_j}\big(\dive\,u_j\big)\big(\dive(\phi_ju_j^\perp)\big)\,dx\right\}\\
&&\int_{\Gamma}K'(u\cdot\nu)(u_1\cdot\tau)\,ds.
\end{eqnarray*}
Integrating by parts, a vanishing first variation of this type leads to the condition
\begin{eqnarray}
&&-L\sum_{j=1}^2\int_{\Omega_j}\nabla\big(\dive\,u_j\big)\cdot\phi_ju_j^\perp)\big)\,dx
+L\int_{\Gamma} \left\{\big(\dive\,u_1\big)\phi_1 u_1^\perp\cdot\nu-\big(\dive\,u_2\big)\phi_2 u_2^\perp\cdot\nu\right\}\,ds\nonumber\\
&&+\int_{\Gamma}K'(u\cdot\nu)(u_1\cdot\tau)\,ds=0.\label{fvar}
\end{eqnarray}

Now by taking the functions $\phi_j$ to be supported off of $J_u$ we arrive at condition \eqref{graddiv}. This also handles case (i)
where $B(p,R)\cap J_u=\emptyset.$ Then, in light of \eqref{graddiv}, along with \eqref{usefulstuff} and \eqref{phi12} we find that
\[
\int_{\Gamma}\left\{K'(u\cdot\nu)(u_1\cdot\tau)+L\big(\dive\,u_1-\dive\,u_2\big)\right\}(u_1\cdot\tau)\,\phi_1\,ds=0.
\]
Since $u_1\cdot\tau\not=0$ along the jump set and $\phi_1$ is arbitrary, we arrive at \eqref{natbc}.
\end{proof}

\begin{cor}\label{conservation} (cf. \cite{GSV}, Cor. 4.2). Suppose $u$ is smooth and critical for $E_0$ in the sense of \eqref{graddiv}. Then writing $u$ locally in terms of a lifting as $u(x)=e^{i\theta(x)}$ and defining the scalar $v:=\dive u$ one has that \eqref{graddiv} is equivalent to the following system for the two scalars $\theta$ and $v$:
\begin{eqnarray}
&&-\sin\theta\, \theta_{x_1}+\cos\theta \,\theta_{x_2}=v\label{thetav}\\
&&-\sin\theta\, v_{x_1}+\cos\theta \,v_{x_2}=0.\label{vconstant}
\end{eqnarray}
Consequently, starting from any initial curve in $\Omega$ parametrized via $s\mapsto \big(x_1^0(s),x_2^0(s)\big)$ along which $\theta$ and $v$ take values
$\theta_0(s)$ and $v_0(s)$ respectively, the characteristic curves, say\\
$t\mapsto \big(x_1(s,t),\,x_2(s,t)\big)$, are given by
\begin{eqnarray}
&&x_1(s,t)=\frac{1}{v_0(s)}\left[\cos\big(v_0(s)t+\theta_0(s)\big)-\cos\theta_0(s)\right]+x_1^0(s),\label{xst}\\
&&x_2(s,t)=\frac{1}{v_0(s)}\left[\sin\big(v_0(s)t+\theta_0(s)\big)-\sin\theta_0(s)\right]+x_2^0(s),\label{yst}
\end{eqnarray}
whenever $v_0(s)\not=0.$ The corresponding solutions $\theta(s,t)$ and $v(s,t)$ are given by
\begin{equation}
\theta(s,t)=v_0(s)t+\theta_0(s),\quad v(s,t)=v_0(s),\label{tvst}
\end{equation}
so that the characteristics are circular arcs of curvature $v_0(s)$ and carry constant values of the divergence. In case the divergence vanishes somewhere along the initial curve, i.e. $v_0(s)=0$, then the characteristic is a straight line.
\end{cor}
\vskip.1in
We also consider the implications of criticality with respect to perturbations of the jump set itself.
\bthm\label{critthm2}(Variations of the jump set)\\
Under the same assumptions on $u$ as in the previous theorem, suppose in addition to the criticality with respect to perturbations that fix the location of $J_u$, one also assumes the 
 vanishing of the first variation of $E_0,$ evaluated at $u$,  allowing for local perturbations of the jump set $J_u\cap\Omega$ itself. Then along any points of $J_u$ where $u$ jumps between two $\mathbb{S}^1$-valued maps $u_1$ and $u_2$, a vanishing first variation leads to the condition
\begin{equation}
\frac{L}{2}\left((\dive\,u_1)^2-(\dive\,u_2)^2\right)
-L (\dive\,u_1+\dive\,u_2 )'\,(u_1\cdot\tau)-L (\dive\,u_1+\dive\,u_2)\,(u_1\cdot\tau)'
-K(u\cdot\nu)\,\kappa=0,
\label{wally}
\end{equation}
at any point $p\in J_u$ such that  $J_u$, $u_1$ and $u_2$ are sufficiently smooth in some ball centered at $p$ . Here $\kappa$ denotes the curvature of $J_u$,
 $\tau$ denotes the unit tangent to $J_u$ and $\nu$ is the unit normal to $J_u$ pointing from the $u_1$ side of $J_u$ to the $u_2$ side. The notation $(\cdot)' $ refers to the tangential derivative along the jump set. 
 
 For portions of $J_u$  separating an $\mathbb{S}^1$-valued state $u^*$ from the isotropic state $0$, criticality takes the form
 \beq
 \frac{L}{2}\left(\dive u^*\right)^2-L\left(\dive u^*\right)'(u^*\cdot\tau)+\frac{K(0)}{2}\kappa=\lambda,\label{intercrit}
 \eeq
  where $\lambda$ is a Lagrange multiplier that is present only if $E_0$ is considered subject to an area constraint on the measure of
  the isotropic phase $0$. Also, since $u\in H_\dive(\Omega)$ requires that $u^*\cdot\nu=0$ along such a portion of $J_u$, we note that
  in \eqref{intercrit} one either has $u^*\cdot\tau\equiv 1$ or $u^*\cdot\tau\equiv-1$.
  \ethm
 \noindent
 \begin{proof}
  To derive condition \eqref{wally} assume that for some point $p\in J_u$ the following conditions hold in a ball $B(p,R)$ for some radius $R$:\\
(i) The set $B(p,R)\cap J_u$ is a smooth curve, which we denote by $\Gamma$ and which admits a smooth parametrization by arclength, which we denote by $r:[-s_0,s_0]\to\Omega$ for some $s_0>0$ with $r(0)=p.$\\
(ii) On either side of $\Gamma$ the critical point $u$ is $C^2$ with $C^1$ traces on $J_u$. We will denote the two components of  $B(p,R)\setminus \Gamma$ by $\Omega_1$ and $\Omega_2$ and we denote $u$ on these two sets by $u_j:\Omega_j\to \mathbb{S}^1$, for $j=1,2$.
Again, our convention for the unit normal $\nu$ is that it points out of $\Omega_1$ into $\Omega_2$.

For the calculation it will be convenient to assume that for $j=1,2$,  $u_j$ has been smoothly extended so as to be defined in an open neighborhood of $\Gamma$. We take this extension to be executed so that $u_1$ is constant along $\nu$ and so that $u_2$ is constant along $-\nu$.

In order to effect a smooth perturbation of $\Gamma$, $u_1$ and $u_2$ we now introduce a vector field $X\in C_0^1(B(p,R);\R^2)$. For convenience we will assume that $X(x)$ is parallel to $\nu(x)$ for $x\in \Gamma$ and so we introduce the scalar function $h:[-s_0,s_0]\to\R$ such that
\beq
X(r(s))=h(s)\nu(s)\quad\mbox{for}\;s\in [-s_0,s_0],\label{hdefn}
\eeq
where $h(\pm s_0)=0$.
Here we have written $\nu(s)$ for the composition $\nu(r(s)).$
 Then let $\Psi:B(p,R) \times (-T, T) \rightarrow \Omega$ solve
\beq\label{odepsi} \frac{\partial \Psi}{\partial t} \,= \, X (\Psi)
\qquad \Psi (x, 0) \, = \, x, \eeq for some $T >0$. Expanding in $t$ we find that 
\begin{equation}D\Psi (\cdot, t) \, \sim\, I \, + \, t \nabla X \, +o(t),\label{psiexp}
\end{equation}
so that in particular one has the identity
\beq
J\Psi(x,t):={\rm det}\,D\Psi\sim 1+t\dive\,X(x)+o(1).\label{Jac}
\eeq
Throughout this proof, the symbol $\sim$ refers to an equivalence up to terms that are $o(t)$.

Now we define the evolution of the curve $\Gamma$ via the vector field $X$ by
$
\Gamma^t:=\Psi(\Gamma,t)
$
with corresponding parametrization
\beq
r^t(s):=\Psi(r(s),t)\sim r(s)+X(r(s))t\sim r(s)+th(s)\nu(s),\label{tparam}
\eeq
in light of \eqref{hdefn}. A simple calculation goes to show that the normal $\nu^t$ to $\Gamma^t$ takes the form
\beq
\nu^t(s) \sim \nu(s)-th'(s)\tau(s),\label{nut} 
\eeq
where we have introduce the notation $\tau$ for the unit tangent $r'(s)$ to $\Gamma$.
We caution that the parameter $s$ used to parametrize $\Gamma^t$ is not an arclength parametrization on this deformed curve.
Indeed one finds through an application of the Frenet equation that 
\[
r^t\,'(s)=r'(s)+th'(s)\nu(s)+th\nu'(s)=(1-th(s)\kappa(s))\tau(s)+th'(s)\nu(s)
\]
where $\kappa$ denotes the curvature of $\Gamma$, so that
\beq
\abs{r^t\,'(s)}\sim 1-th(s)\kappa(s).\label{notarc}
\eeq

Similarly, we define the deformation of the two sets $\Omega_1$ and $\Omega_2$ via
\beq\label{Omegat}
 \Omega_j^t \, : = \, \Psi(\Omega_j, t)\quad\mbox{for}\;j=1,2.
\eeq 
To define the allowable evolution of the critical point $u$ given by $u_1$ and $u_2$ requires a little more care. Firstly, we must maintain
the property of being $\mathbb{S}^1$-valued, so to that end we introduce smooth functions $\phi_j:B(p,R)\times (-\tau,\tau)\to\R$ such that the perturbations of $u_1$ and $u_2$ take the form
\beq
u^t_j(x):=u_j(x)e^{it\phi_j(x,t)},\label{phit}
\eeq
shifting just for the moment to complex notation. Introducing $\phi_j(x):=\phi_j(x,0)$, expanding \eqref{phit} and reverting back to an $R^2$-valued description of $u^t_j$
we find that for $x\in\Omega_j^t$  one has
\beq
u^t_j(x)\sim u_j(x)+t\phi_j(x)u_j(x)^{\perp}.\label{utexp}
\eeq
As before $(a,b)^\perp=(-b,a).$

Secondly, we must preserve to $O(t)$ the $H_\dive$ condition \eqref{cnormal}, namely
\beq
u^t_1\cdot \nu^t=u^t_2\cdot \nu^t\quad\mbox{along}\;\Gamma^t.\label{Hdivet}
\eeq
To this end, we observe that along $\Gamma^t$ one has
\begin{eqnarray}
u^t_j(r^t(s))&\sim& u^t_j\big(r(s)+th(s)\nu(s)\big)+t\phi_j\big(r(s)+th(s)\nu(s)\big)u_j^\perp\big(r(s)+th(s)\nu(s)\big)\nonumber\\
&\sim& u_j(r(s))+t
\phi_j(r(s))u_j^\perp(r(s)).\label{gtexp}
\end{eqnarray}
It is here that we require the slight extensions of the original functions $u_j$ that are constant along the normal direction of $\nu$  to make \eqref{utexp} well-defined in $\Omega_j^t\setminus \Omega_j$ and to make \eqref{gtexp} correct to $O(t)$.

Then once we apply \eqref{nut} and \eqref{gtexp} to \eqref{Hdivet} 
we arrive at the requirement that 
\beq
(u_1+t\phi_1u_1^\perp)\cdot (\nu  - t h'\tau)\sim (u_2+t\phi_2u_2^\perp)\cdot (\nu-th'\tau)\quad\mbox{along}\;\Gamma.\label{tnorm}
\eeq
Equating terms at $O(t)$ and using \eqref{usefulstuff},
we find that necessarily,
\beq
h'(s)=\frac{1}{2}\big(\phi_1(r(s))+\phi_2(r(s))\big)\quad\mbox{for}\;s\in[-s_0,s_0].\label{hphi}
\eeq
For later use we also record that fact, based on expanding the left hand side of \eqref{tnorm}, that 
\beq
u^t_1(r^t(s))\cdot \nu^t(s)\sim u(r(s))\cdot\nu(s)+
t\big(\phi_1(r(s))-h'(s)\big)(u_1\cdot\tau)\quad\mbox{for}\;s\in [-s_0,s_0].\label{normcomp}
\eeq

With these preliminaries taken care of, we are now ready to proceed with the calculation of the first variation 
$\frac{d}{dt}_{|_{t=0}} E_0(u^t).$ We begin with the variation of the divergence term in the energy taken over $\Omega^t_j$ for $j=1,2$. We observe that
\begin{eqnarray*}
&& \int_{\Omega_j^t}\big(\dive\,u^t_j\big)^2\,dx\sim
 \int_{\Omega^t_j}\big(\dive\,u_j+t\,\dive\,(\phi_ju_j^\perp)\big)^2\,dx\\
 &&\sim\int_{\Omega_j}\left[\dive\,u_j(\Psi(y,t))+t\,\dive\,\phi_j(\Psi(y,t))u_j^\perp(\Psi(y,t))\right]^2(1+t\dive X(y))\,dy,
\end{eqnarray*}
where we have utilized the change of variables $x=\Psi(y,t)$ and invoked \eqref{Jac} to obtain the leading order behavior of the Jacobian of the change of variables. Then, since $\Psi\sim y+tX(y)$ we find
\begin{eqnarray}
&&
\frac{d}{dt}_{|_t=0}  \int_{\Omega^t_j}\big(\dive\,u^t_j\big)^2\,dx=\nonumber\\
&&\int_{\Omega_j}\left[
(\dive\,u_j(y))^2\,\dive X+2\,\dive\,u_j(y)\dive\,(\phi_j(y)u_j^\perp(y))+\frac{\partial}{\partial t}_{|_{t=0}}\left(\dive\,u_j(y+tX(y)\right)^2
\right]\,dy=\nonumber\\
&&
\int_{\Omega_j}\left[
(\dive\,u_j(y))^2\dive X+2\,\dive\,u_j(y)\dive\,(\phi_j(y)u_j^\perp(y))+2\,\dive\,u_j(y)\nabla\dive\,u_j(y)\cdot X(y)\right]\,dy=\nonumber\\
&&
\int_{\Omega_j}\left[\dive\left((\dive\,u_j)^2X\right)+2\,\dive\,u_j(y)\dive\,(\phi_j(y)u_j^\perp(y))\right]\,dy\nonumber\\
&&\label{jacint}
\end{eqnarray}
Applying the divergence theorem, and invoking \eqref{graddiv} along with the compact support of $X$ within the ball $B(p,R)$, we conclude that
\begin{eqnarray}
&&
 \frac{d}{dt}_{|_{t=0}}\frac{L}{2}\left(\int_{\Omega^t_1}(\dive\,u^t_1)^2\,dx+\int_{\Omega^t_2}(\dive\,u^t_2)^2\,dx\right)=\nonumber\\
&&
\frac{L}{2}\int_{\Gamma}\left\{\left((\dive\,u_1)^2-(\dive\,u_2)^2\right)h+2\,(\dive\,u_1)\phi_1u_1^\perp\cdot\nu-
2\,(\dive\,u_2)\phi_2u_2^\perp\cdot\nu\right\}\,ds=\nonumber\\
&&
\frac{L}{2}\int_{\Gamma}\left\{\left((\dive\,u_1)^2-(\dive\,u_2)^2\right)h+2\,\left(\phi_1\dive\,u_1+\phi_2\dive\,u_2\right)(u_1\cdot\tau)\right\}\,ds,
\label{divenergy}
\end{eqnarray}
where in the last line we used \eqref{usefulstuff}.

We turn now to the variation of the jump energy. By \eqref{normcomp} we have
\begin{eqnarray*}
&&
K(u^t(r^t(s))\cdot\nu^t(s))\sim K\big(u(r(s))\cdot\nu(s)+t(\phi_1(r(s))-h'(s))(u_1(r(s)\cdot\tau(s))\big)\\
&&
\sim K\big(u\cdot\nu\big)+tK'\big(u\cdot\nu\big)\,(\phi_1-h')(u_1\cdot\tau),
\end{eqnarray*}
where all terms in the last line are evaluated along $\Gamma$, that is, evaluated at $x=r(s)$.
Then we can appeal to \eqref{notarc} to calculate that
\begin{eqnarray}
&&
 \frac{d}{dt}_{|_{t=0}}\int_{\Gamma^t}K(u^t\cdot\nu^t)\,ds=\nonumber\\
 &&
  \frac{d}{dt}_{|_{t=0}}\int_{-s_0}^{s_0}\bigg\{(K(u(r(s))\cdot\nu(s))+t\,K'(u(r(s))\cdot\nu(s))(\phi_1(r(s))-h'(s))(u_1(r(s)\cdot\tau(s))\bigg\}
  \bigg\{1-th(s)\kappa(s)\bigg\}\,ds\nonumber\\
  &&
 =\int_{\Gamma}K'(u\cdot\nu)(\phi_1-h')(u_1\cdot\tau)-K(u\cdot\nu)\,h\kappa\,ds\nonumber\\
 &&
 =\int_{\Gamma}L\big(\dive\,u_2-\dive\,u_1\big)(\phi_1-h')(u_1\cdot\tau)-K(u\cdot\nu)\,h\kappa\,ds,\nonumber\\
 &&\label{wallvar}
\end{eqnarray}
 where in the last line we have used the criticality condition \eqref{natbc}.

Combining \eqref{divenergy} and \eqref{wallvar} we obtain
\begin{eqnarray*}
&&\frac{d}{dt}_{|_{t=0}} E_0(u^t)=\\
&&
\int_{\Gamma}\left\{\frac{L}{2}\left((\dive\,u_1)^2-(\dive\,u_2)^2\right)-K(u\cdot\nu)\,\kappa \right\}\,h\,ds\\
&&
+L\int_{\Gamma}\left\{ (\phi_1+\phi_2)\dive\,u_2+(\dive\,u_1-\dive\,u_2)h'  \right\}\,(u_1\cdot\tau)\,ds\\
&&
=\int_{\Gamma}\left\{\frac{L}{2}\left((\dive\,u_1)^2-(\dive\,u_2)^2\right)-K(u\cdot\nu)\,\kappa \right\}\,h\,ds\\
&&
+L\int_{\Gamma}\left (\dive\,u_1+\dive\,u_2  \right)\,(u_1\cdot\tau)h'\,ds,
\end{eqnarray*}
in light of \eqref{hphi}. Integrating by parts in the last integrals, and using that $h(-s_0)=h(s_0)=0$ we finally obtain
\begin{eqnarray*}
&&\frac{d}{dt}_{|_t=0} E_0(u^t)=\\
&&\int_{\Gamma}\left\{\frac{L}{2}\left((\dive\,u_1)^2-(\dive\,u_2)^2\right)
- L(\dive\,u_1+\dive\,u_2 )'\,(u_1\cdot\tau)- L(\dive\,u_1+\dive\,u_2)\,(u_1\cdot\tau)'
-K(u\cdot\nu)\,\kappa \right\}\,h\,ds.
\end{eqnarray*}
Since criticality implies that this last integral must vanish for all $h$, we obtain \eqref{wally}. 

The derivation of \eqref{intercrit} follows along similar lines so we omit the details. One difference to note, however, is that in the presence of an area constraint on the measure of
$\{u\equiv 0\}$, the normal component $h$ of the vector field $X$ along $\Gamma$ must additionally satisfy the requirement
\[
\int_{-s_0}^{s_0}h(s)\,ds=0
\]
so that the perturbed jump set preserves area to $O(t)$.
This condition leads to the appearance of the Lagrange multiplier in \eqref{intercrit}.
\end{proof}
\vskip.2in
Our last consequence of criticality for a vector field $u$ with respect to the functional $E_0$ concerns the possible presence in $\Omega$ of a junction point $P$ such that for some $R>0$, the set
$B(p,R)\cap J_u$ consists of four curves meeting at $p$. We wish to focus on the configuration where two of these curves, which we label as $\Gamma_{01}$ and $\Gamma_{03}$, are interfaces separating an isotropic region, which we label as $\Omega_0$, from two disjoint regions, $\Omega_1$ and $\Omega_3$, where $u$ is given by $u_1:\Omega_1\to \mathbb{S}^1$ and $u_3:\Omega_3\to \mathbb{S}^1$, respectively. Wedged between $\Omega_1$ and $\Omega_3$ we assume there exists a set $\Omega_2$ where $u$ takes on another $\mathbb{S}^1$-valued state $u_2$. The dashed curve separating $\Omega_1$ from $\Omega_2$, representing the wall across which $u$ jumps from $u_1$ to $u_2$ we denote by $\Gamma_{12}$, and the dashed curve separating $\Omega_2$ from $\Omega_3$, representing the wall across which $u$ jumps from $u_2$ to $u_3$ we denote by $\Gamma_{12}$. We write $\tau_{ij}$ and $\nu_{ij}$ for the unit tangent and unit normal to the curve $\Gamma_{ij}$ where each $\tau_{ij}$ points away from the junction $P$ and $\nu_{ij}$ points from the region $\Omega_i$ into the region $\Omega_j$. See Fig. \ref{Junctionfig}.
\afterpage{
\begin{figure}[h]
\centering
\includegraphics[scale=.80]{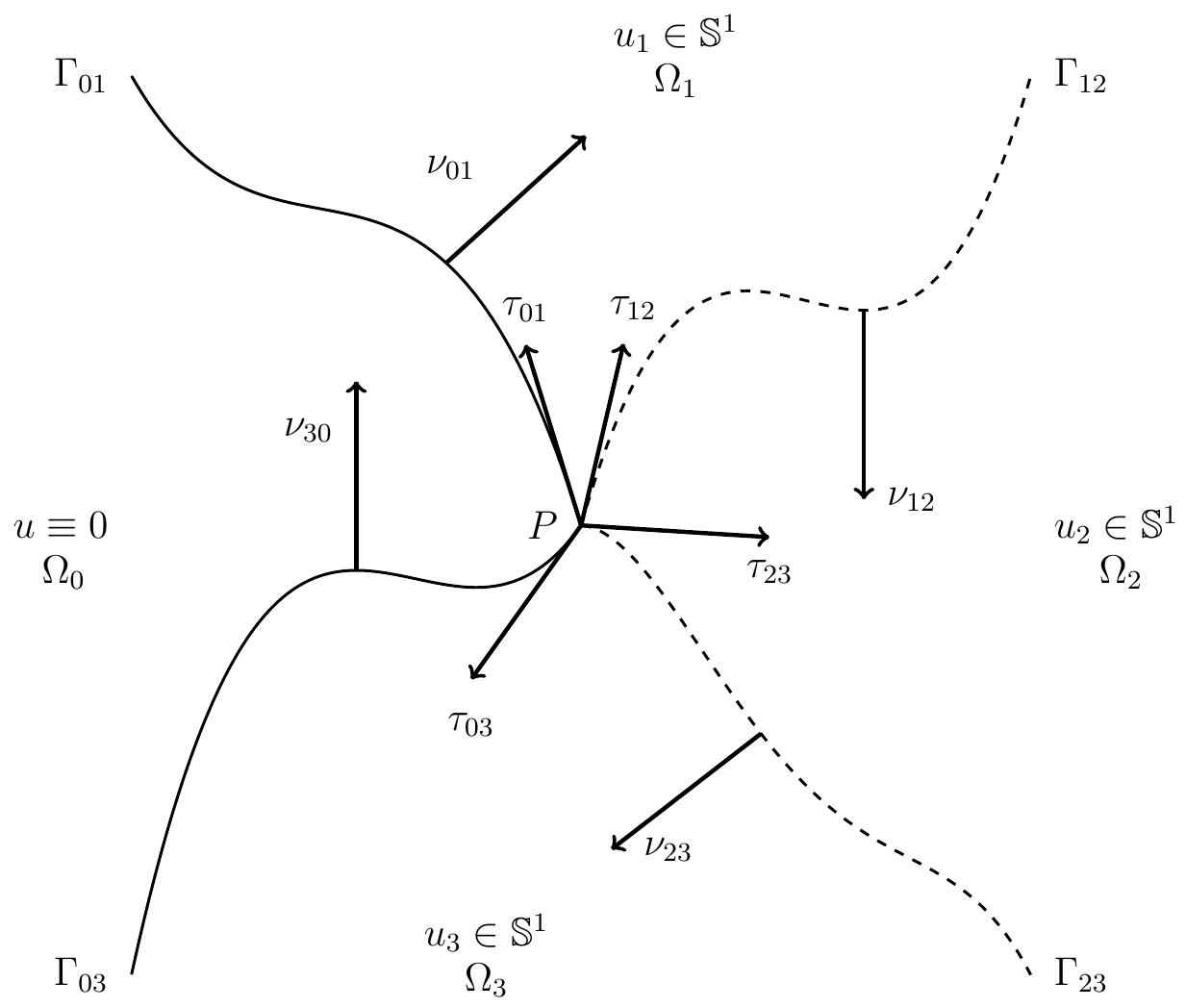}
\caption{A configuration with a junction point $P$ at which two components of the interface, $\Gamma_{01}$ and $\Gamma_{03}$ meet two components of the wall $\Gamma_{12}$ and $\Gamma_{23}$.}
 \label{Junctionfig}
\end{figure}
}

Our reason for focusing on this particular configuration is predicated on the belief that it is somehow quite generic behavior in a neighborhood of a singular point on the isotropic-nematic phase boundary; see the discussion in Section \ref{eyeballs}. This belief is grounded in the findings of numerous numerical experiments we have conducted and examples we have constructed for this model, some of which appear in the last section of this article. Our hope is that the condition derived in Theorem \ref{junction} below will be of use in constructing particular candidates for minimizers of $ E_0 $ as well as perhaps being of use in ruling out certain junction configurations that are found to violate \eqref{junk}.

To state the next result we must introduce the notation $\tau_{ij}$ for the unit tangent on $\Gamma_{ij}$ oriented so as to point away from $P$, and $\nu_{ij}$ for the unit normal to $\Gamma_{ij}$, pointing from region $\Omega_i$ into $\Omega_j$.

\bthm\label{junction}(Criticality conditions at a junction). Assume a configuration in a neighborhood of a point $P\in\Omega$ as described above and as depicted in Fig. \ref{Junctionfig}. Assume that in a neighborhood of $P$ the functions $u_j$ and their divergences $\dive\,u_j$ for $j=1,2,3$ are all smooth in the closure of $\Omega_j$ including at the junction point $P$. Assume further that the four curves $\Gamma_{01},\Gamma_{12}, \Gamma_{23}$ and $\Gamma_{03}$ are all smooth near $P$. Then criticality of $E_0$ with respect to variations
of $P$, the four curves $\Gamma_{ij}$ and the three functions $u_j$ leads to the condition
\begin{eqnarray}
&&\frac{K(0)}{2}\big(\tau_{01}+\tau_{03}\big)+K\big(u_1\cdot\nu_{12}\big)\tau_{12}+K\big(u_2\cdot\nu_{23}\big)\tau_{23}\nonumber\\
&&
=L\bigg\{ (\dive\,u_1)(u_1\cdot\tau_{01}) \nu_{01}  + (\dive\,u_3)(u_3\cdot\tau_{03}) \nu_{03}    \bigg\} \nonumber\\
&&
-L\bigg\{\big(\dive\,u_1+\dive\,u_2\big)\big(u_1\cdot\tau_{12}\big)\nu_{12}+  
             \big(\dive\,u_2+\dive\,u_3\big)\big(u_2\cdot\tau_{23}\big)\nu_{23}  \bigg\}\label{junko}
\end{eqnarray}
where all quantities above are evaluated at $P$.
\ethm
The proof of Theorem \ref{junction} can be found in the appendix.

\section{Examples: Analytical constructions for large $L$ and some numerics}\label{examp}
We conclude with an exploration of possible morphologies for our limiting energy $E_0$, which we recall is given by 
\begin{equation}\notag
E_0(u)=\f{L}{2} \int_\Omega (\dive u)^2 \,d x+ \frac{K(0)}{2} \mathrm{Per}_\Omega (\{|u| = 1\}) + \int_{J_u \cap \{ |u| = 1\}} K(u \cdot \nu) \,d \sh^1,
\end{equation}
with the cost $K$ given by \eqref{1d}. After describing in Section \ref{1dex}  some numerics that complement our rigorous work in Section \ref{oned} for the case where $\Omega$ is a rectangle, we will focus on two main settings:  (i) the case where $\Omega$ is a disk and competitors must satisfy a boundary condition in the sense of \eqref{diri} where $g$ has degree $k\in\mathbb{Z}$; and (ii) the case of an island of isotropic phase, generated by an area constraint, lying inside a nematic whose far field is given by $\vec{e}_1$. 

For both settings (i) and (ii) we will not work directly with $E_0$ but rather with a problem that at least formally can be viewed as the large $L$ limit of $E_0$, namely
\beq
E_0^{\infty}(u):=\frac{K(0)}{2} \mathrm{Per}_\Omega (\{|u| = 1\}) + \int_{J_u \cap \{ |u| = 1\}} K(u \cdot \nu) \,d \sh^1,\label{AVlimit}
\eeq
defined for $u\in (BV\cap H_\dive)(\Omega,  \mathbb{S}^1\cup\{0\})$ such that 
\beq
\dive\,u=0\quad\mbox{in}\;\Omega,\label{nodiv}
\eeq and perhaps supplemented by the condition $u_{\partial\Omega}\cdot\nu_{\partial \Omega}=g\cdot \nu_{\partial \Omega}$ on $\partial\Omega$ if one wishes to specify Dirichlet data $g:\partial\Omega\to \mathbb{S}^1\cup\{0\}$, or such that $\abs{\{u=0\}}=const$ or
$\abs{\{\abs{u}=1\}}=const$ if one wishes to specify an area constraint. We also note that the $H_\dive$ requirement still enforces the condition that competitors have trace from the nematic side that is tangent to any interface, i.e. \eqref{tanjump} where $u_-=0$.

We will construct critical points for $E_0^{\infty}$ that we expect to be local or even globally minimal and 
we observe that these divergence-free vector fields are competitors in the minimization of $E_0$ for finite $L$. Thus, we expect that they may well be close to critical points or {\it perhaps} even minimizers of $E_0$ when $L$ is large. As we shall see, this expectation is supported by simulations on the gradient flow for $E_\e$ where $L$ is large but fixed and then $\e$ is taken to be small. 

Regarding all simulations in this section, we obtain critical points for the energy $E_\e$ by simulating gradient flow for $E_\e$ using the software package COMSOL \cite{comsol}. Unless specified otherwise, we do not claim that solutions that we obtain are minimizers of $E_\e$ or prove that these solutions converge to critical points of the limiting energy. We will infer such convergence in cases where we are able to show via an analytical construction that a similar looking critical point of $E_0$ does exist. 

We consider $E_0^{\infty}$ rather than $E_0$ here in part because, as we will describe below, the divergence-free condition \eqref{nodiv} provides a rigidity that simplifies the search for critical points. We hasten to add, however, that to us minimization of $E_0^{\infty}$ is a fascinating and nontrivial problem in its own right that
one might view as a version of the Aviles-Giga limiting problem which allows for phase transitions, i.e. isotropic regions, as well as walls. Of course this entire project represents just an initial investigation of $E_\e$ and $E_0$ that we hope will generate interest in future analysis of critical points and minimization of these functionals for $L$ finite. In that vein, we hope the work in this section provides intuition and techniques that can be generalized, and that the criticality conditions derived in Section \ref{critsec} provide some tools.

So what does criticality mean for $E_0^{\infty}$? Within the nematic region where $\abs{u}=1$, but away from
the jump set $J_u$, if we locally describe a competitor $u$ via $u(x)=\big(\cos\theta(x),\sin\theta(x)\big)$, then \eqref{nodiv} implies that
\[
\nabla\theta\cdot \big(-\sin\theta,\cos\theta\big)=0.
\]
Defining the characteristic direction via $x_1'=-\sin\theta,\;x_2'=\cos\theta$ we see that $\theta$ and therefore $u$ is constant along characteristics and further that $u$ is orthogonal to characteristics and so one concludes in particular that:
\beq
\mbox{Characteristics}\;\mbox{for}\;E_0^{\infty}\;\mbox{must be straight lines along which}\; u\;\mbox{is orthogonal and constant.}
\label{AVrecipe}
\eeq
This rigidity, familiar to those who work on Aviles-Giga, is what will allow us to carry out some of the analytical constructions in this section.

On the other hand, this amount of rigidity limits one's ability build a rich class of variations of $E_0^{\infty}$ and so we will not attempt to directly compute $L=\infty$ analogues of the ODE's \eqref{wally} or \eqref{intercrit} or the junction condition \eqref{junko}.

\subsection{Critical points of $E_0$ in a rectangle.}\label{1dex}
Here we take $\Omega$ to be the rectangle $(-0.2,0.2)\times(-0.5,0.5)$  and we seek critical points of the energy $E_0$ which satisfy the boundary conditions $u\left(\cdot,\pm1/2\right)=\pm{\vec{e}}_1=(\pm1,0)$ and satisfy periodic boundary conditions on the sides $x=\pm \frac{1}{2}$.

As discussed in Section \ref{oned}, when restricting minimization of $E_\e$ to one-dimensional competitors which in this case are functions of $y$, we obtain full $\Gamma$-convergence of the one-dimensional analog  of $E_\e$ to that of $E_0$. Further, the behavior of minimizers of $E_0$ among one-dimensional competitors is determined by the value of $L$. When $L$ exceeds a certain threshold, the bulk divergence contribution vanishes and the energy of a critical point is associated solely with a wall along the $x$-axis that separates the regions of zero divergence.  When $L$ falls below the threshold value, the bulk divergence contribution is present along with a cost of the wall associated with the jump set of the minimizer. When $L$ tends to zero, the wall disappears and the energy minimizing vector field is essentially a linear interpolation of the boundary data. 
\begin{figure}
\centering
\begin{subfigure}{.5\textwidth}
  \centering
  \includegraphics[width=\linewidth]{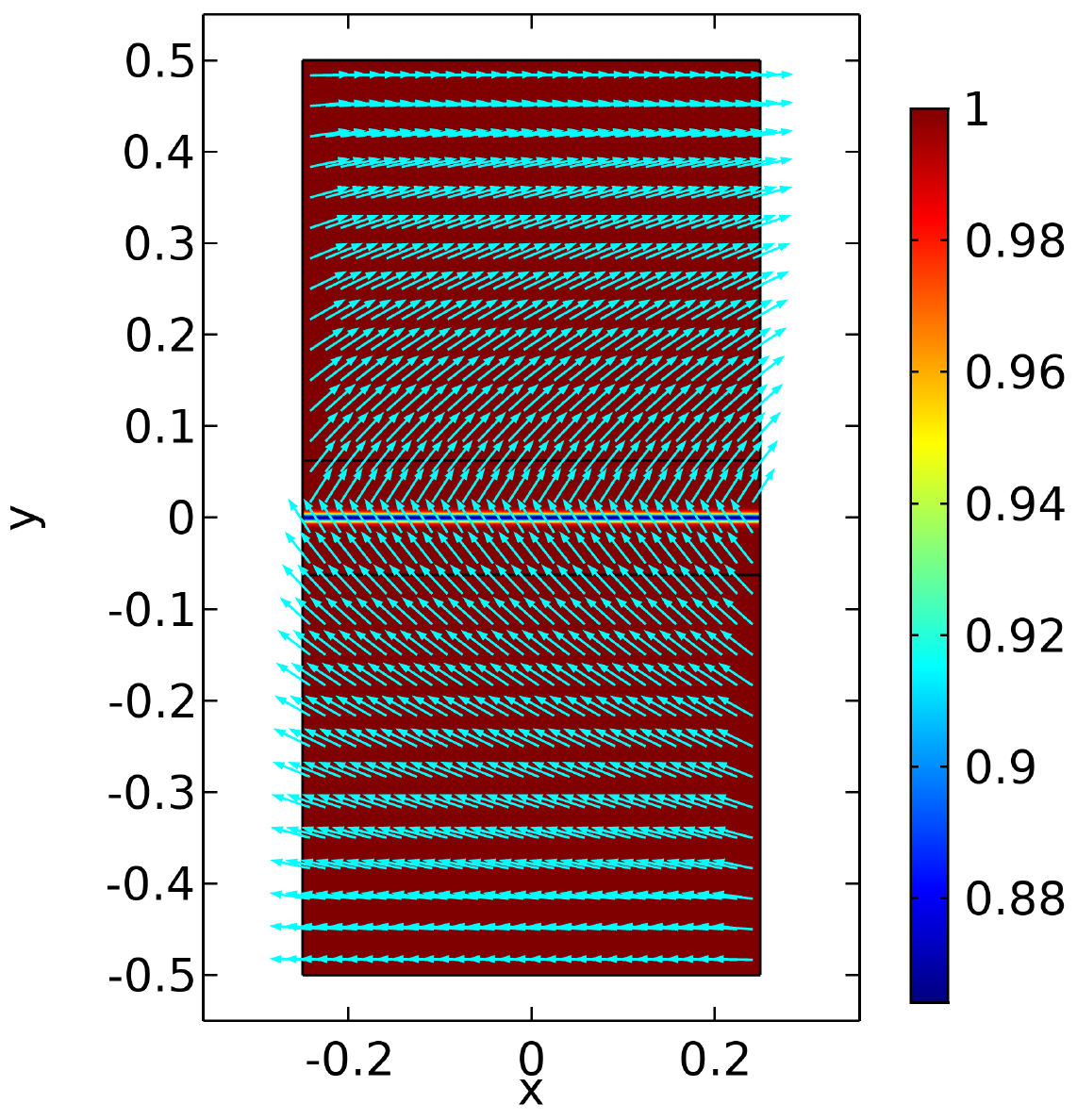}
  \caption{$L=.4$ }
\end{subfigure}%
\begin{subfigure}{.5\textwidth}
  \centering
  \includegraphics[width=\linewidth]{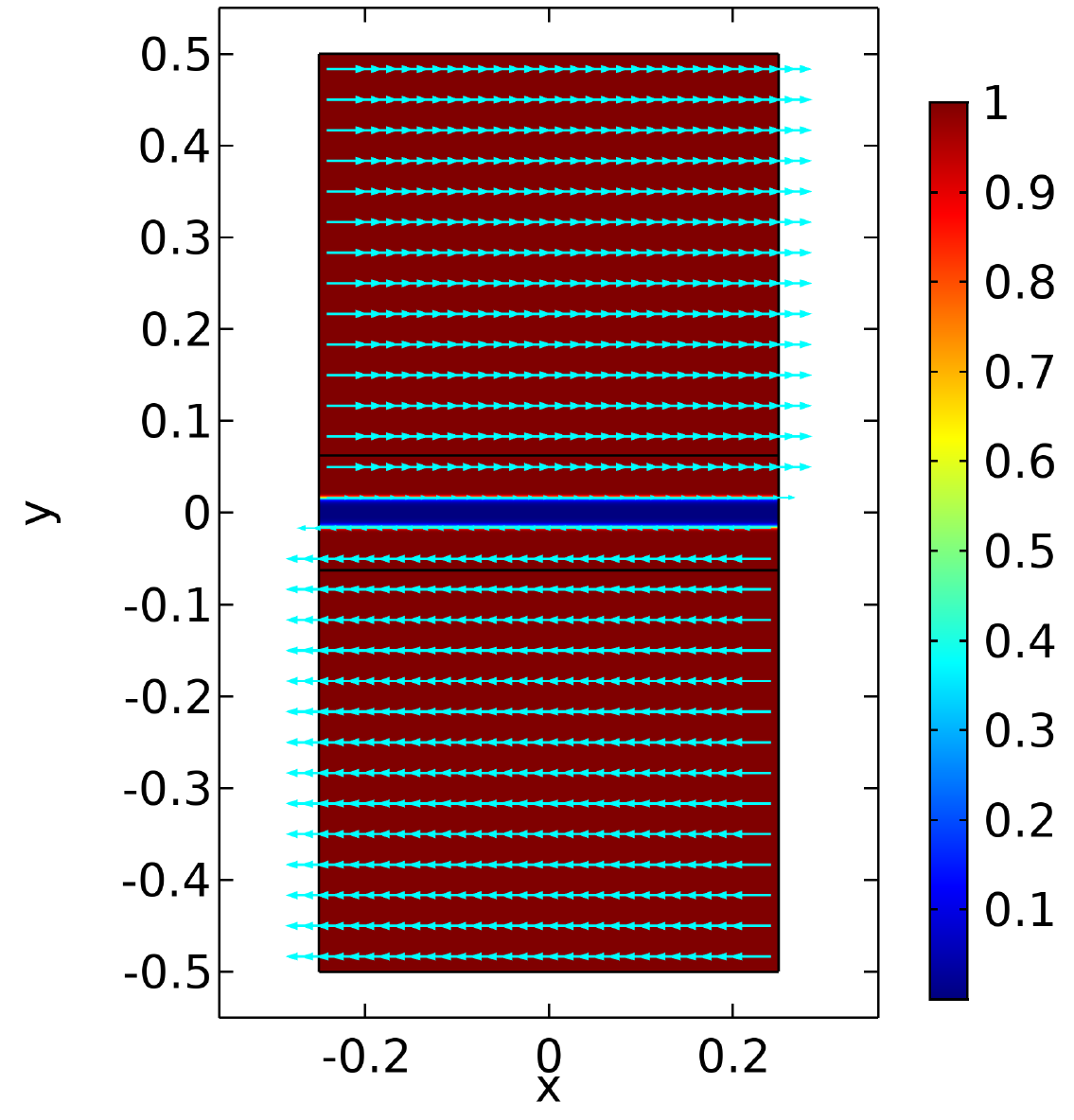}
  \caption{$L=.5$}
\end{subfigure}
\caption{Critical points of $E_\e$ in the rectangle. Here $\varepsilon=0.001$.}
\label{fig:1d.1}
\end{figure}

Figs.~\ref{fig:1d.1}-\ref{fig:1d.3} present the results of simulations for the gradient flow for $E_\e$ in the rectangle. It is evident that, even though the simulations are fully two-dimensional, the critical points obtained in this way are one-dimensional and conform to the picture described in the previous paragraph. Two main observations follow from these figures. First, the results seem to indicate that the wall cost is indeed one-dimensional as we conjectured earlier in the paper. Second, in all simulations done in the rectangle, the critical points we observe are always one-dimensional, even for large values of $L$. This is in contrast to the results in \cite{GSV} for the version of the problem with the Ginzburg-Landau instead of the Chern-Simons-Higgs potential. In that work, one-dimensional critical points are found to be unstable with respect to formation of cross-tie configurations for large $L$---such instability does not seem to be present here, at least numerically.
\begin{figure}[H]
    \centering
    \includegraphics[width=2.5in]{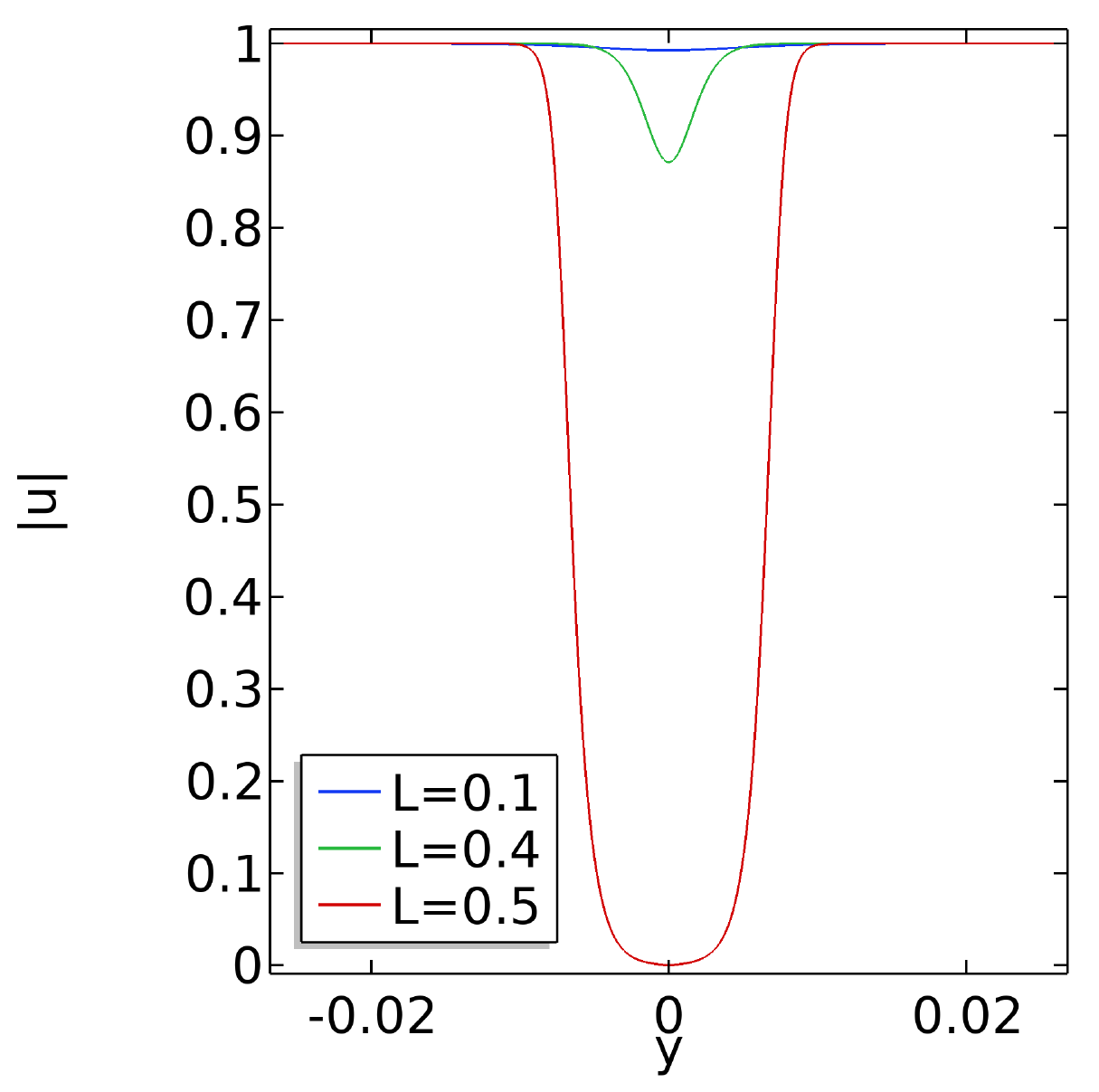}
    \caption{Cross-section of the wall for a critical point of $E_\e$ in the rectangle. The $y$-axis is as shown in Fig.~\ref{fig:1d.1} and $\varepsilon=0.001$. When $L\geq0.5$, the profile is independent of $L$ (not shown).}
    \label{fig:1d.3}
\end{figure}

\subsection{Degrees other than $0$ or $1$ are too costly}\label{divraghav}

Before we begin the constructions and numerics pertaining to $E_0^{\infty}$, we first present
 a theorem which will elucidate the behavior of certain critical points for $E_0$ and provide an explanation for some of the morphology to come. The theorem yields a lower bound for the $L^2$-norm of the divergence, in the spirit of analogous lower bound results of Jerrard \cite{jerrard} and Sandier \cite{sandier} for the Ginzburg-Landau energy. The proofs of the Jerrard/Sandier results rely crucially on the fact that the square of the gradient of a function is a sum of squares of its components, a feature that is not shared by the square of the divergence of a vector field. We overcome this difficulty by working in Fourier space. 

\bthm\label{thmdiv}
Fix $0 < \rho < \rho^\prime \leqslant 1$, set $A:= \{ x \in \R^2: \rho < |x| < \rho^\prime\}$ and let $C_t$ be a circle of radius $t$ centered at the origin. Suppose that  $u \in  C^1(\overline{A};\R^2)$ is such that $|u| \geqslant 1/2$ on $A$ and $\mathrm{deg}\, (u, C_t) = d \neq 0,1$ for any $t \in [\rho,\rho^\prime]$.   Then  
\begin{align}
&\int_A (\dive u)^2 \,dx \geqslant |\pi d \log(\rho^\prime/\rho) + 4|, & d < 0, \label{eq1.1.1}\\
&\int_A (\dive u)^2 \,dx \geqslant |\pi (d-1) \log(\rho^\prime/\rho) - 4|, & d > 1.  \label{eq1.1.2}
\end{align}
\ethm
\brk By majorizing $\int_A (\dive u)^2$ by $\int_A |\nabla u|^2$ in \eqref{eq1.1.1}-\eqref{eq1.1.2}, it follows from results for $\abs{\nabla u}^2$ that the scaling in $\rho^\prime/\rho$ is optimal. Note that there is no similar lower bound when $d=1$ due to existence of the divergence-free vector field $\vec{e}_\theta.$ 
\erk
\begin{proof}
The proof of this result proceeds using Fourier series.\par
1. Developing $u$ in a Fourier series, given by 
\begin{align*}u \sim \sum_{n \in \mathbb{Z}} u_n(r) e^{i n \theta},\end{align*}
we first derive a formula for the degree of $u$ in terms of its Fourier coefficients. Denoting by $u^t$ the restriction of $u$ to $C_t,$ and writing $u_n = f_n + i g_n,$  we compute
\begin{align} \notag
d := \deg(u^t, C_t) &= \frac{1}{2\pi} \int_{C_t} u^t \times u^t_\tau \,d \sh^1 \\ \notag
&= \frac{1}{2\pi} \int_{C_t} \sum_n n \begin{pmatrix}
f_n \cos n\theta - g_n \sin n\theta \\
f_n \sin n \theta + g_n \cos n \theta
\end{pmatrix} \times \begin{pmatrix}
-f_n \sin n\theta - g_n \cos n\theta \\
f_n \cos n \theta - g_n \sin n \theta
\end{pmatrix}\\
&= \sum_{n \in \mathbb{Z}} n \left( |f_n(t)|^2 + |g_n(t)|^2\right)  \label{degform} \\
&= \sum_{n \in \mathbb{Z}} n|u_n(t)|^2 \label{degform2},
\end{align}
where in the last line we have used orthogonality. \\

2. As in the proof of Thm. 5.1 in \cite{GSV}, we find 
\begin{align*}
\dive u = \sum_{n \in \mathbb{Z} } \dive V_n, 
\end{align*}
in $L^2,$ where we have 
\begin{align*}
\dive V_1 &= \left(f_1^\prime(r) + \frac{f_1(r)}{r} \right), \\
\dive V_n &= \left( f_n^\prime(r) + \frac{n f_n(r)}{r}\right) \cos(n-1)\theta - \left( g_n^\prime(r) + \frac{n g_n(r)}{r}\right)\sin(n-1)\theta.  \hspace{1cm} n \neq 1.\end{align*}
It follows that 
\begin{align*}
\frac{1}{\pi}\int_A (\dive u)^2 &= 2 \int_\rho^{\rho^\prime} \left( f_1^\prime + \frac{f_1}{r}\right)^2 r\,dr + \sum_{n \neq 1} \int_\rho^{\rho^\prime}\left( \left( f_n^\prime + \frac{n f_n(r)}{r}\right)^2 + \left( g_n^\prime + \frac{ng_n(r)}{r} \right)^2 \right) r\,dr \\
&\geqslant  \int_\rho^{\rho^\prime} \left[2\frac{f_1^2}{r}  + \sum_{n \neq 1 } \frac{n^2(f_n(r)^2 + g_n(r)^2)}{r} \right]  \,dr \\
&\quad + \int_\rho^{\rho^\prime} \left[4 f_1(r) f_1^\prime(r) + \sum_{n \in \mathbb{Z}, n \neq 1} 2 n \left(f_n(r) f_n^\prime(r) + g_n(r) g_n^\prime(r) \right) \right]\,dr \\
&:= I + II. 
\end{align*}
We estimate the integrals $I$ and $II$ separately as follows, beginning with $II.$ From Eqn. \eqref{degform} and the assumption that $\deg(u,C_t)=d$ for each $t\in [\rho,\rho']$, we obtain that for each $r$ 
\begin{align}\notag
II &= \int_\rho^{\rho'}\f{\partial}{\partial r}\left[ 2f_1^2 + \sum_{n\in \mathbb{Z},n\neq 1}n(f_n^2+g_n^2)   \right]\, dr \\
&= \int_\rho^{\rho'}\f{\partial}{\partial r}\left[ f_1^2-g_1^2 + \sum_{n\in \mathbb{Z}}n(f_n^2+g_n^2)   \right]\, dr\notag \\
&=\int_\rho^{\rho'}\f{\partial}{\partial r}\left[ f_1^2-g_1^2 + d  \right]\, dr\notag \\
&=f_1(\rho')^2-f_1(\rho)^2+g_1(\rho)^2 - g_1(\rho')^2.
\end{align}
Using now the definition of $f_1, g_1,$ and the fact that $|u| \leqslant 1,$ we find that $\|f_1\|_\infty, \|g_1\|_\infty \leqslant 1.$ It follows that 
\begin{align}
|II| \leqslant 4. 
\end{align}
We next turn to estimating $I.$ Let us first suppose $d> 1.$ We have 
\begin{align} \notag
I &\geqslant \int_\rho^{\rho^\prime} \left( \frac{f_1(r)^2}{r} + \sum_{n \in \mathbb{Z}, n \neq 1} \frac{n^2(f_n(r)^2 + g_n(r)^2)}{r} \right) \,dr \\ \notag
&= \int_\rho^{\rho^\prime}\left(  \frac{f_1(r)^2}{r} + \sum_{n \in \mathbb{Z}, n \neq 1} \frac{(n^2 -n)(f_n(r)^2 + g_n(r)^2)}{r} + \frac{n(f_n(r)^2 + g_n(r)^2)}{r}\right) \,dr \\ \label{0.10}
&\geqslant \int_\rho^{\rho^\prime}\left( \frac{f_1(r)^2}{r}+ \sum_{n \in \mathbb{Z}, n \neq 1} \frac{n(f_n(r)^2 + g_n(r)^2)}{r}\right) \,dr \\ \label{0.11}
&= \int_\rho^{\rho^\prime} \frac{d - g_1^2(r)}{r} \,dr \\ \label{0.12}
&\geqslant \int_\rho^{\rho^\prime} \frac{d-1}{r}\,dr\\
&= (d-1) \log(\rho^\prime/\rho). \label{g1}
\end{align}
In going from \eqref{0.10} to \eqref{0.11} we have used \eqref{g1} while in going from \eqref{0.11} to \eqref{0.12} we have used that $g_1^2 \leqslant 1.$ 
This completes the proof of the theorem when $d > 1,$ so we turn our attention to when $d < 0.$ In this case, we have 
\begin{align*}
I &\geqslant \int_\rho^{\rho^\prime}\left(  \frac{f_1(r)^2}{r} + \sum_{n \in \mathbb{Z}, n \neq 1} \frac{n^2(f_n(r)^2 + g_n(r)^2)}{r}\right) \,dr \\ \notag
&= \int_\rho^{\rho^\prime}\left(  \frac{f_1(r)^2}{r} + \sum_{n \in \mathbb{Z}, n \neq 1} \frac{(n^2 +n)(f_n(r)^2 + g_n(r)^2)}{r} - \frac{n(f_n(r)^2 + g_n(r)^2)}{r}\right) \,dr \\ 
&\geqslant \int_\rho^{\rho^\prime}\sum_{n \in \mathbb{Z}} - n \frac{f_n^2(r) + g_n^2(r)}{r}\,dr \\
&\geqslant -d \int_\rho^{\rho^\prime} \frac{1}{r}\,dr \\
&= -d \log (\rho^\prime/\rho) = |d|\log (\rho^\prime/\rho). 
\end{align*}
\end{proof}
It also is worth mentioning that among degree 1 singularities, the $L^2$-norm of the divergence can vary greatly and may or may not satisfy a lower bound of the type in the previous theorem. For example, for a Ginzburg-Landau vortex $\frac{x}{|x|}$, the $L^2$-norm of the divergence taken over an annulus centered at the origin blows up logarithmically as the inner radius approaches 0. However, an $\vec{e}_\theta$ vortex, given by $\frac{x^{\perp}}{|x|}$, is divergence free. This observation is relevant to our model, especially at corner-type defects on the phase boundary. In many of our examples, the director $u$, which must be tangent to the phase boundary, switches the sense of tangency at a corner. If such a switch occurs at a corner of the phase boundary in the interior of the domain, then walls must intersect the defect in order to avoid infinite energy from the bulk divergence term; see Fig. \ref{nowalls}. Conversely, if $u$ \textit{does not} change its sense of tangency at a corner on the interface, then the singularity can be locally resolved by the formation of a partial $\vec{e}_\theta$ vortex in which an infinite family of characteristics emanate from the defect.
\begin{figure}
\centering
\begin{subfigure}{.45\textwidth}
  \centering
  \includegraphics[scale=.8]{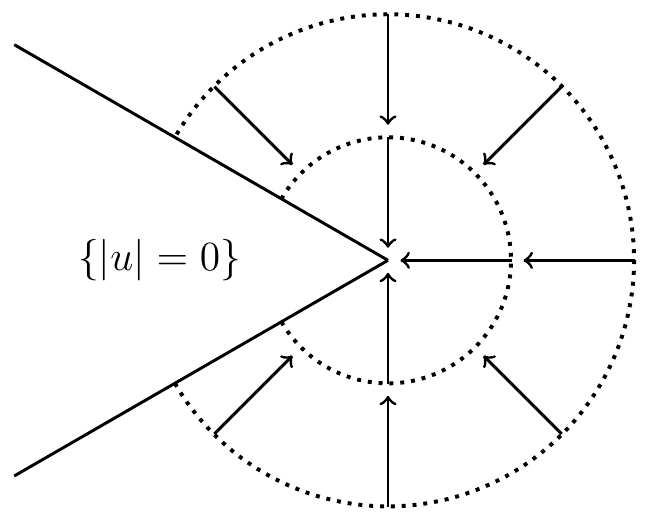}
\end{subfigure}%
\begin{subfigure}{.55\textwidth}
  \centering
  \includegraphics[scale=.8]{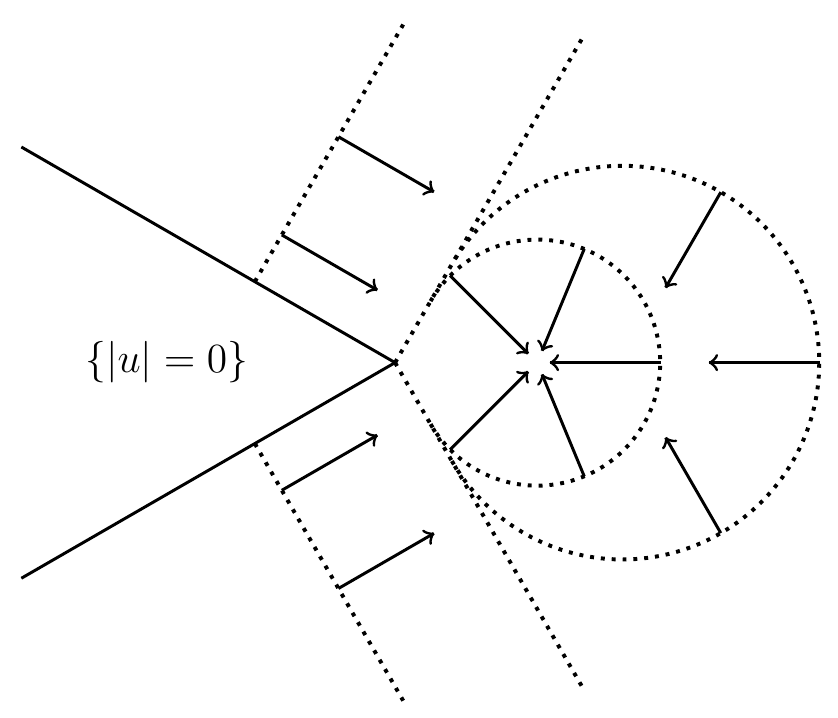}
\end{subfigure}
\caption{Corner on an interface at which $u$ changes tangency: If there are no walls, a Ginzburg-Landau vortex forms in one of two ways, resulting in infinite $E_0$ energy. The dotted lines represent characteristics, and the arrows represent the $\mathbb{S}^1$-valued director, which is perpendicular to the characteristics. 
}
\label{nowalls}
\end{figure}

\subsection{Critical points of $E_0$ and $E_0^{\infty}$ in a disk.}\label{diskex}
In this section we consider critical points of the energy $E_0^{\infty}$ in the disk $\Omega$ of radius $R$ among competitors satisfying  the boundary condition
\begin{equation}
u\big(R\cos \frac{s}{R},R\sin \frac{s}{R}\big)\cdot\nu_{\partial\Omega} = \big(\cos(ks+\alpha),\sin(ks+\alpha)\big)\cdot\nu_{\partial\Omega}\quad \textup{for }s\in\left[0,2\pi R\right),\label{sec4bc}
\end{equation}
where $k\in\mathbb Z$, $\alpha\in\mathbb R$, and the boundary is parametrized with respect to arc-length. 

 The simplest cases to consider are $(k,\alpha)=(1,\pi/2)$ and when $k=0$  for which minimizers of $E_0$ are the divergence-free vortex $$u_0=\vec{e}_\theta=\left(\frac{-y}{\sqrt{x^2+y^2}},\frac{x}{\sqrt{x^2+y^2}}\right),$$ and the constant state $$u_0=\big(\cos\alpha,\sin\alpha\big),$$ respectively. Indeed, trivially, in both cases $E_0(u_0)=\min{E_0}=0$. Hence our principal interest in this section will be to understand the behavior of critical points for other choices of $(k,\alpha)$.

We begin by considering the case where $k$ is a negative integer and $\alpha=\pi$. To gain some insight into how these boundary conditions influence the morphology of interfaces and walls, we present in Figure \ref{fig:deg1_neg} the large-time asymptotics for gradient flow dynamics for the energy $E_\e$ with boundary conditions $u|_{\partial\Omega}=-\big(\cos ks, \sin ks\big)$ for two values of $L$. Then in Figure \ref{fig:deg2_neg} we present simulations for data with degrees $-2$ and $-3$. Although we do {\it not} impose an area constraint in these simulations in order to induce a phase transition, these numerics nonetheless indicate a substantial presence of the isotropic phase in the form of an island with $2\abs{k}+2$ boundary singularities. Generally speaking, these islands appear to grow in size as $\abs{k}$ grows, and for $k<-1$, both configurations with a single or multiple vortices are possible. Studies on vortices using the Ginzburg-Landau potential such as \cite{BBH} or---more appropriately to this study---the Chern-Simons-Higgs potential with $L = 0$ in the elastic energy \cite{Spirn-Kurzke} tempt one to think of these islands for $\e > 0$ as ``defects'' arising from the negative degree boundary condition. However, the numerics and Theorem \ref{thmdiv} indicate that the cores of the defects do not shrink in the $\e \to 0$ limit. Indeed, from Theorem \ref{thmdiv}  it follows that a defect with a negative degree must either be inside an isotropic region or have walls originating from the defect. The latter situation was, in fact, observed in \cite{GSV} for the degree $-1$ defects while the Ginzburg-Landau potential considered in \cite{GSV} did not allow for presence of interfaces.\par
\begin{figure}
\centering
\begin{subfigure}{.4\textwidth}
  \centering
  \includegraphics[width=\linewidth]{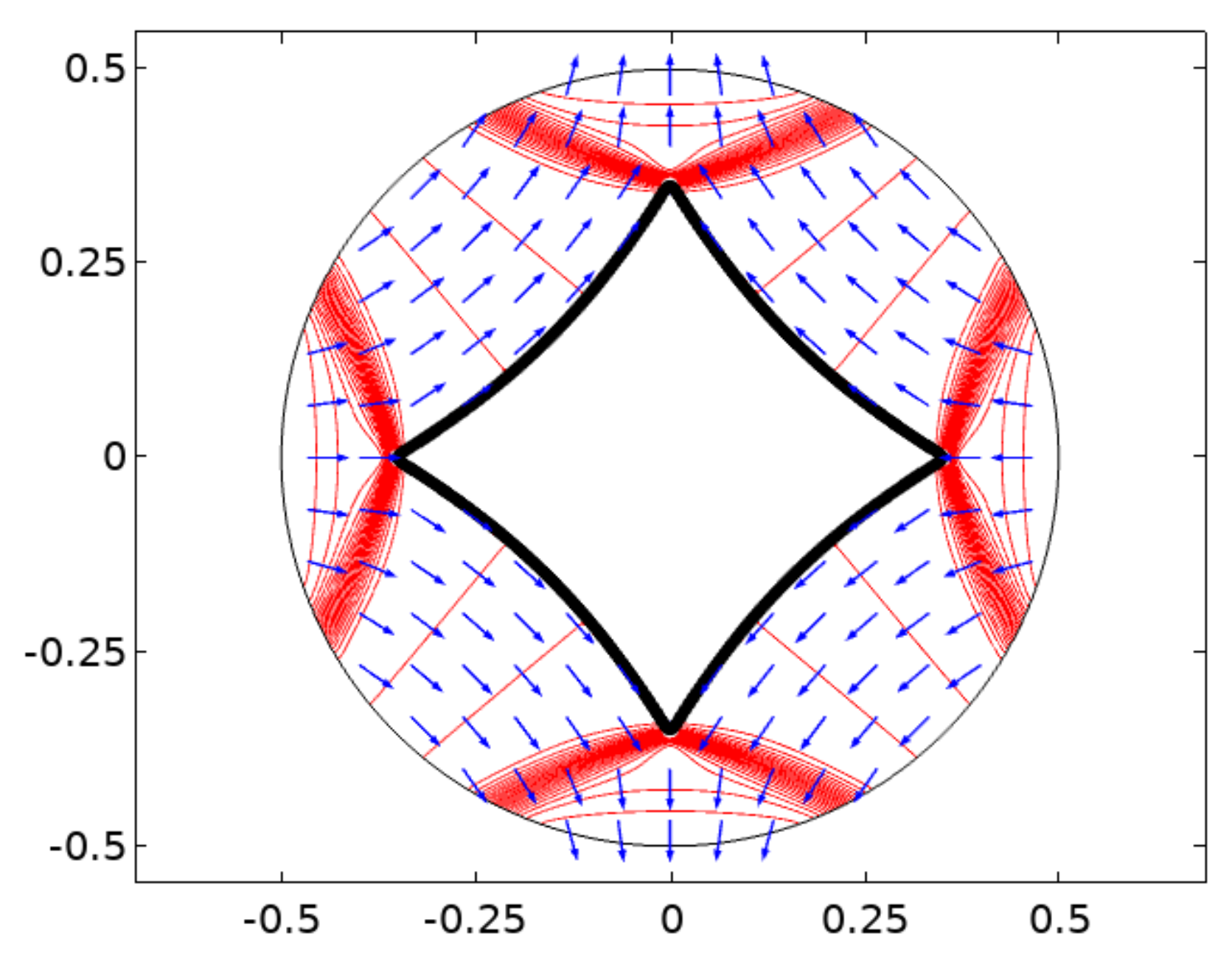}
  \caption{$L=0.4$, $\e=0.005$}
  \label{fig:tactoid1}
\end{subfigure}%
\begin{subfigure}{.4\textwidth}
  \centering
  \includegraphics[width=\linewidth]{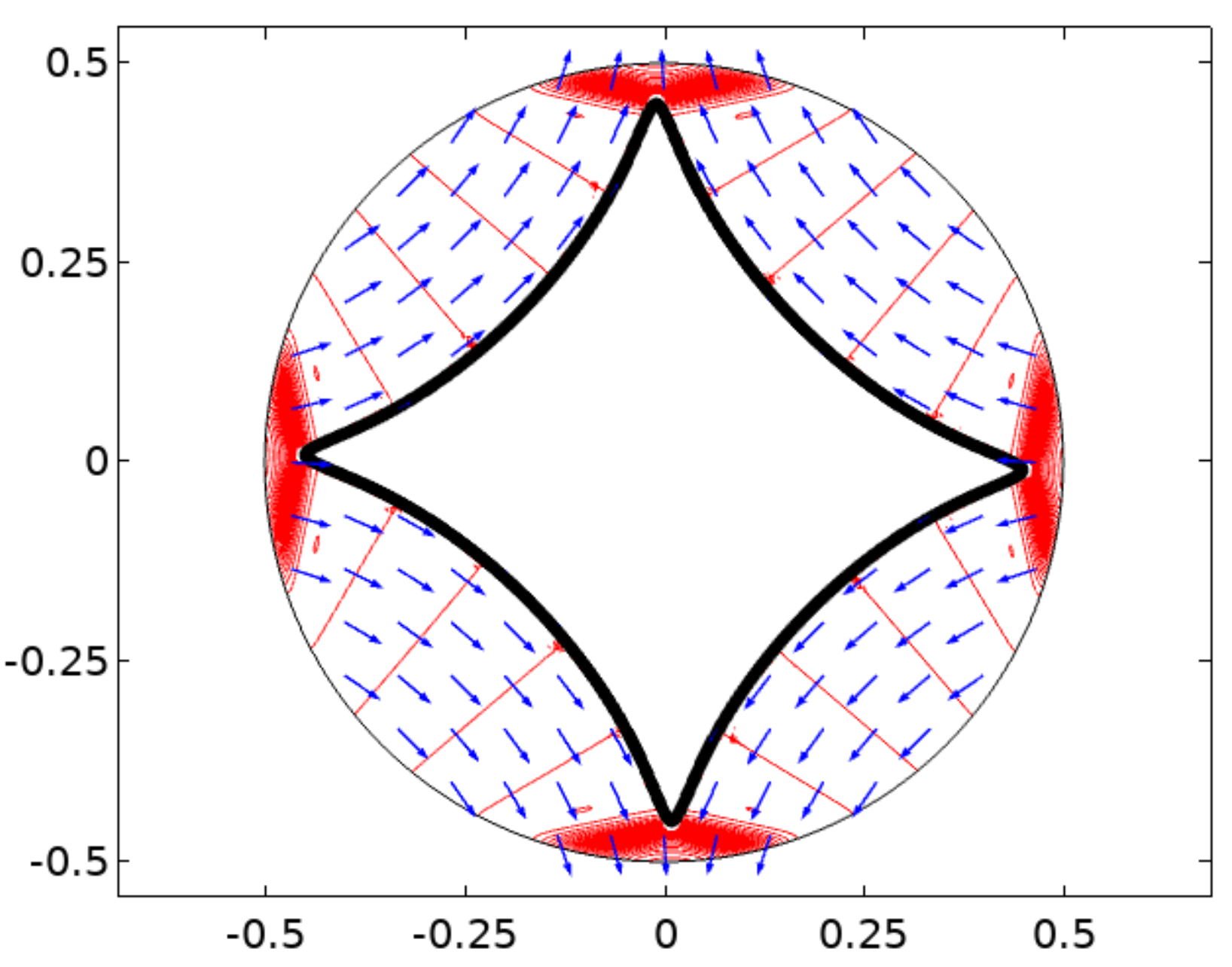}
  \caption{$L=2$, $\e=0.005$}
  \label{fig:tactoid2}
\end{subfigure}
\caption{Critical points of $E_\e$ for $(k,\alpha)=(-1,\pi)$. The red curves represent level sets of $\mathrm{div}\,u$ while $|u|=0.5$ on the black curves that enclose the isotropic phase. }
\label{fig:deg1_neg}
\end{figure}

\begin{figure}
\centering
\begin{subfigure}{.4\textwidth}
  \centering
  \includegraphics[width=\linewidth]{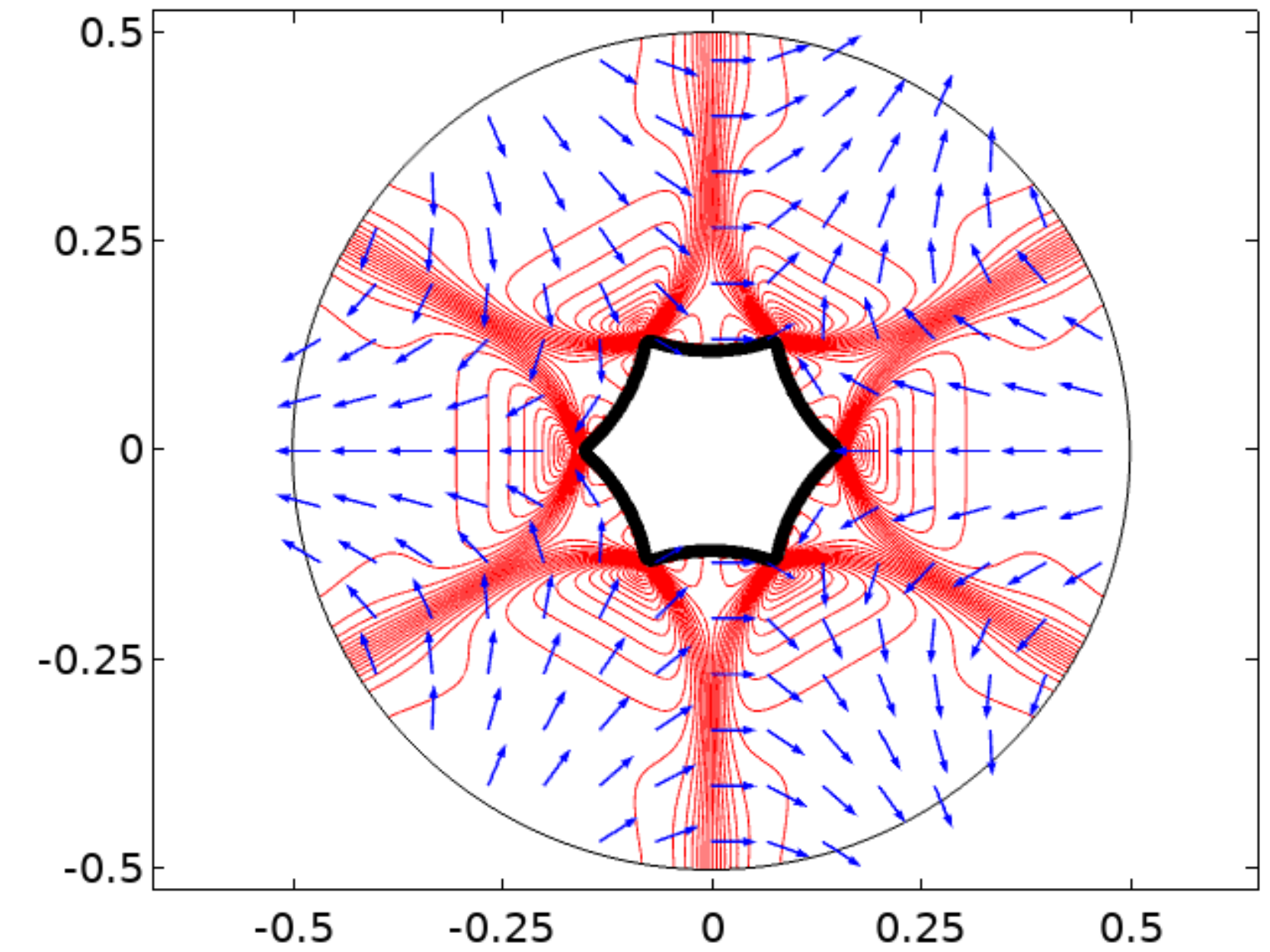}
  \caption{$L=0.04$, $\e=0.005$}
  \label{fig:tactoid3}
\end{subfigure}%
\begin{subfigure}{.4\textwidth}
  \centering
  \includegraphics[width=\linewidth]{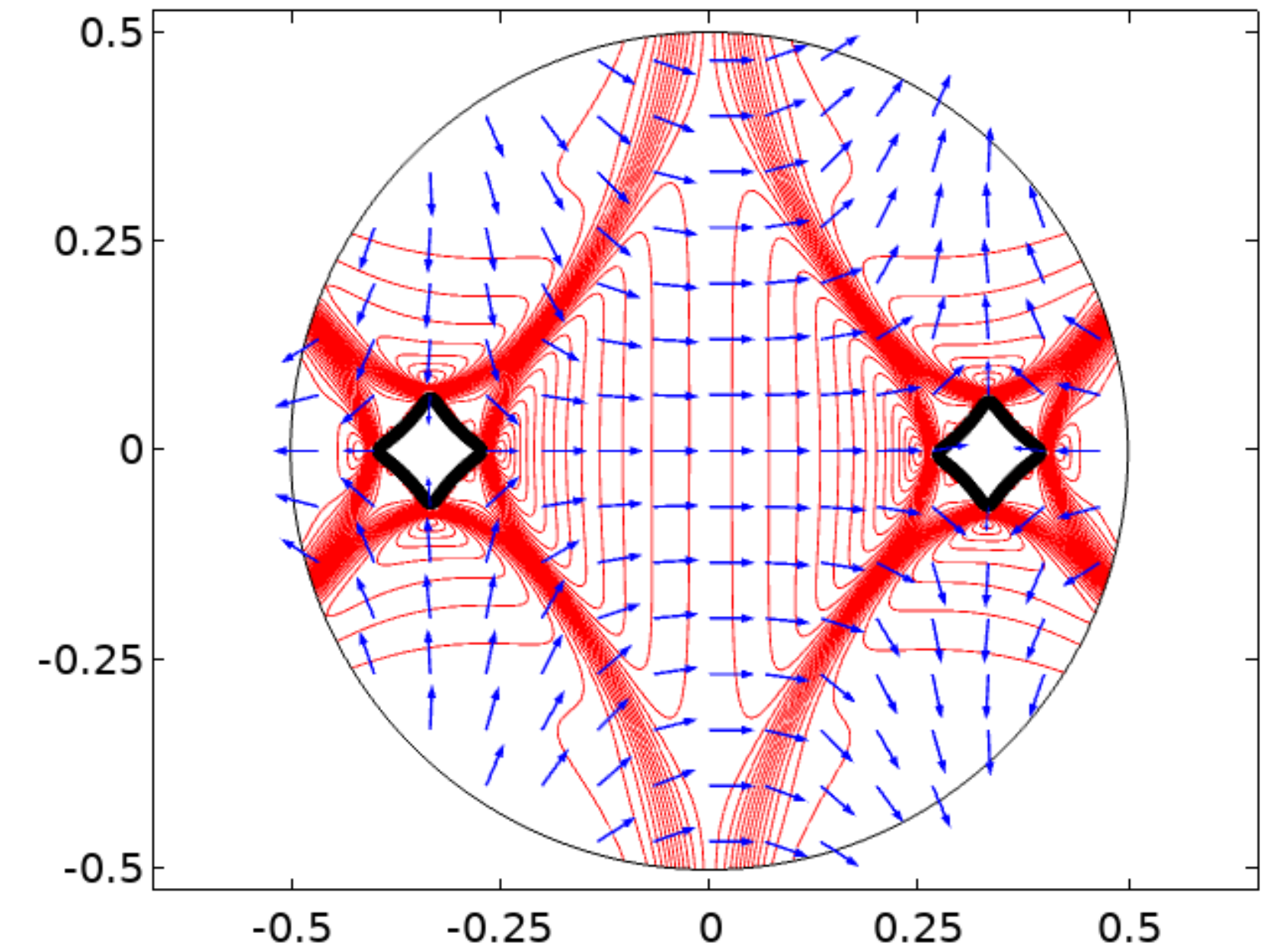}
  \caption{$L=0.1$, $\e=0.005$}
  \label{fig:tactoid4}
\end{subfigure}
\begin{subfigure}{.4\textwidth}
  \centering
  \includegraphics[width=1.05\linewidth]{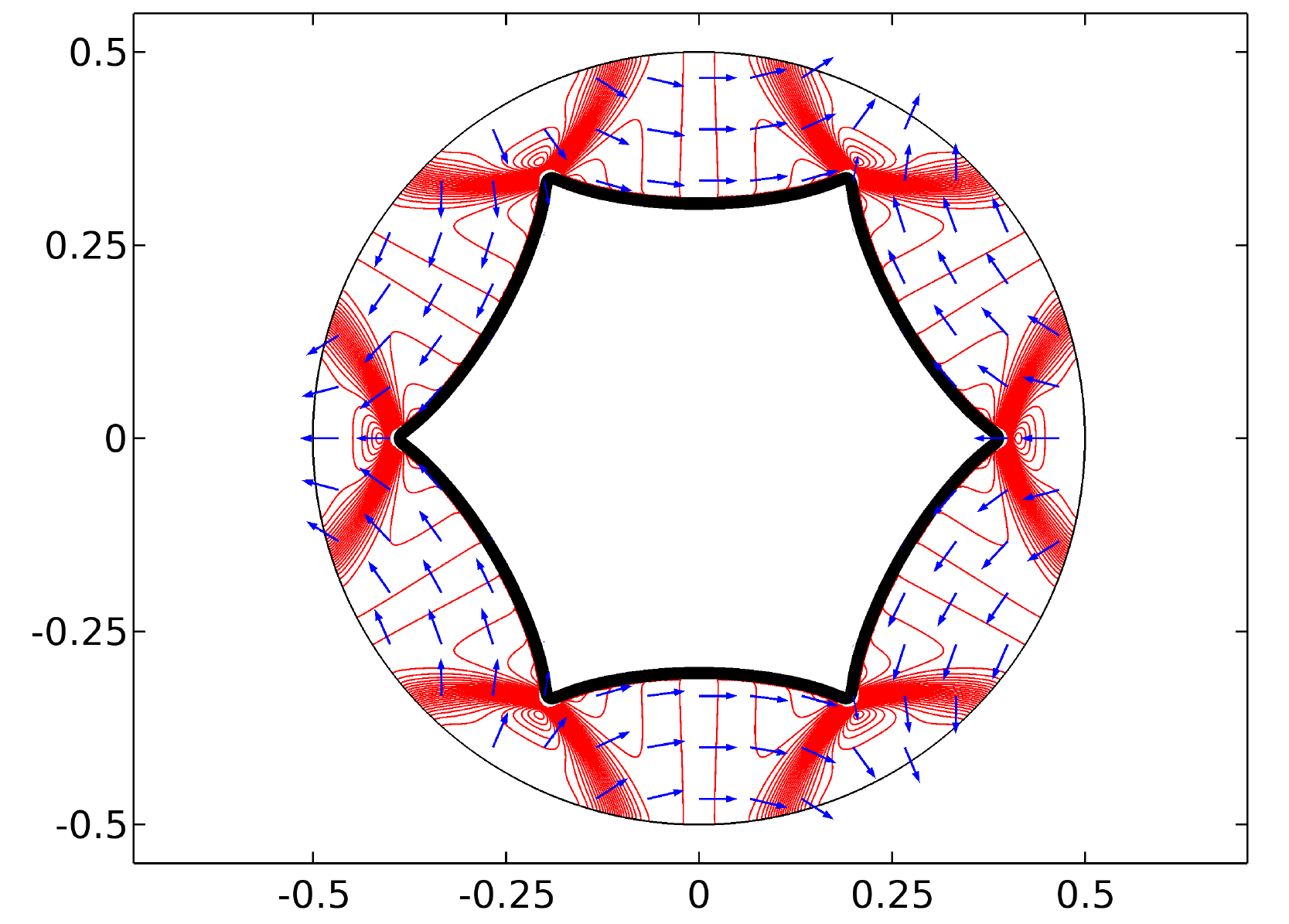}
  \caption{$L=0.1$, $\e=0.005$}
  \label{fig:tactoid5}
\end{subfigure}
\begin{subfigure}{.4\textwidth}
  \centering
  \includegraphics[width=\linewidth]{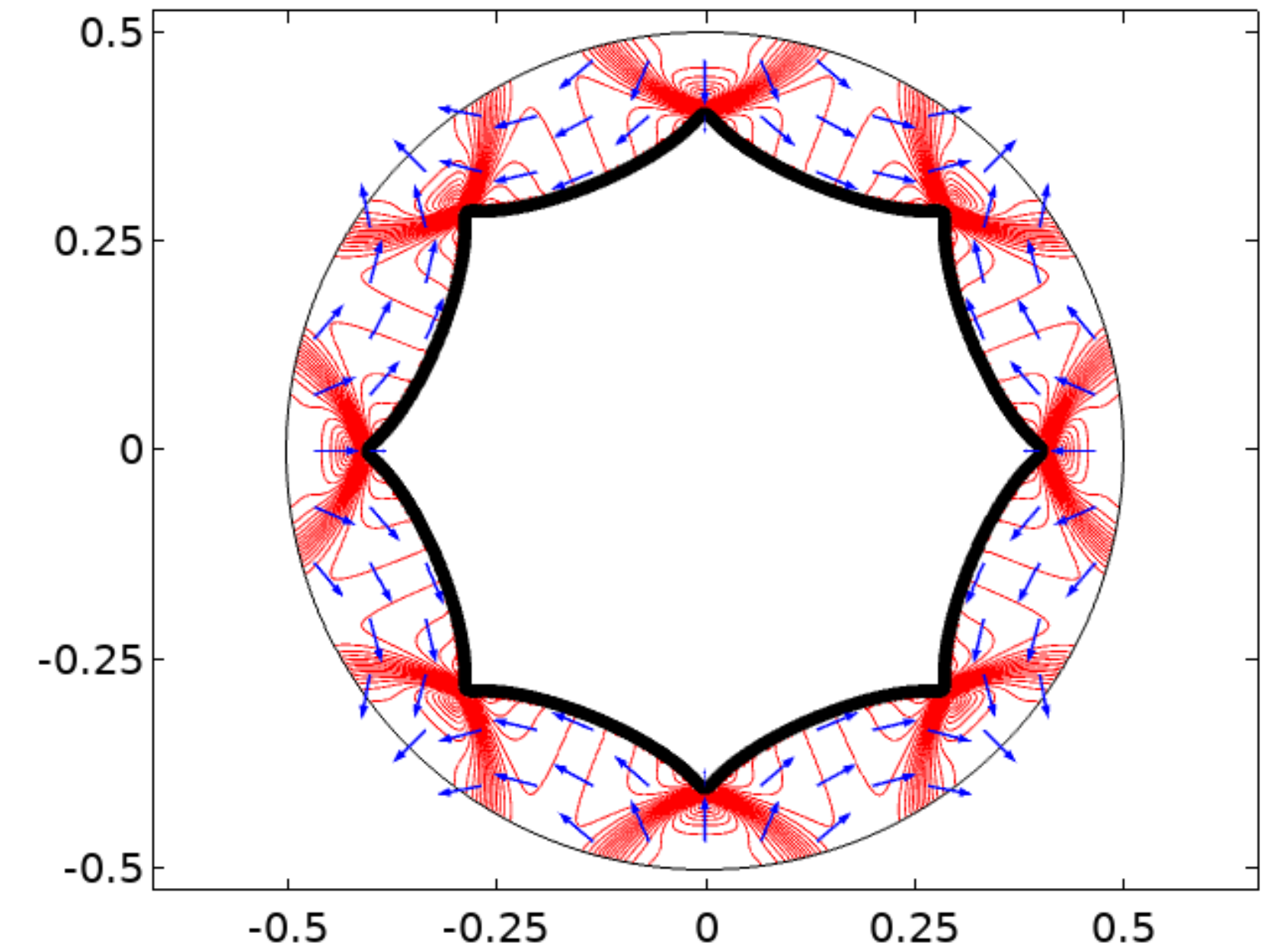}
  \caption{$L=0.04$, $\e=0.005$}
  \label{fig:tactoid6}
\end{subfigure}
\caption{Critical points of $E_\e$ for (a-c): $(k,\alpha)=(-2,\pi)$  and (d): $(k,\alpha)=(-3,\pi)$. The red curves represent level sets of $\mathrm{div}\,u$ while $|u|=0.5$ on the black curves that enclose the isotropic phase. The configurations (b) and (c) are obtained starting from different initial conditions. Configuration (b) has slightly lower energy in $E_\e$.}
\label{fig:deg2_neg}
\end{figure}

We now provide some analytical evidence that supports the observations  in Figs.~\ref{fig:deg1_neg}-\ref{fig:deg2_neg}. Motivated by the gradient flow simulations, we  construct critical points for $E_0^{\infty}$ and so divergence-free competitors for $E_0$. These constructions will have only interface, but no walls, with singular points of the interface always touching the boundary of the disk, though of course the numerics suggest that for $L$ finite, there should exist walls branching off the phase boundary singularities and attaching to $\partial\Omega$. 

\beg\upshape \label{astroid}\normalfont
In this example, $\Omega=B(0,1)$, and we are interested in competitors which exhibit the symmetry
\begin{equation}\label{kfold}
u\left(e^{\pi  i /(k+1)}x \right)=e^{-\pi k i/(k+1)}u(x)
\quad\mbox{for}\;k\in\mathbb{N}.\end{equation}
This is the symmetry exhibited by the configurations in Figs.~\ref{fig:deg1_neg}-\ref{fig:deg2_neg}. The construction will proceed by issuing characteristics off $\partial \Omega$ and by adhering to the condition \eqref{AVrecipe}.  

Owing to the condition \eqref{kfold}, we construct a critical point of $E_0^{\infty}$ in the sector $S$ corresponding to $[0,\pi/(k+1)]$ and then extend  the construction to the rest of the domain by symmetry. Shifting out of complex notation and parametrizing $\partial \Omega \cap \partial S$ by $(\cos s, \sin s)$ for $s \in [0,\pi/(k+1)]$, we will insist that $u_{|_{\partial \Omega}} = (-\cos ks, \sin ks),$ rather than just having agreement between the normal component of $u$ and that of the data. 

Integrating the characteristic equations then yields 
\begin{align} \label{chartactoid}
    &\left(x_1(s,t), x_2(s,t) \right) = (\cos s , \sin s) - t (\sin ks , \cos ks)\\
    &u\left( x_1(s,t), x_2(s,t) \right) = (- \cos ks, \sin ks),
\end{align}
with $s \in [0,\pi/(k+1)]$, $t \geqslant 0. $ We represent the interface in the form 
\[(p(s), q(s)) := \left( x_1(s,t(s)), x_2(s,t(s)) \right)
\]for an appropriate arrival time $t(s) \geqslant 0$. Here, for each $s$, a characteristic arrives at the interface at the time $t(s)$ and we require that $u$ at the point of arrival is tangent to the interface, that is
\begin{align*}
    \left( p^\prime(s), q^\prime(s) \right) = \alpha\, (u(p(s)), u(q(s))) \, \mbox{ for some } \alpha \in \R.
\end{align*}
Using this expression in \eqref{chartactoid}, we find that 
\begin{align*}
    x_1^\prime(s,t(s)) = -\sin s - t^\prime(s) \sin ks - kt(s) \cos ks &= - \alpha \cos ks, \\
    x_2^\prime(s,t(s)) = \cos s - t^\prime(s) \cos ks +k t(s) \sin ks &= \alpha \sin ks. 
\end{align*}
Upon rearrangement, we have
\begin{align*}
    -\sin s +\alpha \cos ks  - k t(s) \cos ks &= t^\prime(s) \sin ks, \\
     \cos s -\alpha \sin ks  + k t(s) \sin ks &= t^\prime(s) \cos ks. 
\end{align*}
Multiplying these equations by $\sin ks$ and $\cos ks$, respectively, and adding the results gives
\begin{align*}
   t^\prime =  - \sin s \sin ks + \cos s \cos ks= \cos (k+1)s. 
\end{align*}
Integration then yields 
\begin{align*}
    t(s) = \frac{1}{k+1}\sin (k+1)s + c.
\end{align*}
Motivated by numerics, we seek an interface that meets $\partial \Omega$ at $(1,0),$ so that $t(0) = 0.$ Then $c = 0$, so that $t(s) = \frac{1}{k+1}\sin (k+1)s.$
The parametric equation of the interface in the sector $S$ is now given by
\begin{align} \label{interfacep}
    p(s) &= \cos s - \frac{1}{k+1} \sin(k+1)s \sin ks = \left(1 - \frac{1}{2(k+1)}  \right) \cos s + \frac{1}{2(k+1)}\cos(2k+1)s, \\
    \label{interfaceq} q(s) &= \sin s - \frac{1}{k+1} \sin(k+1)s \cos ks=\left(1 - \frac{1}{2(k+1)}  \right) \sin s - \frac{1}{2(k+1)}\sin(2k+1)s . 
\end{align}
Extending the interface to all of $\Omega$ via the symmetry condition \eqref{kfold}, we obtain a closed curve with $2(k+1)$ evenly-spaced cusps. When $k = 1$ one checks that
$p(s)=\cos^3s$ and $q(s)=\sin^3s$ and the interface satisfies the equation $x_1^{2/3} + x_2^{2/3} = 1.$ In Fig. \ref{fig:astroidinfty} we compare the graph of this curve with the contour line $|u|=0.5$ for the critical point obtained numerically via gradient flow when $L=2$. It is clear that the two curves are very close to each other, which is quite striking since one might only expect a strong connection between the critical points of $E_\e$ and $E_0^{\infty}$ for $L$ large.
\begin{figure}
    \centering
    \includegraphics[scale=.4]{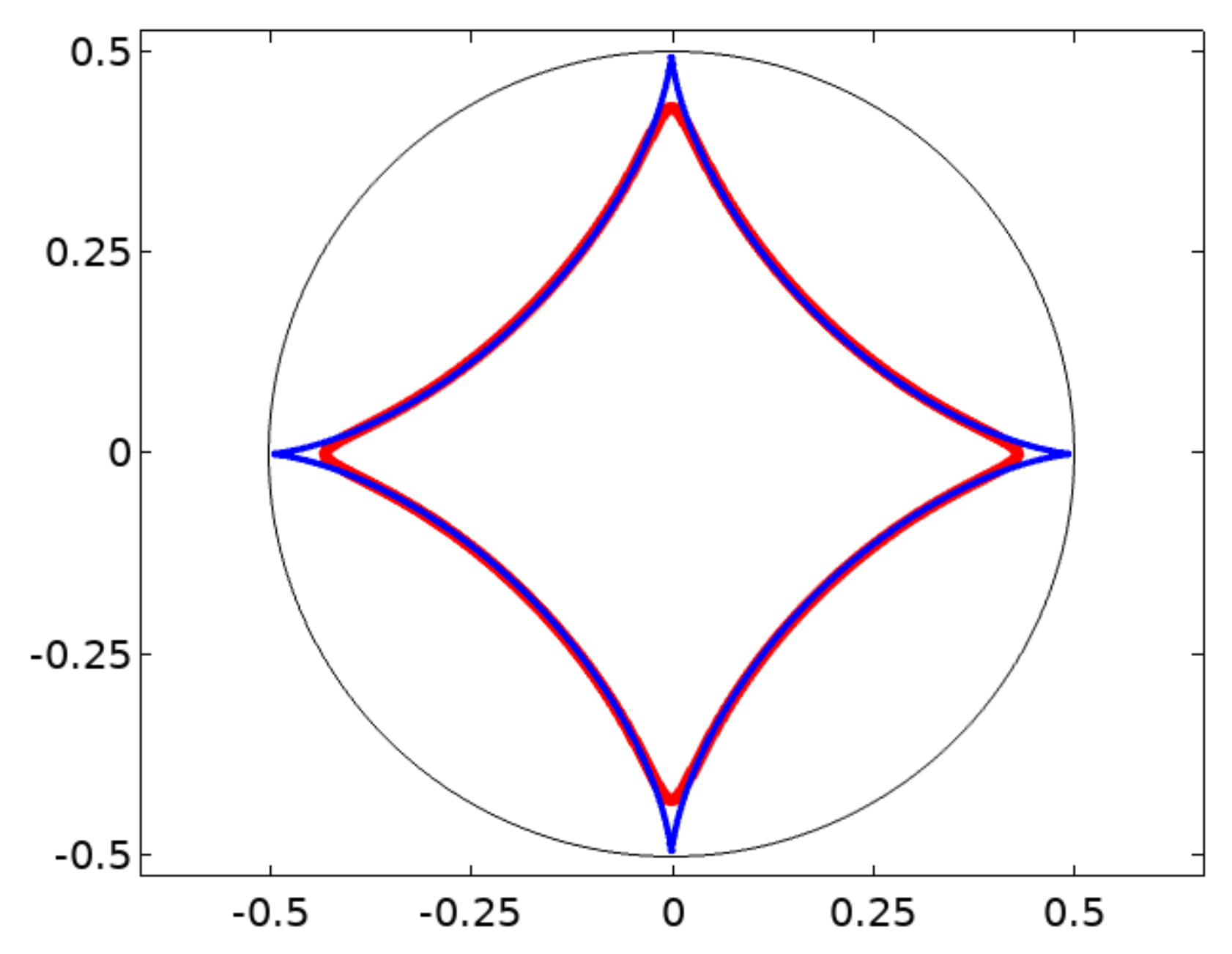}
    \caption{Contour line $|u|=0.5$ for the critical point with $(k,\alpha)=(-1,\pi)$ obtained via gradient flow (red) and the plot of $x_1^{2/3} + x_2^{2/3} = 1$ (blue).  Here $L=2$ and $\e=0.005$.}
    \label{fig:astroidinfty}
\end{figure}

One can also check that for the construction obtained above, the area of the isotropic island increases with $\abs{k}$. 
In fact, a calculation that we omit goes to show that in the $k\to \infty$ limit,
the isotropic region fills the entire disk!
\eeg

The preceding calculation can also be used in the case of $L < \infty$ in order to reconstruct parts of critical points of $E_0$. Recall that, in this case, by Corollary \ref{conservation}, the characteristics are circular arcs of finite radii that may run directly from the interface to $\partial \Omega$. In Fig.~\ref{fig:deg1_neg}, for example, red curves correspond to level sets of $\dive u$ and thus the characteristics for $u$ connect large portions of the interface to the boundary. In order to fully reproduce the critical point of $E_0$ completely, however, one needs to allow for the presence of walls, as evidenced by the gradient flow numerics in Fig.~\ref{fig:deg1_neg}. Although a similar approach yielded critical points of \eqref{limdivBBH} for degree $-1$ boundary data in \cite{GSV}, such a construction will be more elaborate here and we do not pursue this issue further in the present paper.

We conclude this section by considering the boundary data in \eqref{sec4bc} corresponding to $k$ positive and $\alpha=0$.
 The results of the gradient flow simulations are shown in Fig.~\ref{fig:deg2_pos}. Not surprisingly, when $L$ is small for $k=2$, the stable configuration consists of two degree one vortices looking locally like $\vec{e}_\theta$,  see Fig.~\ref{fig:tactoid7}. As $L$ increases, however, these vortices collapse onto  and spread along $\partial\Omega$ while forming two walls along the upper and lower halves of the boundary, respectively, cf. Fig.~\ref{fig:tactoid8}. Indeed this simulation suggests that for $E_0$ with $L$  large, the preferred state is $u\equiv\vec{e}_1$. 
In fact, if one tries to construct a competitor $u$ having a `boundary wall' for this boundary data, that is, a unit vector field such that the normal component of the data is met but the tangential component switches sign, then one finds 
\[
u\cdot\nu_{\partial\Omega}=(\cos 2s,\sin 2s)\cdot (\cos s,\sin s)=\cos s=\vec{e}_1 \cdot(\cos s,\sin s)\] 
and 
\[
u\cdot \tau_{\partial\Omega}=-(\cos 2s,\sin 2s)\cdot (-\sin s,\cos s)=-\sin s=\vec{e}_1 \cdot (-\sin s,\cos s).
\]
Thus such a competitor $u$ must have trace $\vec{e}_1$ along $\partial\Omega$ and there is no need then to accumulate divergence inside the disk by varying from constancy.
 
\begin{figure}
\centering
\begin{subfigure}{.57\textwidth}
  \centering
  \includegraphics[width=\linewidth]{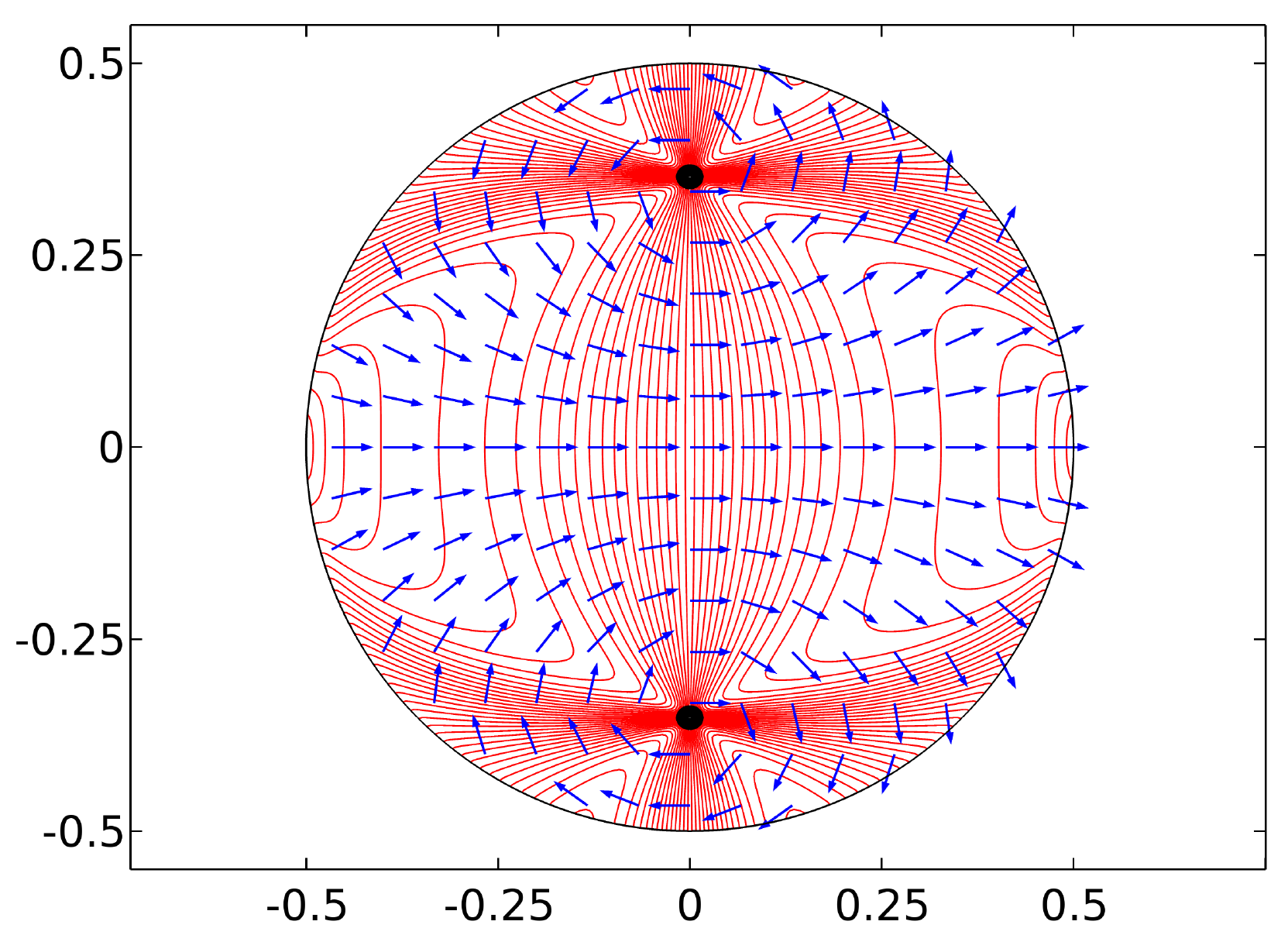}
  \caption{$L=0.04$, $\e=0.005$}
  \label{fig:tactoid7}
\end{subfigure}%
\begin{subfigure}{.5\textwidth}
  \centering
  \includegraphics[width=\linewidth]{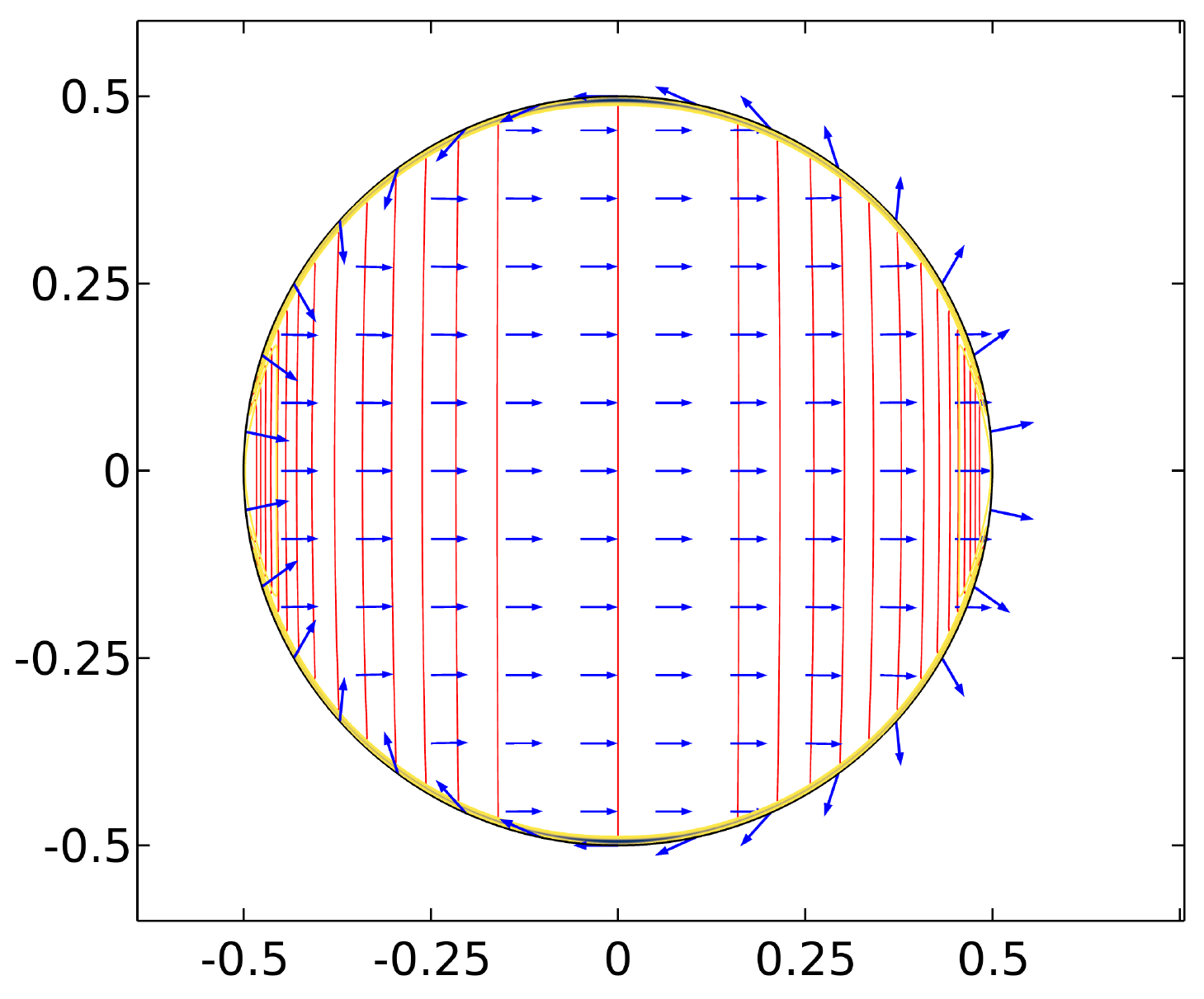}
  \caption{$L=10$, $\e=0.001$}
  \label{fig:tactoid8}
\end{subfigure}
\begin{subfigure}{.5\textwidth}
  \centering
  \includegraphics[width=1.05\linewidth]{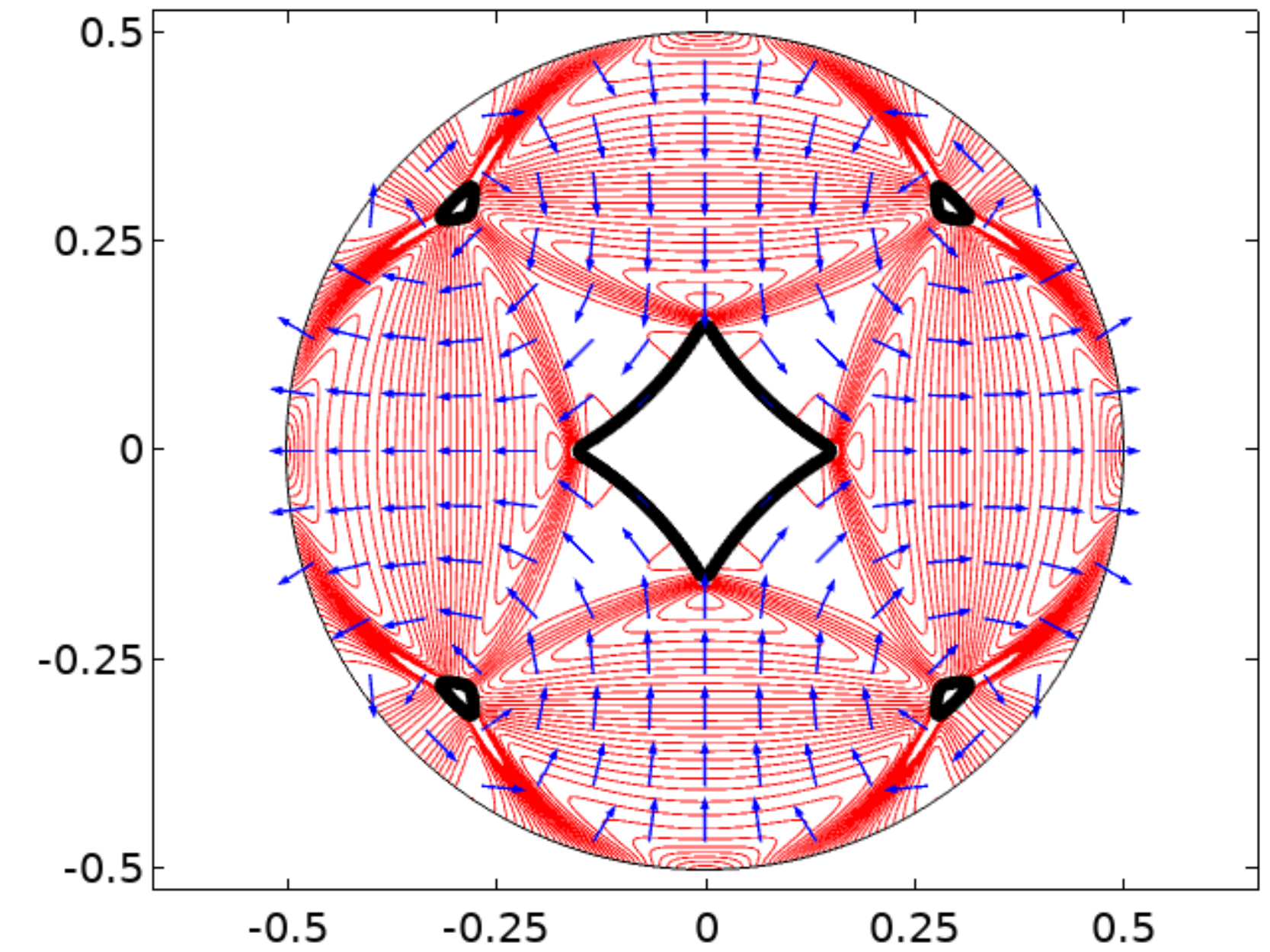}
  \caption{$L=0.04$, $\e=0.005$}
  \label{fig:tactoid9}
\end{subfigure}
\caption{Critical points of $E_\e$ for (a-b): $(k,\alpha)=(2,0)$  and (c): $(k,\alpha)=(3,0)$. The red curves represent level sets of $\mathrm{div}\,u$ while $|u|=0.5$ on the black curves that enclose the isotropic phase. The wall adjacent to the boundary in plot (b) is indicated in yellow.}
\label{fig:deg2_pos}
\end{figure}

In the case of a degree $3$ boundary data, cf. Figure \ref{fig:tactoid9}, the behavior is more complex---the degree $3$ vortex appears to split into four degree $1$ vortices and one degree $-1$ vortex. The four $+1$ vortices approach the boundary of the domain with an increasing $L$ while the degree $-1$ vortex remains at the center of the disk. 

We use the simulations in the case of $(k,\alpha)=(2,0)$ to test Conjecture \eqref{conjecture} on the one-dimensional character of the wall cost. The walls in this example turn out to be significantly deeper than in other cases that we considered and it is therefore easier to compare the numerically computed wall profiles with the corresponding heteroclinic connection. Consider the critical point for $E_\e$ depicted in Fig.~\ref{fig:tactoid8}. For a large value of $L$ the defects present inside $\Omega$ for small $\e$ spread along the boundary to form two boundary walls. Due to symmetry, it is sufficient to consider the wall in the first quadrant. Then along each ray emanating at angle $\theta$ from the origin, the wall connects the vector $\vec{e}_1=(1,0)$ to the vector $(\cos{2\theta},\sin{2\theta})$. Because the normal to the boundary/interface is $(\cos{\theta},\sin{\theta})$, the normal component of the vector field is continuous across the wall, while the tangential component reverses sign. The jump in the tangential component across the wall grows as $\theta$ changes from $0$ to $\pi/2$. In Fig.~\ref{fig:comp} we plot cross-sections of $|u|$ for the critical point of $E_\e$ shown in Fig.~\ref{fig:tactoid8}, where the cross-sections are shown along several rays $\theta=const$. In Fig.~\ref{fig:comp1}, 
\begin{figure}
\centering
\begin{subfigure}{.425\textwidth}
  \centering
  \includegraphics[width=\linewidth]{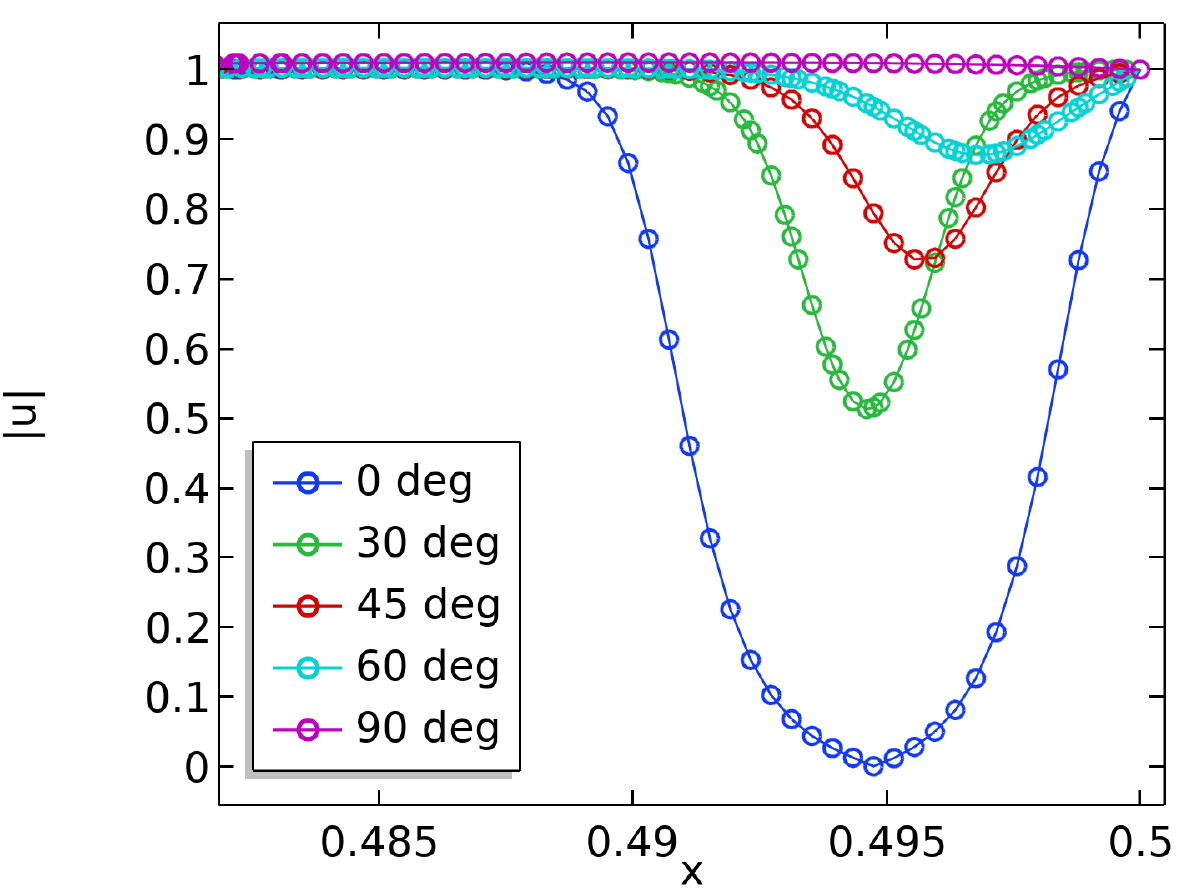}
  \caption{}
  \label{fig:comp}
\end{subfigure}%
\begin{subfigure}{.4\textwidth}
  \centering
  \includegraphics[width=\linewidth]{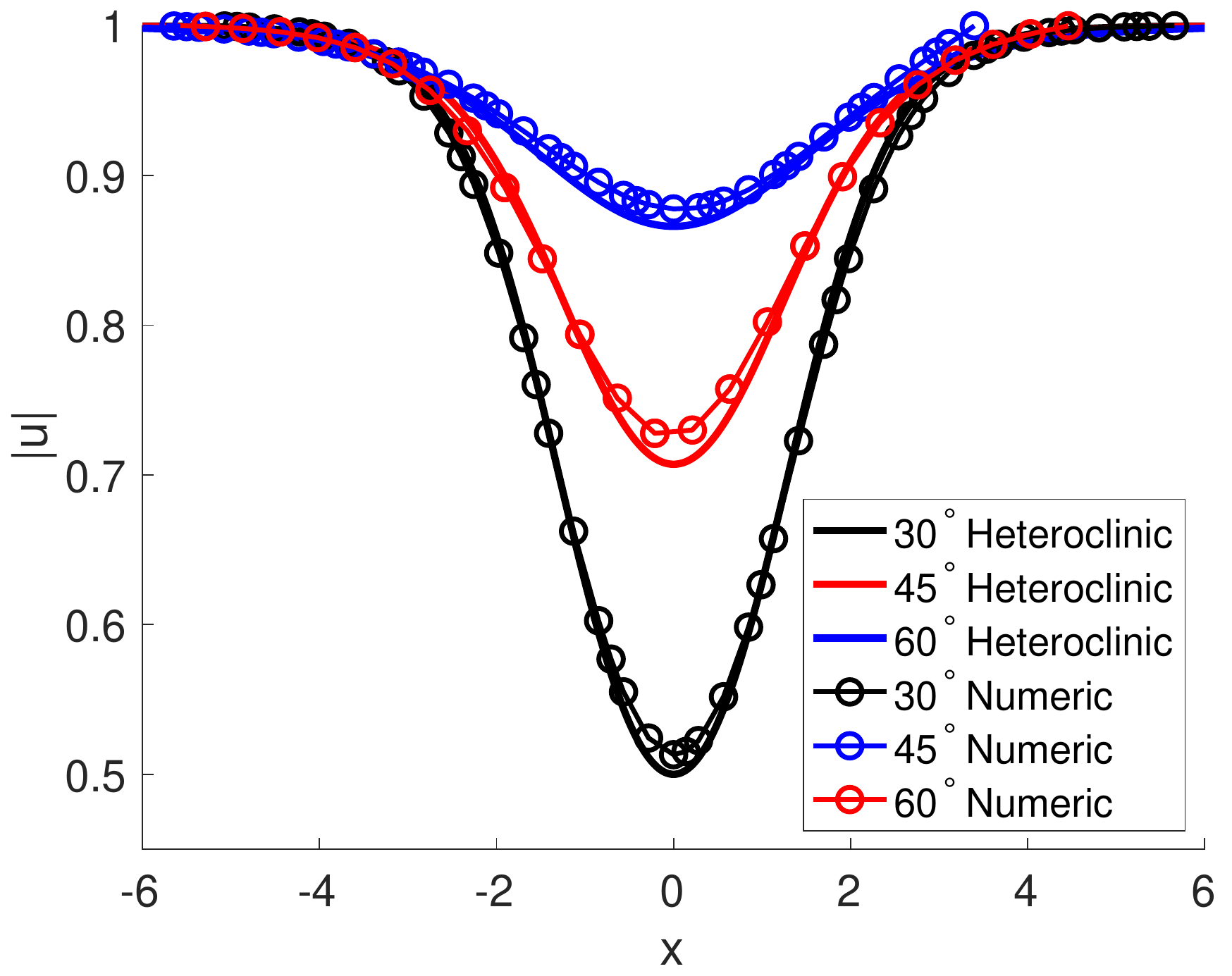}
  \caption{}
  \label{fig:comp1}
\end{subfigure}
\caption{Wall cost is one-dimensional for the critical point shown in Fig.~\ref{fig:tactoid8}. (a) Cross-sections of Fig.~\ref{fig:tactoid8} in the direction of angle $\theta$. Only parts of cross-sections closest to the boundary are shown; (b) Comparison between the numerical and analytical wall profiles.}
\label{fig:comp_sol}
\end{figure}
the same scaled and translated profiles are shown together with the corresponding solutions of the ODE that describes a heteroclinic connection associated with $E_\e$, assuming one-dimensional cost. As Fig.~\ref{fig:comp1} demonstrates, the graphs are close to each other for all respective values of $\theta$---this is in agreement with our conjecture that the cost is one-dimensional.

\subsection{Examples for Degree Zero Boundary Data}\label{eyeballs}
In this section, we analytically and numerically construct an example with an isotropic tactoid which exhibits two defects on the phase boundary. Let us describe a key feature of this example. Recall that at a nematic-isotropic interface for $E_0$, the trace of a $BV\cap H_\dive$ competitor $u$ from the nematic region is tangent to the interface, cf. \eqref{tanjump}. If, for example, $u$ is smooth and does not change the sense of tangency along the interface, then the degree of $u$ around any connected component of the interface is $1$. If we specify a degree $0$ boundary condition around $\partial \Omega$ or at infinity, this mismatch can be rectified by the presence of two defects along the interface, similar to the construction of the recovery sequence in Section \ref{firsttry}. This is the effect we will see in the following example.\par

We begin with some numerics. Fig.~\ref{walleye} shows the result of gradient flow simulation in a large rectangular domain with constant boundary data $\vec{e}_1$. We observe in Fig. \ref{walleye} that (i) the interface surrounds a single isotropic island, (ii) there appear to be two walls which intersect the two defects on the interface, and (iii) the solutions possess the symmetries
\begin{equation*}
(u_1(x_1,x_2),u_2(x_1,x_2)) = (u_1(x_1,-x_2),-u_2(x_1,-x_2))\end{equation*} and \begin{equation*} (u_1(x_1,x_2),u_2(x_1,x_2)) =(u_1(-x_1,x_2),-u_2(x_1,-x_2)).
\end{equation*}
Furthermore, in Fig. \ref{walleye} we see that (iv) the walls divide the plane into three regions, with
\begin{equation}\label{approach}
u\equiv \vec{e}_1 \textup{ in the two regions not containing the isotropic tactoid.}
\end{equation}

This simulation, though depicting transient behavior, leads us to seek a critical point of $E_0^{\infty}$ satisying (i)-(iv) consisting of an isotropic tactoid in an infinite sea of nematic, where in the far field, $u\to \vec{e}_1$. To induce a static--and presumably stable--critical point having an isotropic phase, we will enforce an area constraint of the form
$\abs{\{u=0\}}=const.$

The complication here--and it is a significant one--is that an interface and a wall are rigidly linked via the straight-line characteristics lying in between them, and the requirement of tangency of $u$ along the interface and agreement of the normal component of $u$ with that of $\vec{e}_1$ along the wall make the construction rather daunting. 

 Somewhat surprisingly, we are able to achieve this construction by deriving a formula of the form 
\begin{equation}\label{form}
E_0^\infty(u) = \int_{\partial\{\abs{u}=1\}} f(\theta) \, d\mathcal{H}^1
\end{equation}
for an explicit $f$, for competitors $u$ satisfying (i)--(iv), where $\theta$ is the angle the tangent vector to $\partial\{ |u|=1\}$ makes with the horizontal. For such $u$, the energy $E_0^\infty$ therefore only depends on the interface, a reflection of the afore-mentioned rigidity of this problem. We then consider variations of the interface to derive an ODE \eqref{ode} for $\theta$ along with the junction condition \eqref{simpjunction} at the intersection of walls and interfaces. Numerically integrating this ODE yields a configuration which closely resembles the results of the simulations shown in Figure \ref{walleye}, cf. Figure \ref{walleyeanalytical}.
\vskip.1in
\noindent
\underline{\textbf{The Role of Walls:}}
Before embarking on this construction, let us comment on the role of walls in this example. A natural question is: given these conditions, namely an area constraint on the isotropic region and the requirement that $u\equiv \vec{e}_1$ in the far field, is it necessary for a critical point to have walls? While we do not as yet have a proof, we believe the answer is yes. Let us present some heuristic arguments to this effect. Working within the symmetry assumption (iii), consider the possibility of constructing a competitor without walls. Then one of the two configurations depicted in Fig.~\ref{Explanationwalleye} are possible where the isotropic island either has two corners and no walls or it has two cusps and no walls. In the former case, one can show that partial vortices should form near the corners in the nematic phase, causing the divergence contribution to the energy to blow up; see Fig. \ref{nowalls} above.
If there are two cusps and no walls, Fig.~\ref{Explanationwalleye}  demonstrates that this is not possible as the characteristics emanating from the interface would have to intersect non-tangentially. In light of these observations, junctions between interfaces and walls appear to be fairly generic, making in particular the junction condition \eqref{junko} potentially important when analyzing candidates for possible critical points or minimizers when $L$ is finite.
\begin{figure}
\centering
\begin{subfigure}{.35\textwidth}
  \centering
  \includegraphics[width=\linewidth]{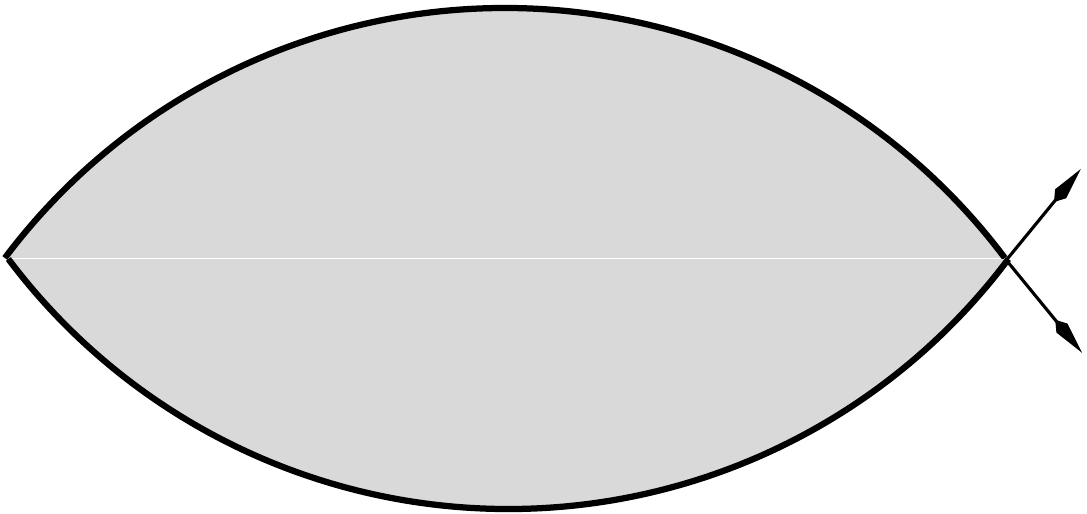}
\end{subfigure}\qquad\qquad
\begin{subfigure}{.4\textwidth}
  \centering
  \includegraphics[width=\linewidth]{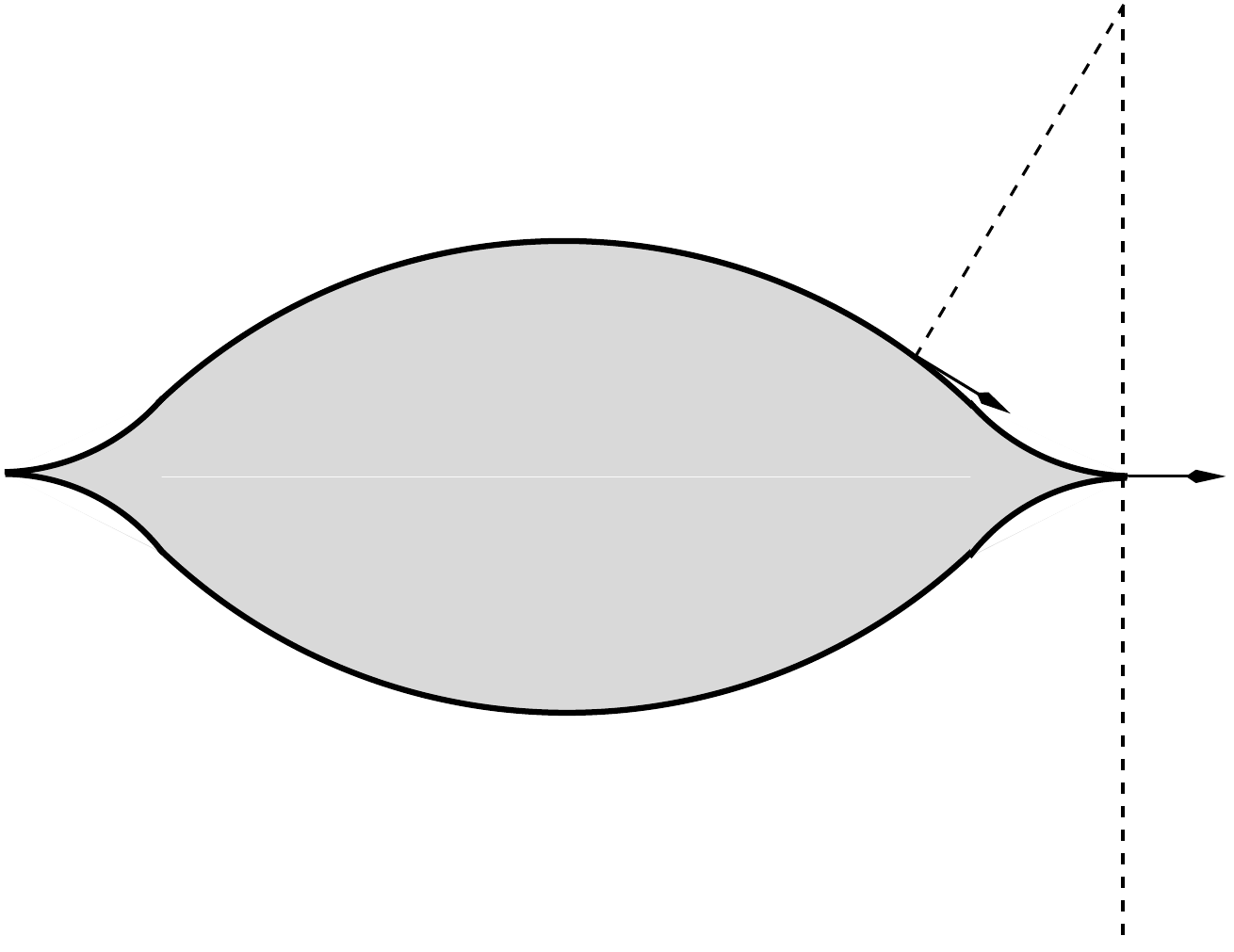}
\end{subfigure}
\caption{An isotropic island (marked in gray) surrounded by the nematic medium  in $\mathbb R^2$ without wall singularities and satisfying $u=\vec{e}_1$ at infinity. The vector field is tangent to the boundary of the island and switches the sense of tangency at the boundary singularities.  Left: the island has a corner and thus an infinite energy $E_0$ due to the bulk divergence term---cf. Fig.~\ref{nowalls}. Right: the island has a cusp---impossible without a wall since the characteristics will intersect. Characteristics are indicated by the dashed lines.}
\label{Explanationwalleye}
\end{figure}
\par
\beg\label{isotropiceyeball}\upshape
For this calculation, by (iii), it suffices to consider the problem in the first quadrant $Q_1$. Let us assume that 
$\partial\{|u|=1 \}\cap Q_1$ is smooth and can be parametrized by $r(\sigma)$ where $r:[0,L]\to Q_1$, with $r(0)$ on the $x_1$-axis and $r(L)$ on the $x_2$-axis; see Figure \ref{Eyeballcalccropped}.
\begin{figure}[h]
\centering
\includegraphics[scale=.8]{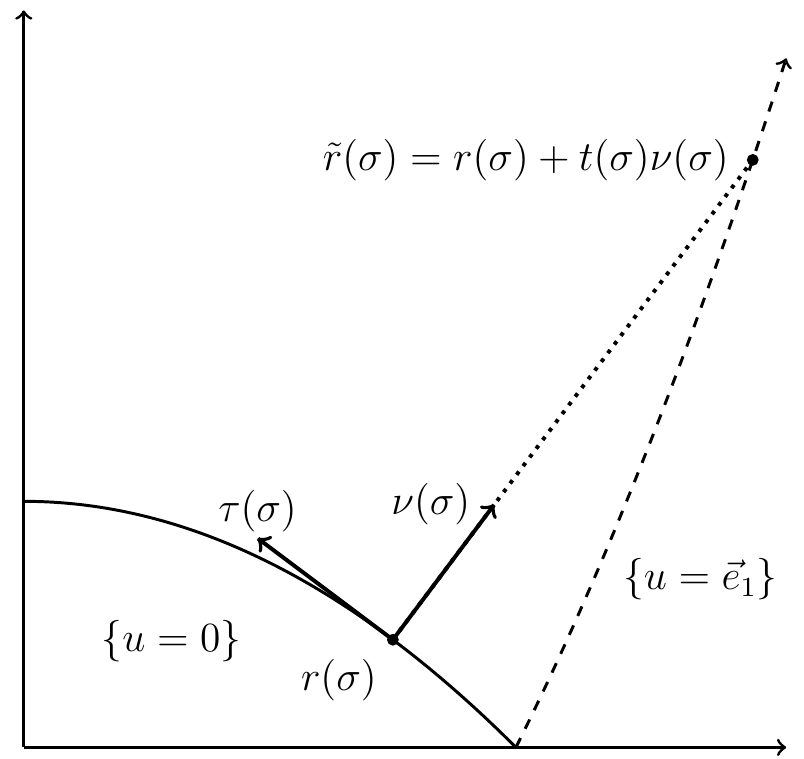}
\caption{The isotropic island is surrounded by the nematic region. The dotted line represents a straight line characteristic, and the dashed line represents the wall.}
 \label{Eyeballcalccropped}
\end{figure}
We do not assume that the interface is parametrized by the arclength variable $s$ in this derivation. Then
\begin{align}\label{r}
r'(\sigma)/|r'(\sigma)| = \tau(\sigma) = (\cos \theta(\sigma), \sin \theta(\sigma))=-u(r(\sigma)). 
\end{align}
Let us define 
\begin{equation*}
\rho(\sigma) = |r'(\sigma)|, 
\end{equation*}
and the normal vector
\begin{equation*}
\nu(\sigma) = (\sin \theta(\sigma),- \cos \theta(\sigma)).
\end{equation*}\par
We now deduce the location of the wall, which we will see is determined by the interface. Recall that in light of \eqref{AVrecipe}, $u$ is perpendicular to the straight characteristics, which themselves intersect the interface perpendicularly,  so we can parametrize the wall by shooting characteristics off of the interface until they hit the wall. We can write a parametrized path $\tilde{r}$ for the wall then as
\begin{equation}\label{tilder}
\tilde{r}(\sigma) := r(\sigma) + t(\sigma)\nu(\sigma).
\end{equation}
Hence by \eqref{r} the trace of $u$ on the wall from the left, denoted here by $\tilde{u}$, is given by
\begin{equation}\label{utilde}
    \tilde{u}(\sigma)=-(\cos \theta(\sigma), \sin \theta(\sigma)).
\end{equation}
We define a function $\psi$ by the equation
\begin{equation}\notag
    \tilde{r}'(\sigma) = |\tilde{r}'(\sigma)|(\cos \psi(\sigma),\sin \psi(\sigma)).
\end{equation}
Then the tangent and normal vectors to the wall are given by
\begin{equation}\label{nutilde}
    \tilde{\tau} = (\cos \psi, \sin \psi) \quad \textup{and}\quad \tilde{\nu} = (\sin \psi, - \cos \psi).
\end{equation}\par
Next, we collect some relations between several of the above quantities which will be useful in the following calculations. From the continuity of the normal traces across a wall, we have
\begin{equation}\notag
    \tilde{u} \cdot \tilde{\nu} = \vec{e}_1 \cdot \tilde{\nu},
\end{equation}
which we rewrite using \eqref{utilde}, \eqref{nutilde}, and the angle subtraction identity for $\sin$ as
\begin{equation}\label{sin}
\sin (\theta - \psi ) = \sin \psi.
\end{equation}
Similarly, the condition $\tilde{u} \cdot \tilde{\tau} = - \vec{e}_1 \cdot \tau$ for the tangential components across a wall, cf. \eqref{tanjump}, can be expressed as
\begin{equation}\label{cos}
    \cos(\theta - \psi ) = \cos \psi.
\end{equation}
From \eqref{sin}, \eqref{cos} it follows that $\psi = \theta/2 + k\pi$ for some integer $k$, and $k$ is in fact $0$ since at $\sigma=0$, $\psi(0) \leq \theta(0) \leq \psi(0) + \pi/2$. Thus
\begin{equation}\label{psi}
\psi = \theta /2.
\end{equation}\par
We can now write the energy $E_0(u,Q_1)$ in the first quadrant as
\begin{align}\notag
E_0(u,Q_1) &= \frac{K(0)}{2} \per_{Q_1}(\{|u|=1 \})+\int_{J_u \cap \{|u|=1 \}}\notag K(u \cdot \nu)\,d \mathcal{H}^1 \\ 
&=\int_0^L\left(\frac{K(0)}{2}|r'(\sigma)|+K(\vec{e}_1 \cdot \tilde{\nu}(\sigma))|\tilde{r}'(\sigma)| \right) \,d\sigma\label{ee}.
\end{align}
One can use \eqref{tilder} and the orthogonality of $\tau$ and $\nu$ to easily calculate $$|\tilde{r}'|=\left((\rho+t\theta')^2+(t')^2\right)^{1/2}.$$
Substituting this and $$\vec{e}_1 \cdot \nu = \sin \psi = \sin (\theta /2) $$into \eqref{ee} yields
\begin{align}\notag
    E_0(u,Q_1) &= \int_0^L\left(\frac{K(0)}{2}|r'|+K(\sin (\theta / 2))\left((\rho+t\theta')^2+(t')^2\right)^{1/2} \right) \,d\sigma\notag\\ &=\int_{\partial \{|u|=1\}}\frac{K(0)}{2}\,d\mathcal{H}^1+\int_0^L \left(K(\sin (\theta / 2))\left((\rho+t\theta')^2+(t')^2\right)^{1/2} \right) \,d\sigma.\label{e0}
\end{align}
In order to obtain a formula for $E_0$ depending only on the interface, it remains to simplify \begin{equation}\label{mess}
\left((\rho+t\theta')^2+(t')^2\right)^{1/2}.\end{equation}\par
We begin this simplification by finding an expression for $t$ in terms of $\theta$ and $\rho$. Using the definitions \eqref{r}, \eqref{tilder}, and \eqref{nutilde} for $r$, $\tilde{r}$, and $\tilde{\nu}$, respectively, along with \eqref{psi}, we calculate
\begin{samepage}
\begin{align}\notag
0 &= \tilde{r}' \cdot \tilde{\nu} \\ \notag
&=\left(r' + t' \nu + t\nu' \right) \cdot \tilde{\nu} \notag \\
&= \left[ \rho(\cos \theta, \sin \theta) + t'(\sin \theta, - \cos \theta)+ t \theta' (\cos \theta, \sin \theta)\right]  \cdot (\sin (\theta /2), - \cos (\theta /2))\label{dot}.
\end{align}
\end{samepage}
Expanding out \eqref{dot} and using the angle subtraction formulae for sine and cosine eventually gives
\begin{equation}\label{ode1}
t' \cos (\theta / 2) - (\rho + t\theta')\sin (\theta /2) =0.
\end{equation}
Now we observe from our symmetry assumption on $u$ that $u \equiv \vec{e}_1$ on the $x_2$-axis, so that $\theta(L) = \pi$. If we assume that $\theta(\sigma)$ does not reach $\pi$ until $\theta(L)=\pi$, which in terms of the interface means that \begin{equation}\label{pos1}
\textit{the tangent vector to the interface is not horizontal in the interior of }Q_1,
\end{equation}
then we can divide \eqref{ode1} by $\cos (\theta /2)$. This results in the following ODE for $t$:
\begin{equation}\label{odet}
    t' - \frac{\sin (\theta /2)}{\cos (\theta /2)} \theta ' t= \rho \tan (\theta /2).
\end{equation}
Multiplying both sides of \eqref{odet} by the integrating factor
\begin{equation}\notag
M = \exp\left(-2\displaystyle\int\frac{\sin (\theta /2)}{\cos (\theta /2)} \frac{\theta'}{2}\right) = \exp \left(2\ln (\cos (\theta /2)) \right)= \cos^2 (\theta /2) 
\end{equation}
results in
\begin{equation}\notag
(t \cos^2(\theta /2))' = \rho \sin(\theta /2) \cos(\theta/2)= \frac{1}{2} \rho \sin \theta.
\end{equation}
Integrating both sides, dividing by $\cos^2(\theta/2)$, and using the half angle formula for cosine, we obtain
\begin{equation}\label{tform}
t = \frac{1}{2\cos^2(\theta/2)} \int_0^\sigma \rho(y) \sin \theta(y)\, dy = \frac{1}{1+\cos \theta} \int_0^\sigma \rho \sin \theta \,dy.
\end{equation}
Finally, let us record the identity
\begin{equation}\label{abs}
     \rho + t\theta' = t' \frac{\cos(\theta/2)}{\sin(\theta/2)},
\end{equation}
which follows from rearranging \eqref{ode1}.\par
We now use the formula \eqref{tform} for $t$ to calculate \eqref{mess}, the quantity we set out to simplify. Let us assume that $t'>0$, which means that
\begin{equation}\label{pos2}
    \textit{the length of characteristics connecting the interface to the wall increases in }\sigma.
\end{equation}
Then plugging in \eqref{abs} for \eqref{mess} and using the assumptions \eqref{pos1} and \eqref{pos2}, namely $\theta/2 \leq \pi/2$ and $t'>0$, we write
\begin{align}\notag
\left((\rho+t\theta')^2+(t')^2\right)^{1/2} = \left((t')^2 \frac{\cos^2(\theta/2)}{\sin^2(\theta/2)} + (t')^2\right)^{1/2} = \frac{t'}{\sin (\theta /2)}.
\end{align}
Utilizing the formula \eqref{tform} to calculate $t'$ and then a half angle formula for cosine and a double angle formula for sine, we arrive at
\begin{align}\notag
\left((\rho+t\theta')^2+(t')^2\right)^{1/2} &= \frac{1}{\sin (\theta /2)} \left[ \frac{\rho\sin\theta}{1+\cos\theta}+\frac{\theta' \sin \theta}{(1+\cos \theta)^2}\int_0^\sigma \rho \sin \theta\,dy\right]\notag \\
&= \frac{\sin \theta}{2\sin (\theta /2) (1+\cos\theta)}\left[ \rho+\frac{\theta' }{1+\cos \theta}\int_0^\sigma \rho \sin \theta\,dy\right] \notag \\
&= \frac{2 \sin (\theta/2)\cos(\theta/2)}{2\sin (\theta /2) \cos^2(\theta/2)}\left[ \rho+\frac{\theta' }{1+\cos \theta}\int_0^\sigma \rho \sin \theta\,dy\right]\notag \\
&= \frac{1}{ \cos(\theta/2)}\left[ \rho+\frac{\theta' }{1+\cos \theta}\int_0^\sigma \rho \sin \theta\,dy\right]. \label{rtwid'}
\end{align}\par
Now we are ready to use the expression \eqref{rtwid'} for \eqref{mess} in the $E_0$ energy \eqref{e0}. We have
\begin{align}
E_0(u,Q_1)&=E_0(\rho,\theta)\notag \\ &= \int_{\partial\{|u|=1\}} \frac{K(0)}{2} \, d\mathcal{H}^1+ \int_0^L \left(\frac{K(\sin(\theta/2))}{\cos(\theta/2)} \left[\rho + \frac{\theta' }{1+\cos \theta}\int_0^\sigma \rho \sin \theta \,dy\right ]\right)\,d\sigma .\notag
\end{align}
We focus on the term 
\begin{equation}\label{line}
    \int_0^L \left(\frac{K(\sin(\theta/2))}{\cos(\theta/2)} \frac{\theta' }{1+\cos \theta}\int_0^\sigma \rho \sin \theta \,dy\right) \,d\sigma
.\end{equation}
Let us define the function $H(v)$ by the equations $$H'(v)=\frac{K(v)}{(1-v^2)^2},\quad H(0)=0.$$
It follows from \eqref{1d} that $H$ remains finite as $v$ approaches $1$ so long as $V(t)$ approaches $0$ as $t \nearrow 1$ at least as fast as $c(1-t^2)^p$ for some $p>1$ and $c>0$, an assumption which is satisfied by $W_{CSH}$. A straightforward calculation, which we omit, using the chain rule, the definition of $H'$, and some trigonometric identities yields
\begin{equation}\label{ident}
    (H(\sin (\theta/2)))' = \frac{K(\sin (\theta/2))\theta'}{\cos(\theta/2)(1+\cos \theta)}.
\end{equation}
Inserting this expression into the last integral in \eqref{line}, that term becomes
\begin{align}\label{insert}
    \int_0^L  (H(\sin (\theta /2)))' \left(\int_0^\sigma \rho \sin \theta \,dy\right) \, d\sigma,
\end{align}
which we integrate by parts to obtain
\begin{align}\notag
\left[H(\sin (\theta /2))\int_0^\sigma \rho \sin \theta\,dy\right]\Biggr|_0^L - \int_0^L H(\sin (\theta /2)) \rho \sin \theta \, d\sigma.
\end{align}
Note that by our symmetry assumptions, $\theta(L)=\pi$, so that
\begin{align}\notag
   \left[ H(\sin (\theta /2))\int_0^\sigma \rho \sin \theta \,dy \right]\Biggr|_0^L &- \int_0^L H(\sin (\theta /2)) \rho \sin \theta  d\sigma \\ &= H(1)\int_0^L \rho \sin \theta \, d\sigma - \int_0^L H(\sin (\theta /2)) \rho \sin \theta \, d\sigma.\label{ibp}
\end{align}We combine \eqref{ident}--\eqref{ibp} to rewrite \eqref{line}:
\begin{equation}\label{final}
    \int_0^L \left(\frac{K(\sin(\theta/2))}{\cos(\theta/2)} \frac{\theta' }{1+\cos \theta}\int_0^\sigma \rho \sin \theta\,dy \right) \,d\sigma = \int_0^L (H(1) - H(\sin (\theta /2))\rho\sin \theta \, d\sigma.
\end{equation}
Using the right hand side of \eqref{final} for \eqref{line}, we finally have
\begin{samepage}
\begin{align*}
 E_0(u,Q_1)&=E_0(\rho,\theta)\notag \\ &= \int_{\partial\{|u|=1\}} \frac{K(0)}{2} \, d\mathcal{H}^1+ \int_0^L \left(\frac{K(\sin(\theta/2))}{\cos(\theta/2)} \left[\rho + \frac{\theta' }{1+\cos \theta}\int_0^\sigma \rho \sin \theta\,dy \right ]\right)\,d\sigma \\
 &= \int_{\partial\{|u|=1\}} \frac{K(0)}{2} \, d\mathcal{H}^1 + \int_0^L \left(\frac{K(\sin(\theta/2))}{\cos(\theta/2)} +(H(1)-H(\sin(\theta/2))\sin \theta \right)\rho\,d\sigma\\
 &=:  \int_{\partial\{|u|=1\}} f(\theta) \, d\mathcal{H}^1.
\end{align*}
\end{samepage}
Thus we arrive at \eqref{form}.\par
We turn now to the criticality conditions for $\theta$. For any $u$ with smooth interface, we parametrize the interface of length $l$ by arclength $s$. Then the standard derivation \cite{GuAn} gives the following condition on the interface 
\begin{align}\label{ode}
    (f''(\theta)+f(\theta))\theta'+\lambda=0.
\end{align}
along with the the junction condition 
\begin{equation}\label{simpjunction}
    f'(\theta)\sin \theta - f(\theta) \cos \theta =0 
\end{equation}
at $s=0$, the intersection of $\partial \{|u|=1\}$ with the $x_1$-axis.

The solution of \eqref{ode}-\eqref{simpjunction} for $\lambda=1$ is depicted in Fig.~\ref{walleyeanalytical} and bears a strong resemblance to a configuration observed in gradient flow dynamics shown in Fg.~\ref{walleye}.
\eeg
\begin{figure}
\centering
\begin{subfigure}{.35\textwidth}
  \centering
  \includegraphics[width=\linewidth]{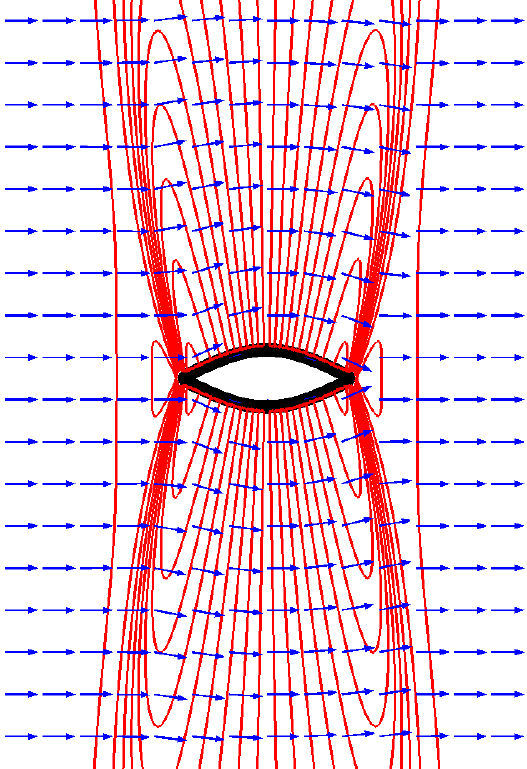}
  \caption{}
  \label{walleye}
\end{subfigure}\qquad
\begin{subfigure}{.35\textwidth}
  \centering
  \includegraphics[width=\linewidth]{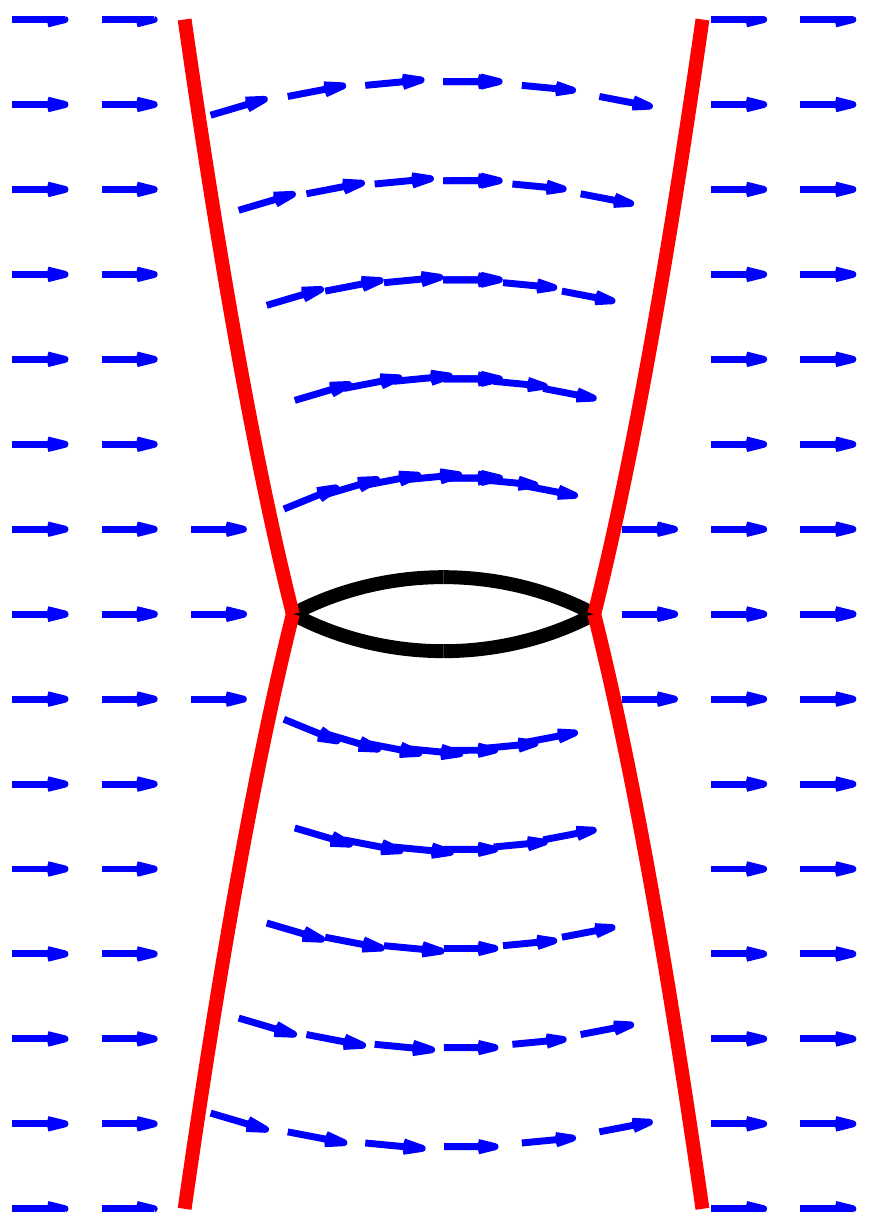}
  \caption{}
  \label{walleyeanalytical}
\end{subfigure}%
\caption{Isotropic island in $\mathbb R^2$ with $u=\vec{e}_1$ at infinity: (a) Gradient flow simulation in a large domain intended to represent $\mathbb R^2$. The isotropic region is shrinking and the solution shown is a transient. Here $L=10$, $\varepsilon=0.02$; (b) Solution of \eqref{ode}-\eqref{simpjunction} for $\lambda=1$.}
\label{TactoidDivergenceandwalleye}
\end{figure}

\section{Appendix}
We present here the proof of Theorem \ref{junction}. See Fig. \ref{Junctionfig} for a guide to the notation.

\begin{proof}
The derivation of \eqref{junk} follows the same general lines as those appearing in the proof of Theorem \ref{critthm2}. However,
a major complicating consideration is that it is no longer possible to assume that the deforming vector field $X$ is normal to all four curves $\Gamma_{ij}$ since they all meet at $p$. Instead we will have to incorporate tangential components of $X$ along these four curves as well. 

To this end, we assume simply that $X\in C_0^1(B(p,R);\R^2)$ and again introduce the map $\Psi$ via \eqref{odepsi}. We assume that each
$\Gamma_{ij}$ is smoothly parametrized by arclength through a map $r_{ij}:[0,s_0]\to\Gamma_{ij}$ for some $s_0>0$ with $r_{ij}(0)=p$. Then we replace
\eqref{hdefn} by 
\beq
X(r_{ij}(s))=\htan_{ij}(s)\tau_{ij}(s)+\hnor_{ij}(s)\nu_{ij}(s)\quad\mbox{for}\;s\in [0,s_0],\label{tannor}
\eeq
where 
\[
\htan_{ij}:=X(r_{ij}(s))\cdot\tau_{ij}(s)\quad\mbox{and}\quad \hnor_{ij}(s):=X(r_{ij}(s))\cdot\nu_{ij}(s).
\]
As a consequence of the compact support of $X$, we have that
\beq
\htan_{ij}(s_0)=\hnor_{ij}(s_0)=0\quad\mbox{for all functions}\;\htan_{ij}\;\mbox{and}\;\hnor_{ij}\label{zeroendpoint}
\eeq
but we stress that none of these functions is assumed to vanish at $s=0$, namely at the location of the junction $P$.
 
We now deform each region $\Omega_j$, for $j=0,1,2,3$ by the map $\Psi$ to form
four contiguous regions $\Omega_j^t:=\Psi(\Omega_j,t)$ and we deform the four boundary curves $\Gamma_{ij}$ to form
four new boundary curves $\Gamma_{ij}^t:=\Psi(\Gamma_{ij},t).$ Of course the junction point $P$ is also carried along by this flow.

The four curves $\Gamma_{ij}^t$ are parametrized by $s\mapsto \Psi(r_{ij}(s),t)$ which we denote by $r_{ij}^t(s)$ though $s$ no longer represents arclength.
Indeed one calculates that
\beq
r_{ij}^t(x)\sim r_{ij}(s)+t\big(\htan_{ij}(s)\tau_{ij}(s)+\hnor_{ij}(s)\nu_{ij}(s)\big)\label{newtan}
\eeq
from which it follows that
\beq
\abs{r_{ij}^t\,'(s)}\sim 1+t\big(\htan_{ij}\,'(s)-\hnor_{ij}(s)\kappa_{ij}(s)\big),\label{ijarc}
\eeq
where $\kappa_{ij}(s)$ denotes the curvature of $\Gamma_{ij}$ at $r_{ij}(s)$ (compare with \eqref{notarc}) and we have invoked the Frenet relations $\tau_{ij}'=\kappa_{ij}\nu_{ij}$ and $\nu_{ij}'=-\kappa_{ij}\tau_{ij}$. A related calculation
goes to show that the unit normal $\nu_{ij}^t$ to $\Gamma_{ij}^t$ is given by
\beq
\nu_{ij}^t\sim \nu_{ij}-t\big(\hnor_{ij}\,'+\kappa_{ij}\htan_{ij}\big)\tau_{ij}.\label{newnorm}
\eeq

Now in the ball $B(p,R)$ the unperturbed critical point is given by
\[
u(x)=\left\{\begin{matrix} 0&\;\mbox{for}\;x\in\Omega_0,\\
u_1(x)&\;\mbox{for}\;x\in\Omega_1,\\
u_2(x)&\;\mbox{for}\;x\in\Omega_2,\\
u_3(x)&\;\mbox{for}\;x\in\Omega_3\end{matrix}\right.
\]
and we wish to perturb it into a new function $u^t$ given by
\[
u^t(x)=\left\{\begin{matrix} 0&\;\mbox{if}\;x\in\Omega_0^t,\\
u_1^t(x)&\;\mbox{for}\;x\in\Omega_1^t,\\
u_2^t(x)&\;\mbox{for}\;x\in\Omega_2^t,\\
u_3^t(x)&\;\mbox{for}\;x\in\Omega_3^t\end{matrix}\right..
\]
To carry this out, as in the previous proof, we extend the domain of definition of $u_{j}$ to a neighborhood of $\Omega_{j}$ in such a way that the extension is constant along the normals to the  boundary of its original domain of definition. Then we introduce three functions
$\phi_1,\phi_2$ and $\phi_3$ such that
\beq
u_j^t(x)\sim u_j(x)+t\phi_j(x)u_j(x)^{\perp}\quad\mbox{for}\;x\in\Omega_{j}^t\;\mbox{and for}\;j=1,2,3\label{newutexp}
\eeq
so as to preserve the required $\mathbb{S}^1$-valued nature of $u_j^t.$

We must also take care to preserve the property $u^t\in H_\dive$ in the sense of \eqref{cnormal} and this requires that the following four conditions hold to $O(t)$ along
$\Gamma_{01},\Gamma_{12},\Gamma_{23}$ and $\Gamma_{03}$ respectively:
\begin{eqnarray}
&&u_1^t(r_{01}^t(s))\cdot\nu_{01}^t(s)=0,\qquad u_1^t(r_{12}^t(s))\cdot\nu_{12}^t(s)=u_2^t(r_{12}^t(s))\cdot\nu_{12}^t(s),\nonumber\\
&&u_2^t(r_{23}^t(s))\cdot\nu_{23}^t(s)=u_3^t(r_{23}^t(s))\cdot\nu_{23}^t(s),\qquad\mbox{and}\qquad u_3^t(r_{03}^t(s))\cdot\nu_{03}^t(s)=0\;
\mbox{for}\; s\in [0,s_0].\nonumber\\
&&\label{needhdiv}
\end{eqnarray}
We note that the first and last of these conditions implies at $t=0$ that either $u_1\equiv \tau_{01}$ or $\equiv-\tau_{01}$ along $\Gamma_{01}$ and
likewise either $u_3\equiv \tau_{03}$ or $\equiv-\tau_{03}$ along $\Gamma_{03}$. 

Substituting  \eqref{newtan} and \eqref{newnorm} into the four conditions of \eqref{needhdiv}, and expanding to $O(t)$, a tedious but straight-forward calculation leads to the following requirements relating the traces of the $\phi_j$ to $\hnor_{ij}\,'$:
\begin{eqnarray}
&&\phi_1\big(r_{01}(s)\big)=\hnor_{01}\,'(s),\label{h01}\\
&& \frac{1}{2}\bigg(\phi_1\big(r_{12}(s)\big)+\phi_2\big(r_{12}(s)\big)\bigg)=\hnor_{12}\,'(s),\label{h12}\\
&& \frac{1}{2}\bigg(\phi_2\big(r_{23}(s)\big)+\phi_3\big(r_{23}(s)\big)\bigg)=\hnor_{23}\,'(s),\label{h23}\\
&&\phi_3\big(r_{03}(s)\big)=\hnor_{03}\,'(s),\label{h03}
\end{eqnarray}
for $s\in [0,s_0].$

With the perturbations of the four curves $\Gamma_{ij}$ and three functions $u_j^t$ defined, we are ready to compute the variation of
$E_0$ in a neighborhood of the junction point $P$. Carrying out the calculation \eqref{jacint} in $\Omega_j$ for $j=1,2,3$ and then applying the divergence theorem we find with the aid of \eqref{graddiv} that
\begin{eqnarray*}
&&
 \frac{d}{dt}_{|_t=0}\sum_{j=1}^3\left(\int_{\Omega^t_j}(\dive\,u^t_j)^2\,dx\right)=\\
&&
-\int_{\Gamma_{01}}\bigg( (\dive\,u_1)^2\,\hnor_{01}+2(\dive\, u_1)(u_1\cdot\tau_{01})\phi_1            \bigg)\,ds\\
&&+\int_{\Gamma_{12}}\bigg( (\dive\,u_1)^2\,\hnor_{12}+2(\dive\, u_1)(u_1\cdot\tau_{12}  )\phi_1            \bigg)\,ds\\
&&
-\int_{\Gamma_{12}}\bigg( (\dive\,u_2)^2\,\hnor_{12}+2(\dive\, u_2)(u_2\cdot\tau_{12}  )\phi_2            \bigg)\,ds\\
&&+\int_{\Gamma_{23}}\bigg( (\dive\,u_2)^2\,\hnor_{23}+2(\dive\, u_2)(u_2\cdot\tau_{23}  )\phi_2            \bigg)\,ds\\
&&
-\int_{\Gamma_{23}}\bigg( (\dive\,u_3)^2\,\hnor_{23}+2(\dive\, u_3)(u_3\cdot\tau_{23}  )\phi_3            \bigg)\,ds\\
&&
-\int_{\Gamma_{03}}\bigg( (\dive\,u_3)^2\,\hnor_{03}+2(\dive\, u_3)(u_3\cdot\tau_{03})\phi_3,            \bigg)\,ds.
\end{eqnarray*}
where we have used the fact that $u_1^\perp\cdot\nu_{01}=u_1\cdot\tau_{01}$, $u_1^\perp\cdot\nu_{12}=u_1\cdot\tau_{12}$, etc.

Now we appeal to the relations \eqref{h01}--\eqref{h03}, along with the conditions $u_2\cdot\tau_{12}=-u_1\cdot\tau_{12}$ and
$u_3\cdot\tau_{23}=-u_2\cdot\tau_{23}$ and perform an integration by parts to find
\begin{eqnarray}
&&
 \frac{d}{dt}_{|_t=0}\sum_{j=1}^3\left(\int_{\Omega^t_j}(\dive\,u^t_j)^2\,dx\right)=\nonumber\\
&&
\int_{\Gamma_{01}}\bigg\{ -(\dive\,u_1)^2+2(\dive\,u_1)'(u_1\cdot\tau_{01}) \bigg\}\hnor_{01}\,ds+\nonumber\\
&&\int_{\Gamma_{12}}\bigg\{\bigg((\dive\,u_1)^2- (\dive\,u_2)^2-4\big[(\dive\,u_2)'(u_1\cdot\tau_{12})+
(\dive\,u_2)(u_1\cdot\tau_{12})'\big]\bigg)\hnor_{12}+2(\dive\, u_1-\dive\,u_2)(u_1\cdot\tau_{12}  )\phi_1            \bigg\}\,ds\nonumber\\
&&\int_{\Gamma_{23}}\bigg\{\bigg((\dive\,u_2)^2- (\dive\,u_3)^2-4\big[(\dive\,u_3)'(u_2\cdot\tau_{23})+
(\dive\,u_3)(u_2\cdot\tau_{23})'\big]\bigg)\hnor_{23}+2(\dive\, u_2-\dive\,u_3)(u_2\cdot\tau_{23}  )\phi_2            \bigg\}\,ds\nonumber\\
&&\int_{\Gamma_{03}}\bigg\{ -(\dive\,u_3)^2+2(\dive\,u_3)'(u_3\cdot\tau_{03}) \bigg\}\hnor_{03}\,ds+\nonumber\\
&&\nonumber\\
&&+2\,(\dive\,u_1(p))(u_1(p)\cdot\tau_{01}(0)) \hnor_{01}(0)
-4\,\dive\,u_2(p)(u_1(p)\cdot\tau_{12}(0))\hnor_{12}(0)\nonumber\\
&&\nonumber\\
&&-4\,\dive\,u_3(p)(u_2(p)\cdot\tau_{23}(0))\hnor_{23}(0)+2\,(\dive\,u_3(p))(u_3(p)\cdot\tau_{03}(0)) \hnor_{03}(0).\label{finaldiv}\end{eqnarray}

We turn now to calculating the variations of the four jump energies. 
We begin by invoking \eqref{ijarc} to compute
\begin{eqnarray*}
&&
 \frac{d}{dt}_{|_{t=0}}\bigg(\int_{\Gamma_{01}^t}1\,ds+\int_{\Gamma_{03}^t}1\,ds\bigg)\\
 &&= \frac{d}{dt}_{|_{t=0}}\bigg(\int_0^{s_0}\abs{r_{01}^t\,'(s)}\,ds+\int_0^{s_0}\abs{r_{03}^t\,'(s)}\,ds\bigg)\\
 &&=\frac{d}{dt}_{|_{t=0}}\bigg(\int_0^{s_0}1+t\big(\htan_{01}\,'(s)-\hnor_{01}(s)\kappa_{01}(s)\big)\,ds+\int_0^{s_0}1+t\big(\htan_{03}\,'(s)-\hnor_{03}(s)\kappa_{03}(s)\big)\,ds\bigg)\\
 &&\int_0^{s_0}\big(\htan_{01}\,'(s)-\hnor_{01}(s)\kappa_{01}(s)\big)\,ds+\int_0^{s_0}\big(\htan_{03}\,'(s)-\hnor_{03}(s)\kappa_{03}(s)\big)\,ds
 \end{eqnarray*}
Thus,
\begin{eqnarray}
&&\frac{d}{dt}_{|_{t=0}}\frac{K(0)}{2} \bigg(\mathcal{H}^1(\Gamma_{01}^t)+\mathcal{H}^1(\Gamma_{03}^t)\bigg)=\nonumber\\
&&-\frac{K(0)}{2} \bigg(\int_{\Gamma_{01}}\hnor_{01}\kappa_{01}\,ds+\int_{\Gamma_{03}}\hnor_{03}\kappa_{03}\,ds+\htan_{01}(0)+\htan_{03}(0)\bigg).
\label{firstthird}
\end{eqnarray}
To compute the variation in the jump energies over $\Gamma_{12}^t$ and $\Gamma_{23}^t$ requires an expansion to $O(t)$ of the
quantities $u^t\cdot\nu_{12}^t$ and $u^t\cdot\nu_{23}^t.$ Substituting the expression for $r_{12}^t$  from \eqref{newtan} into the formula for
$u_1^t$ from \eqref{newutexp} and Taylor expanding in $t$ we find with the use of \eqref{newnorm} that along $\Gamma_{12}^t$ we have
\begin{eqnarray}
&&
u^t\cdot\nu_{12}^t\sim \bigg(u_1(r_{12}^t(s))+tu_1^\perp(r_{12}(s))\phi_1(r_{12}(s))\bigg)\cdot\nu_{12}^t\sim\\
&&
\bigg(u_1\big(r_{12}+t\left[\htan_{12}\tau_{12} +\hnor_{12}\nu_{12}\right]\big)+t\phi_1(r_{12})u_1^\perp(r_{12})\bigg)\cdot
\bigg(\nu_{12}-t\big(\hnor_{12}\,'+\kappa_{12}\htan_{12}\big)\tau_{12}\bigg)\nonumber\\
&&
\sim u_1\cdot\nu_{12}+t\bigg[ \left(\phi_1-\hnor_{12}\,'-\kappa_{12}\htan_{12}\right)(u_1\cdot\tau_{12})+\htan_{12}(u_1'\cdot\nu_{12})\bigg],\label{jump12}
\end{eqnarray}
where $u_1$ and $\phi_1$ in the expression above are evaluated at $x=r_{12}(s)$ and $u_1'=\frac{d}{ds}u_1(r_{12}(s)).$ In the last line we have also used that our extension of $u_1$ was constant along $\nu_{12}$ to eliminate the term $\nabla u_1\cdot \nu_{12}$ that would other have been present upon Taylor expanding. 

Similarly, we calculate that along $\Gamma_{23}$ we have
\beq
u^t\cdot\nu_{23}^t\sim
u_2\cdot\nu_{23}+t\bigg[ \left(\phi_2-\hnor_{23}\,'-\kappa_{23}\htan_{23}\right)(u_2\cdot\tau_{23})+\htan_{23}(u_2'\cdot\nu_{23})\bigg].\label{jump23}
\eeq

From \eqref{jump12} and \eqref{jump23}, along with \eqref{ijarc} we can compute that
\begin{eqnarray}
&&\frac{d}{dt}_{|_{t=0}}\bigg(\int_{\Gamma_{12}}K\big(u^t\cdot\nu_{12}^t\big)\,ds+\int_{\Gamma_{23}}K\big(u^t\cdot\nu_{23}^t)\,ds\bigg)\nonumber\\
&&=\frac{d}{dt}_{|_{t=0}}\bigg(\int_0^{s_0}K\big(u^t(r_{12}^t(s))\cdot\nu_{12}^t(s)\big)\abs{r_{12}^t\,'(s)}\,ds+
\int_0^{s_0}K\big(u^t(r_{23}^t(s))\cdot\nu_{23}^t(s)\big)\abs{r_{23}^t\,'(s)}\,ds\bigg)\nonumber\\
&&=
\int_{\Gamma_{12}} K(u_1\cdot\nu_{12})\big(\htan_{12}\,'-\hnor_{12}\kappa_{12}\big)\,ds+\nonumber\\
&&
\int_{\Gamma_{12}} K'(u_1\cdot\nu_{12})\bigg(\big(\phi_1-\hnor_{12}\,'-\htan_{12}\kappa_{12}\big)(u_1\cdot\tau_{12})+
\htan_{12}(u_1'\cdot\nu_{12})\bigg)\,ds+\nonumber\\
&&
\int_{\Gamma_{23}} K(u_2\cdot\nu_{23})\big(\htan_{23}\,'-\hnor_{23}\kappa_{23}\big)\,ds+\nonumber\\
&&
\int_{\Gamma_{23}} K'(u_2\cdot\nu_{23})\bigg(\big(\phi_2-\hnor_{23}\,'-\htan_{23}\kappa_{23}\big)(u_2\cdot\tau_{23})+
\htan_{23}(u_2'\cdot\nu_{23})\bigg)\,ds.\nonumber\label{jump1234}
\end{eqnarray}

Now since 
\[
\frac{d}{ds}\big[K\big(u_1\cdot\nu_{12})\big]=(u_1'\cdot\nu_{12})-\kappa_{12}(u_1\cdot\tau_{12})
\]
and
\[
\frac{d}{ds}\big[K\big(u_2\cdot\nu_{23})\big]=(u_2'\cdot\nu_{23})-\kappa_{23}(u_2\cdot\tau_{23}),
\]
we have that
\[
K(u_1\cdot\nu_{12})\htan_{12}\,'+K'(u_1\cdot\nu_{12})\big(         (u_1'\cdot\nu_{12})-\kappa_{12}(u_1\cdot\tau_{12})          \big)\htan_{12}=\frac{d}{ds}\big[K\big(u_1\cdot\nu_{12})\htan_{12}\big]
\]
and
\[
K(u_2\cdot\nu_{23})\htan_{23}\,'+K'(u_2\cdot\nu_{23})\big(         (u_2'\cdot\nu_{23})-\kappa_{23}(u_2\cdot\tau_{23})          \big)\htan_{23}=\frac{d}{ds}\big[K\big(u_2\cdot\nu_{23})\htan_{23}\big].
\]
Using these last two identities in \eqref{jump1234} and integrating by parts implies that
\begin{eqnarray}
&&\frac{d}{dt}_{|_{t=0}}\bigg(\int_{\Gamma_{12}}K\big(u^t\cdot\nu_{12}^t\big)\,ds+\int_{\Gamma_{23}}K\big(u^t\cdot\nu_{23}^t)\,ds\bigg)\nonumber\\
&&=
-\int_{\Gamma_{12}} K(u_1\cdot\nu_{12})\hnor_{12}\kappa_{12}\,ds+
\int_{\Gamma_{12}} K'(u_1\cdot\nu_{12})(\phi_1-\hnor_{12}\,')(u_1\cdot\tau_{12})\,ds\nonumber\\
&&-\int_{\Gamma_{23}} K(u_2\cdot\nu_{23})\hnor_{23}\kappa_{23}\,ds+
\int_{\Gamma_{23}} K'(u_2\cdot\nu_{23})(\phi_2-\hnor_{23}\,')(u_2\cdot\tau_{23})\,ds\nonumber\\
&&-K\big(u_1(p)\cdot\nu_{12}(0))\htan_{12}(0)-K\big(u_2(p)\cdot\nu_{23}(0))\htan_{23}(0).
\label{jump1223}
\end{eqnarray}
Then invoking the criticality condition \eqref{natbc} from Theorem \ref{critthm1} and integrating by parts we can rewrite this identity as
\begin{eqnarray}
&&\frac{d}{dt}_{|_{t=0}}\bigg(\int_{\Gamma_{12}}K\big(u^t\cdot\nu_{12}^t\big)\,ds+\int_{\Gamma_{23}}K\big(u^t\cdot\nu_{23}^t)\,ds\bigg)\nonumber\\
&&=
-\int_{\Gamma_{12}} K(u_1\cdot\nu_{12})\hnor_{12}\kappa_{12}\,ds+
L\int_{\Gamma_{12}} \big(  \dive\,u_2-\dive\,u_1   \big)(\phi_1-\hnor_{12}\,')(u_1\cdot\tau_{12})\,ds\nonumber\\
&&-\int_{\Gamma_{23}} K(u_2\cdot\nu_{23})\hnor_{23}\kappa_{23}\,ds+
L\int_{\Gamma_{23}} \big( \dive\,u_3-\dive\,u_2 \big)(\phi_2-\hnor_{23}\,')(u_2\cdot\tau_{23})\,ds\nonumber\\
&&-K\big(u_1(p)\cdot\nu_{12}(0))\htan_{12}(0)-K\big(u_2(p)\cdot\nu_{23}(0))\htan_{23}(0)\nonumber\\
&&=
\int_{\Gamma_{12}}\bigg\{ L \big(  \dive\,u_2-\dive\,u_1   \big)'(u_1\cdot\tau_{12})+L\big(  \dive\,u_2-\dive\,u_1   \big)(u_1\cdot\tau_{12})'-K(u_1\cdot\nu_{12})\kappa_{12}\bigg\}\hnor_{12}\,ds\nonumber\\
&&+
L\int_{\Gamma_{12}} \big(  \dive\,u_2-\dive\,u_1   \big)(u_1\cdot\tau_{12})\phi_1\,ds\nonumber\\
&&\int_{\Gamma_{23}}\bigg\{L  \big(  \dive\,u_3-\dive\,u_2   \big)'(u_2\cdot\tau_{23})+L\big(  \dive\,u_3-\dive\,u_2   \big)(u_2\cdot\tau_{23})'-K(u_2\cdot\nu_{23})\kappa_{23}\bigg\}\hnor_{23}\,ds\nonumber\\
&&+
L\int_{\Gamma_{23}} \big(  \dive\,u_3-\dive\,u_2   \big)(u_2\cdot\tau_{23})\phi_2\,ds\nonumber\\
&&-K\big(u_1(p)\cdot\nu_{12}(0))\htan_{12}(0)+L\big(  \dive\,u_2(p)-\dive\,u_1(p)   \big)(u_1(p)\cdot\tau_{12}(0))\hnor_{12}(0)\nonumber\\
&&-K\big(u_2(p)\cdot\nu_{23}(0))\htan_{23}(0)+L\big(  \dive\,u_3(p)-\dive\,u_2(p)   \big)(u_2(p)\cdot\tau_{23}(0))\hnor_{23}(0).\label{newjump1223}
\end{eqnarray}

Now we can combine \eqref{finaldiv}, \eqref{firstthird} and \eqref{newjump1223}, and through a use of the criticality conditions \eqref{wally} and 
\eqref{intercrit} of Theorem \ref{critthm2} we find that all integrals over the four curves $\Gamma_{ij}$ drop, leaving only 
\begin{eqnarray}
&&\frac{d}{dt}_{|_t=0} E_0(u^t)=\nonumber\\
&&-\frac{K(0)}{2}\big(\htan_{01}(0)+\htan_{03}(0)\big)-K\big(u_1(p)\cdot\nu_{12}(0)\big)\htan_{12}-K\big(u_2(p)\cdot\nu_{23}(0)\big)\htan_{23}\nonumber\\
&&
+L\bigg\{ \dive\,u_1(p)(u_1(p)\cdot\tau_{01}(0)) \hnor_{01}(0)  +    \dive\,u_3(p)(u_3(p)\cdot\tau_{03}(0)) \hnor_{03}(0)    \bigg\} \nonumber\\
&&
-L\bigg\{\big(\dive\,u_1(p)+\dive\,u_2(p)\big)\big(u_1(p)\cdot\tau_{12}(0)\big)\hnor_{12}+  
             \big(\dive\,u_2(p)+\dive\,u_3(p)\big)\big(u_2(p)\cdot\tau_{23}(0)\big)\hnor_{23}  \bigg\}\nonumber
\end{eqnarray}
Recall now that $\htan_{01}(0)=X(p)\cdot \tau_{01}(0)$, $\hnor_{01}(0)=X(p)\cdot \nu_{01}(0)$, etc. Thus, the arbitrary value of the vector $X(p)$, implies that a vanishing first variation $\frac{d}{dt}_{|_t=0} E_0(u^t)=0$ leads to the necessary condition at a junction $P$
of the form
\begin{eqnarray}
&&\frac{K(0)}{2}\big(\tau_{01}+\tau_{03}\big)+K\big(u_1\cdot\nu_{12}\big)\tau_{12}+K\big(u_2\cdot\nu_{23}\big)\tau_{23}\nonumber\\
&&
=L\bigg\{ \dive\,u_1(u_1\cdot\tau_{01}) \nu_{01}  + \dive\,u_3(u_3\cdot\tau_{03}) \nu_{03}    \bigg\} \nonumber\\
&&
-L\bigg\{\big(\dive\,u_1+\dive\,u_2\big)\big(u_1\cdot\tau_{12}\big)\nu_{12}+  
             \big(\dive\,u_2+\dive\,u_3\big)\big(u_2\cdot\tau_{23}\big)\nu_{23}  \bigg\}\nonumber\\
             &&\label{junk}
\end{eqnarray}
where all quantities above are evaluated at the junction $P$.

\end{proof}
\newpage
\noindent{\bf Acknowledgments.} {\it PS, MR and RV acknowledge the support from NSF DMS-1362879 and a Simons Collaboration grant 585520. RV also acknowledges the support from an Indiana University College of Arts and Sciences Dissertation Year Fellowship. The research of RV was also partially supported by the Center for Nonlinear Analysis at Carnegie Mellon University and by NSF DMS-1411646. DG acknowledges the support from NSF DMS-1729538.}

\bibliographystyle{acm} 

\bibliography{GNSV} 

\begin{thebibliography}{10}

\bibitem{comsol}
{COMSOL} {Multiphysics\textregistered} {v. 5.3}.
\newblock http://www.comsol.com/.
\newblock {COMSOL} {AB}, {Stockholm}, {Sweden}.

\bibitem{AlougesRiviereSerfaty}
{\sc Alouges, F., Rivi\`ere, T., and Serfaty, S.}
\newblock N\'eel and cross-tie wall energies for planar micromagnetic
  configurations.
\newblock {\em ESAIM Control Optim. Calc. Var. 8\/} (2002), 31--68.
\newblock A tribute to J. L. Lions.

\bibitem{ADM}
{\sc Ambrosio, L., De~Lellis, C., and Mantegazza, C.}
\newblock Line energies for gradient vector fields in the plane.
\newblock {\em Calc. Var. Partial Differential Equations 9}, 4 (1999),
  327--255.

\bibitem{GuAn}
{\sc Angenent, S., and Gurtin, M.~E.}
\newblock Multiphase thermomechanics with interfacial structure 2. {E}volution
  of an isothermal interface.
\newblock {\em Arch. Ration. Mech. Anal. 108}, 3 (Nov 1989), 323--391.

\bibitem{AG}
{\sc Aviles, P., and Giga, Y.}
\newblock On lower semicontinuity of a defect energy obtained by a singular
  limit of the {G}inzburg-{L}andau type energy for gradient fields.
\newblock {\em Proc. Roy. Soc. Edinburgh Sect. A 129}, 1 (1999), 1--17.

\bibitem{baldo}
{\sc Baldo, S.}
\newblock Minimal interface criterion for phase transitions in mixtures of
  {C}ahn-{H}illiard fluids.
\newblock {\em Ann. Inst. H. Poincar\'{e} Anal. Non Lin\'{e}aire 7}, 2 (1990),
  67--90.

\bibitem{BPP}
{\sc Bauman, P., Park, J., and Phillips, D.}
\newblock Analysis of nematic liquid crystals with disclination lines.
\newblock {\em Arch. Ration. Mech. Anal. 205}, 3 (2012), 795--826.

\bibitem{BBH}
{\sc Bethuel, F., Brezis, H., and H\'elein, F.}
\newblock {\em Ginzburg-{L}andau vortices}, vol.~13 of {\em Progress in
  Nonlinear Differential Equations and their Applications}.
\newblock Birkh\"auser Boston, Inc., Boston, MA, 1994.

\bibitem{CD}
{\sc Conti, S., and De~Lellis, C.}
\newblock Sharp upper bounds for a variational problem with singular
  perturbation.
\newblock {\em Math. Ann. 338}, 1 (2007), 119--146.

\bibitem{DO}
{\sc De~Lellis, C., and Otto, F.}
\newblock Structure of entropy solutions to the eikonal equation.
\newblock {\em J. Eur. Math. Soc. (JEMS) 5}, 2 (2003), 107--145.

\bibitem{DA}
{\sc DeBenedictis, A., and Atherton, T.~J.}
\newblock Shape minimisation problems in liquid crystals.
\newblock {\em Liquid Crystals 43}, 13-15 (2016), 2352--2362.

\bibitem{DKMO}
{\sc DeSimone, A., Kohn, R.~V., M\"uller, S., and Otto, F.}
\newblock A compactness result in the gradient theory of phase transitions.
\newblock {\em Proc. Roy. Soc. Edinburgh Sect. A 131}, 4 (2001), 833--844.

\bibitem{FTKLR}
{\sc Fang, J., Teer, E., Knobler, C.~M., Loh, K.-K., and Rudnick, J.}
\newblock Boojums and the shapes of domains in monolayer films.
\newblock {\em Phys. Rev. E 56\/} (Aug 1997), 1859--1868.

\bibitem{F}
{\sc Fonseca, I.}
\newblock The wulff theorem revisited.
\newblock {\em Proceedings: Mathematical and Physical Sciences 432}, 1884
  (1991), 125--145.

\bibitem{FM}
{\sc Fonseca, I., and M\"{u}ller, S.}
\newblock Quasi-convex integrands and lower semicontinuity in {$L^1$}.
\newblock {\em SIAM J. Math. Anal. 23}, 5 (1992), 1081--1098.

\bibitem{giusti}
{\sc Giusti, E.}
\newblock {\em Minimal surfaces and functions of bounded variation}, vol.~80 of
  {\em Monographs in Mathematics}.
\newblock Birkh\"{a}user Verlag, Basel, 1984.

\bibitem{GSV}
{\sc {Golovaty}, D., {Sternberg}, P., and {Venkatraman}, R.}
\newblock {A Ginzburg-Landau type problem for highly anisotropic nematic liquid
  crystals}.
\newblock {\em To appear in SIAM J. Math. Anal.\/} (2018).

\bibitem{Ignat}
{\sc Ignat, R.}
\newblock Singularities of divergence-free vector fields with values into
  {$\mathbb S^1$} or {$\mathbb S^2$}. {A}pplications to micromagnetics.
\newblock {\em Confluentes Math. 4}, 3 (2012), 1230001, 80.

\bibitem{jerrard}
{\sc Jerrard, R.~L.}
\newblock Lower bounds for generalized {G}inzburg-{L}andau functionals.
\newblock {\em SIAM J. Math. Anal. 30}, 4 (1999), 721--746.

\bibitem{JinKohn}
{\sc Jin, W., and Kohn, R.~V.}
\newblock Singular perturbation and the energy of folds.
\newblock {\em J. Nonlinear Sci. 10}, 3 (2000), 355--390.

\bibitem{KSL}
{\sc Kim, Y.-K., Shiyanovskii, S.~V., and Lavrentovich, O.~D.}
\newblock Morphogenesis of defects and tactoids during isotropic–nematic
  phase transition in self-assembled lyotropic chromonic liquid crystals.
\newblock {\em Journal of Physics: Condensed Matter 25}, 40 (2013), 404202.

\bibitem{kohnsternberg}
{\sc Kohn, R.~V., and Sternberg, P.}
\newblock Local minimisers and singular perturbations.
\newblock {\em Proc. Roy. Soc. Edinburgh Sect. A 111}, 1-2 (1989), 69--84.

\bibitem{Spirn-Kurzke}
{\sc Kurzke, M., and Spirn, D.}
\newblock Gamma limit of the nonself-dual {C}hern-{S}imons-{H}iggs energy.
\newblock {\em J. Funct. Anal. 255}, 3 (2008), 535--588.

\bibitem{LO}
{\sc Lamy, X., and Otto, F.}
\newblock On the regularity of weak solutions to {B}urgers' equation with
  finite entropy production.
\newblock {\em Calc. Var. PDE 57}, 4 (2018), Art. 94, 19.

\bibitem{Lorent}
{\sc Lorent, A.}
\newblock A simple proof of the characterization of functions of low {A}viles
  {G}iga energy on a ball {\it via} regularity.
\newblock {\em ESAIM Control Optim. Calc. Var. 18}, 2 (2012), 383--400.

\bibitem{Lorent2}
{\sc Lorent, A.}
\newblock A quantitative characterisation of functions of low {A}viles {G}iga
  energy in convex domains.
\newblock {\em Ann. Sc. Norm. Super. Pisa Cl. Sci. (5) 13}, 1 (2014), 1--66.

\bibitem{ModicaARMA}
{\sc Modica, L.}
\newblock The gradient theory of phase transitions and the minimal interface
  criterion.
\newblock {\em Arch. Ration. Mech. Anal. 98}, 2 (1987), 123--142.

\bibitem{MN}
{\sc Mottram, N.~J., and Newton, C.~J.}
\newblock Introduction to {Q}-tensor theory.
\newblock {\em arXiv preprint arXiv:1409.3542\/} (2014).

\bibitem{murat2}
{\sc Murat, F.}
\newblock Compacit\'{e} par compensation.
\newblock {\em Ann. Scuola Norm. Sup. Pisa Cl. Sci. (4) 5}, 3 (1978), 489--507.

\bibitem{Murat}
{\sc Murat, F.}
\newblock L'injection du c\^{o}ne positif de {$H^{-1}$} dans {$W^{-1,\,q}$} est
  compacte pour tout {$q<2$}.
\newblock {\em J. Math. Pures Appl. (9) 60}, 3 (1981), 309--322.

\bibitem{novack1}
{\sc {Novack}, M.~R.}
\newblock {Dimension reduction for the Landau-de Gennes model: the vanishing
  nematic correlation length limit}.
\newblock {\em To appear in SIAM J. Math. Anal.\/} (2018).

\bibitem{OwenRubinsteinSternberg}
{\sc Owen, N.~C., Rubinstein, J., and Sternberg, P.}
\newblock Minimizers and gradient flows for singularly perturbed bi-stable
  potentials with a {D}irichlet condition.
\newblock {\em Proc. Roy. Soc. London Ser. A 429}, 1877 (1990), 505--532.

\bibitem{polupper1}
{\sc Poliakovsky, A.}
\newblock On the {$\Gamma$}-limit of singular perturbation problems with
  optimal profiles which are not one-dimensional. {P}art {I}: {T}he upper
  bound.
\newblock {\em Differential Integral Equations 26}, 9-10 (2013), 1179--1234.

\bibitem{pollower}
{\sc Poliakovsky, A.}
\newblock On the {$\Gamma$}-limit of singular perturbation problems with
  optimal profiles which are not one-dimensional. {P}art {II}: {T}he lower
  bound.
\newblock {\em Israel J. Math. 210}, 1 (2015), 359--398.

\bibitem{RB}
{\sc Rudnick, J., and Bruinsma, R.}
\newblock Shape of domains in two-dimensional systems: Virtual singularities
  and a generalized wulff construction.
\newblock {\em Phys. Rev. Lett. 74\/} (Mar 1995), 2491--2494.

\bibitem{sandier}
{\sc Sandier, E.}
\newblock Lower bounds for the energy of unit vector fields and applications.
\newblock {\em J. Funct. Anal. 152}, 2 (1998), 379--403.

\bibitem{Sternberg86}
{\sc Sternberg, P.}
\newblock The effect of a singular perturbation on nonconvex variational
  problems.
\newblock {\em Arch. Ration. Mech. Anal. 101}, 3 (1988), 209--260.

\bibitem{tartar}
{\sc Tartar, L.}
\newblock Compensated compactness and applications to partial differential
  equations.
\newblock In {\em Nonlinear analysis and mechanics: {H}eriot-{W}att
  {S}ymposium, {V}ol. {IV}}, vol.~39 of {\em Res. Notes in Math.} Pitman,
  Boston, Mass.-London, 1979, pp.~136--212.

\bibitem{vBOS}
{\sc van Bijnen, R. M.~W., Otten, R. H.~J., and van~der Schoot, P.}
\newblock Texture and shape of two-dimensional domains of nematic liquid
  crystals.
\newblock {\em Phys. Rev. E 86\/} (Nov 2012), 051703.

\end{thebibliography}

\end{document}